\documentclass[14pt]{article}
\usepackage[a4paper, margin = 3cm]{geometry}
\usepackage{amsmath, amssymb, amsthm, bbm, tikz-cd, url}
\usepackage[hidelinks]{hyperref}
\usepackage[affil-it]{authblk}

\DeclareSymbolFont{bbold}{U}{bbold}{m}{n}
\DeclareSymbolFontAlphabet{\mathbbold}{bbold}

\usepackage{amsmath, amsthm, stmaryrd, amsfonts, cleveref, tikz-cd, enumitem}

\newcommand{\ind}{\leavevmode{\parindent=15pt\indent}}

\theoremstyle{plain} %Theorem, Lemma, Corollary, Proposition, Conjecture, Criterion, Assertion
\newtheorem{theorem}{Theorem}[section]
\newtheorem*{theorem*}{Theorem}
\newtheorem{corollary}[theorem]{Corollary}
\newtheorem*{corollary*}{Corollary}
\newtheorem{proposition}[theorem]{Proposition}
\newtheorem*{proposition*}{Proposition}
\newtheorem{lemma}[theorem]{Lemma}
\newtheorem*{lemma*}{Lemma}

\newtheorem*{fact*}{Fact}

\newtheorem*{conjecture*}{Conjecture}

\newtheorem*{criterion*}{Criterion}

\newtheorem*{assertion*}{Assertion}

\newtheorem*{lem_def*}{Lemma-Definition}

\newtheorem*{prop_def*}{Proposition_Definition}

\newtheorem*{thm_def*}{Theorem-Definition}

\theoremstyle{definition} %Definition, Condition, Problem, Example, Exercise, Algorithm, Question, Axiom, Property, Assumption, Hypothesis
\newtheorem{example}[theorem]{Example}
\newtheorem*{example*}{Example}

\newtheorem*{examples*}{Examples}
\newtheorem{definition}[theorem]{Definition}
\newtheorem*{definition*}{Definition}

\newtheorem*{condition*}{Condition}

\newtheorem*{problem*}{Problem}

\newtheorem*{exercise*}{Exercise}

\newtheorem*{algorithm*}{Algorithm}

\newtheorem*{subroutine*}{Subroutine}

\newtheorem*{question*}{Question}

\newtheorem*{axiom*}{Axiom}

\newtheorem*{property*}{Property}

\newtheorem*{assumption*}{Assumption}

\newtheorem*{hypothesis*}{Hypothesis}

\theoremstyle{remark} %Remark, Note, Notation, Claim, Summary, Acknowledgment, Case, Conclusion
\newtheorem{remark}[theorem]{Remark}
\newtheorem{remarks}[theorem]{Remarks}
\newtheorem*{remark*}{Remark}

\newtheorem*{note*}{Note}

\newtheorem*{scholium*}{Scholium}

\newtheorem*{notation*}{Notation}

\newtheorem*{claim*}{Claim}

\newtheorem*{summary*}{Summary}

\newtheorem*{acknowledgment*}{acknowledgment}

\newtheorem*{acknowledgement*}{acknowledgement}

\newtheorem*{case*}{Case}

\newtheorem*{conclusion*}{Conclusion}
\makeatletter
\def\thmheadbrackets#1#2#3{%
  \thmname{#1}\thmnumber{\@ifnotempty{#1}{ }\@upn{#2}}%
  \thmnote{ {\the\thm@notefont[#3]}}}
\makeatother

\newtheoremstyle{brackets}% Name
  {}% space above
  {}% space below
  {\itshape}% body font
  {}% indent
  {\bfseries}% head font
  {.}% punctuation after head
  { }% space after head (has to be space or dimension!)
  {\thmheadbrackets{#1}{#2}{#3}}% head spec

\theoremstyle{brackets}

\newtheorem*{theorembrackets*}{Theorem}

\makeatletter
\def\thmheadnoparens#1#2#3{%
  \thmname{#1}\thmnumber{\@ifnotempty{#1}{ }\@upn{#2}}%
  \thmnote{ {\the\thm@notefont#3}}}
\makeatother

\newtheoremstyle{noparens}% Name
  {}% space above
  {}% space below
  {\itshape}% body font
  {}% indent
  {\bfseries}% head font
  {.}% punctuation after head
  { }% space after head (has to be space or dimension!)
  {\thmheadnoparens{#1}{#2}{#3}}% head spec

\theoremstyle{noparens}

\newtheorem*{theoremnoparens*}{Theorem}

\usepackage{xparse}
\DeclareDocumentCommand{\newmathcommand}{mO{0}m}{%
  \expandafter\let\csname old\string#1\endcsname=#1
  \expandafter\newcommand\csname new\string#1\endcsname[#2]{#3}
  \DeclareRobustCommand#1{%
    \ifmmode
      \expandafter\let\expandafter\next\csname new\string#1\endcsname
    \else
      \expandafter\let\expandafter\next\csname old\string#1\endcsname
    \fi
    \next
  }%
}

\newcommand{\C}{\mathbb{C}}

\newmathcommand{\H}{{\mathbb{H}}}

\newmathcommand{\L}{\mathbb{L}}

\newmathcommand{\O}{\mathbb{O}}
\newmathcommand{\P}{\mathbb{P}}

\newcommand{\R}{\mathbb{R}}
\newmathcommand{\S}{\mathbb{S}}

\newcommand{\Z}{\mathbb{Z}}
\newcommand{\Af}{\mathfrak{A}}
\newcommand{\Bf}{\mathfrak{B}}

\newcommand{\iz}{\mathrm{i}}

\DeclareMathOperator{\spn}{span}

\newmathcommand{\deg}{\operatorname{deg}}

\DeclareMathOperator{\End}{End}

\DeclareMathOperator{\Aut}{Aut}

\DeclareMathOperator{\ext}{ext}

\newmathcommand{\Big}{{\mathrm{Big}}}

\DeclareMathOperator{\ev}{ev}

\newmathcommand{\div}{\operatorname{div}}

\newmathcommand{\char}{\operatorname{char}}

\DeclareMathOperator{\Id}{Id}
\DeclareMathOperator{\id}{id}

\newmathcommand{\unr}{\mathrm{unr}}

\newcommand{\1}{\mathbbm{1}}

\newmathcommand{\Un}{\operatorname{Un}}

\newmathcommand{\sl}{\mathfrak{s}\mathfrak{l}}

\newmathcommand{\un}{\mathfrak{u}}
\newmathcommand{\su}{\mathfrak{s}\mathfrak{u}}

\makeatletter
\def\@maketitle{%
  \newpage
  \null
  \vskip 2em%
  \begin{center}%
  \let \footnote \thanks
    {\Large\bfseries \@title \par}%
    \vskip 1.5em%
    {\normalsize
      \lineskip .5em%
      \begin{tabular}[t]{c}%
        \@author
      \end{tabular}\par}%
    \vskip 1em%
    {\normalsize \@date}%
  \end{center}%
  \par
  \vskip 1.5em}
  
\makeatother

\title{\sc \huge An order-theoretic characterization of JB-algebras}

\author{Mark Roelands%
\thanks{Email: \texttt{m.roelands@math.leidenuniv.nl}}}
\affil{Mathematical Institute, Leiden University, 2300 RA Leiden,
The Netherlands}

\author{Samuel Tiersma%
\thanks{Email: \texttt{s.j.tiersma@math.leidenuniv.nl}}}
\affil{Mathematical Institute, Leiden University, 2300 RA Leiden,
The Netherlands}

\begin{document}

\maketitle
\vspace{-5mm}
\date{}
\vspace{-5mm}

\begin{abstract}
\noindent We give an order-theoretic characterization of the JB-algebras among the complete order unit spaces in terms of the existence of an order-anti-automorphism of the interior of the cone that is homogeneous of degree $-1$. More geometrically, we characterize JB-algebras as those complete order unit spaces for which the interior of the cone is a symmetric Banach--Finsler manifold under Thompson's metric. Furthermore, we show that two order unit spaces are isomorphic if there exists a gauge-reversing bijection between them, thus answering a question raised by Noll--Sch\"afer. These results have previously been established for finite-dimensional resp. reflexive order unit spaces by Walsh resp. Lemmens, Wortel, and the first author.
\end{abstract}



\section{Introduction}\label{Intro}
Jordan algebras form a fecund class of nonassociative algebras, originating in the work of P. Jordan, Wigner and von Neumann on the mathematical formalization of quantum mechanics (cf. \cite{Jordan33}, \cite{JNW}). The theory of Jordan algebras has proved useful in a variety of fields, such as differential geometry \cite{Koecher75}, \cite{Vinberg60}, \cite{Ko99}, Lie theory  \cite{Jac71}, \cite{FaFe77},  \cite{SpVe}, group theory \cite{Zel291} and the study of symmetric spaces  \cite{Loos69}, \cite{Kaup02}, \cite{Up85}, \cite{Sa80}; for an overview see \cite{McC04}. \\
%OPERATOR ALGEBRAS! State JB-algebras here.
\ind Characterizations of Jordan algebras have been developed from three distinct perspectives: functional-analytic, differential-geometric and algebraic. The present paper is concerned with an order-theoretic characterization of JB-algebras, i.e.\ those real Jordan algebras admitting a compatible Banach space structure. More precisely, we prove a necessary and sufficient criterion for an order unit space to be a JB-algebra. In this introduction we will discuss this criterion from each of the three aforementioned viewpoints.

\paragraph{Functional-analytic perspective.}
The algebraic structure in a C$^*$-algebra is intimately connected to its order structure. JB-algebras constitute a versatile framework to study this interconnection. Recall that the hermitian part $A_h := \{x \in A\colon x^* = x\}$ of a C$^*$-algebra $A$ is partially ordered by its cone of squares $A_+ := \{x^2: x \in A_h\}$. It is natural to consider on $A_h$ the operation of squaring and its polarization, the \emph{Jordan product} $x\circ y := \frac{1}{2}(xy+yx).$ The product $\circ$ is commutative but not associative unless $A$ is abelian. However, a weak version of associativity is satisfied, namely the \emph{Jordan identity} $$x \circ (y\circ x^2) = (x \circ y) \circ x^2.$$ The algebra $(A_h, \circ)$ is the archetypical example of a JB-algebra, named after Jordan and Banach.\\
\ind A \emph{JB-algebra} is a commutative real algebra $(V, \circ)$ which satisfies the Jordan identity and is complete for a norm satisfying $$\lVert x^2\rVert = \lVert x\rVert ^2,\ \ \ \lVert x\circ y\rVert \le \lVert x\rVert\lVert y\rVert,\ \ \mbox{and }\ \ \lVert x\rVert^2 \le \lVert x^2+y^2 \rVert \quad \text{ for all }x,y \in V.$$ JB-algebras were introduced by Alfsen, Shultz and St{\o}rmer in \cite{ASS78}, were they established a Gelfand–Naimark type theorem stating that each JB-algebra is isometrically isomorphic to a closed Jordan subalgebra of $C(X, \Af) \oplus B(H)_{h}$ for some compact Hausdorff space $X$ and Hilbert space $H$, where $\Af$ is the exceptional Albert algebra of hermitian $3\times 3$-octonionic matrices.\\
\ind The JB-algebra $V$ is partially ordered by its cone of squares $V_+ = \{x \circ x: x\in V\}$. If $V$ has a unit element $v$, then $(V, V_+, v)$ is an order unit space and the norm on $V$ coincides with the order unit norm.
It is known that a unital order isomorphism between two JB-algebras is in fact a Jordan isomorphism (see \cite[Thm. 4]{WrYo}, or \Cref{YWsym} for a novel proof), which is the analogue of Kadison's result for C*-algebras. This theorem is but one example of the strong interdependence between the order and the algebraic structure in a JB-algebra. \\
\ind Naturally, the question has been raised if one can `recognize' if an order unit space admits a compatible structure of a JB-algebra, preferably in terms of physically meaningful axioms. The possibility of doing so was realized by Connes’ characterization of $\sigma$-finite von Neumann algebras in terms of self-dual, facially homogeneous and orientable cones in complex Hilbert spaces \cite{Connes74}. This result was subsequently generalized to $\sigma$-finite JBW-algebras in \cite{BelIoch78}.\\
\ind The state spaces of unital JB-algebras were characterized by Alfsen and Shultz in the seminal work \cite{ASchar78}. They showed that a compact convex set $K$ is affinely homeomorphic to the state space of a JB-algebra if and only if it has the \emph{pure state properties}, each norm-exposed face of $K$ is \emph{projective} and every element of $A(K)$ is the difference of \emph{orthogonal} positive elements (we refer to \cite{AS03} for a comprehensive account of the theory and the definition of the italicized terms). This result was applied to characterize state spaces of C*-algebras \cite{AHOS} and von Neumann algebras \cite{ISvnA}. A characterization of the state spaces of finite-dimensional JB-algebras in terms appropriate for quantum mechanics was obtained in \cite{AraChar}. This paper establishes a characterization of unital JB-algebras directly in terms of their underlying order unit spaces, rather than through the dual state spaces. 

\ind Let $(V, V_+, v)$ be an order unit space. A bijection $\Phi \colon V^{\circ}_+ \to V^{\circ}_+$ is called \emph{gauge-reversing} if it is an order-anti-isomorphism and is homogeneous of degree $-1$. In case $V$ is finite-dimensional, Walsh proved \cite{Walsh18} that $V$ is a JB-algebra if and only if $V^{\circ}_+$ admits a gauge-reversing bijection. This result has been extended to characterize spin factors \cite{LRvI17} and JH-algebras \cite{LRW25} among the strictly convex resp. reflexive, order unit spaces. This criterion holds for JB-algebras in general.

\begin{theorem}\label{MAINTHM}
Let $(V, V_+, v)$ be a complete order unit space. Then there exists a structure of Jordan algebra on $V$ such that $(V, \lVert\cdot\rVert_v)$ is a JB-algebra with unit element $v$ and cone of squares $V_+$ if and only if there exists a gauge-reversing bijection on $V^\circ_+$.
\end{theorem}

This theorem will be proved as \Cref{MAINTHMC}. A characterization of non-unital JB-algebras among the ordered Banach spaces shall be derived from it in \Cref{MAINTHMNU}. Moreover, \Cref{MAINTHM} can be used to characterize C*-algebras among the unital complex operator systems $(S, \{C_n\}_n, e)$ as those for which the open cone $C_n^{\circ} \subset M_n(S)_h$ admits a gauge-reversing bijection for all $n \in \Z_{\ge 1}$, as will be  discussed in a future publication.

\paragraph{Differential-geometric perspective.} Jordan structures are inextricably linked with the theory of symmetric spaces. The hermitian symmetric spaces were classified by É. Cartan  \cite{Car35}. Via the Harish-Chandra realization, every Hermitian symmetric space of noncompact type is biholomorphic to a bounded symmetric domain in $\C^n$. An algebraic description of bounded symmetric domains has been given by Koecher \cite{Koe} and Loos \cite{Loos1977lecturenotes} in terms of hermitian Jordan triple systems. Kaup extended this result to infinite dimensions in \cite{Kaup83} by proving the equivalence of the category of bounded symmetric domains with base-point to the category of JB*-triples.\\
\ind The bounded symmetric domains of tube type (or symmetric Siegel domains of the first kind) correspond bijectively to JB*-algebras, which are precisely the complexifications of JB-algebras \cite[Thm. 2.8]{WrJBstarornot}. Indeed,  it has been shown in \cite{KaHa1} and \cite{KaHa2} that an open cone $C$ in a Banach space $V$ stems from a JB-algebra if and only if the tube domain $T := V \oplus \iz C \subset V^{\C}$ is symmetric. In finite dimensions this result was shown by Koecher \cite{Koe}, who was in the process led to study Euclidean Jordan algebras (EJAs).\\
A EJA is a finite-dimensional real Jordan algebra which admits an associative positive-definite inner product. EJAs are precisely the finite-dimensional JB-algebras \cite[Ch. III, Sec. 1; Ch. VIII, Sec. 4]{FaKor}, equivalently, the finite-dimensional formally real Jordan algebras \cite[Cor. 3.1.7 and 3.3.8]{HOSt84}. An open cone $C$ in a Hilbert space $(V, (\cdot\mid\cdot))$ is called \emph{symmetric} if it is \emph{homogeneous}, i.e.\ $\Aut(C)$ acts transitively on it, and \emph{self-dual}, which is to say that $$C = \{x\in C: (x\mid y)>0 \text{ for all }y \in \overline{C}\setminus \{0\}\}.$$
It is not too difficult to prove that the open cone of an EJA is a symmetric cone \cite{FaKor}[Thm. III.2.1]. The celebrated Koecher--Vinberg theorem (\cite{Koecher75} and \cite{Vinberg60}) asserts the converse holds: each symmetric cone $C$ in Euclidean space is the interior of the cone of squares in a Euclidean Jordan algebra. The Vinberg $*$-map on the symmetric cone $C$ yields the inversion in the constructed EJA, and makes $C$ into a Riemannian symmetric space. Kai showed in \cite{Kai} that a convex homogeneous cone is symmetric if and only if the associated Vinberg $*$-map is order-reversing. This was the motivation for Walsh's characterizing EJAs via order-reversing bijections.\\
\ind Chu extended the Koecher--Vinberg theorem in \cite{Chu17} to infinite dimension, by proving that symmetric cones in Hilbert spaces correspond to unital JH-algebras, which are precisely the reflexive JB-algebras (see \cite[Thm.~2.8]{LRW25} and \cite[Thm.~2.9]{LRW25}). The open cone of a general JB-algebra is not a Riemannian symmetric space, but it is a symmetric Banach--Finsler manifold.\\
\ind In fact, define the $\Aut(C)$-invariant
\emph{Thompson's metric} $d_T\colon C \times C \to [0, \infty)$ on the open cone $C$ of an order unit space by
$$d_T(x, y) := \log \inf \{\lambda \ge 1: \lambda^{-1}y \le x \le \lambda y\}.$$ Thompson's metric is a Finsler metric (in the sense of \cite[Def. 1.3]{Neeb02}) since it is the integrated distance for the family of gauge norms $(\lVert\cdot\rVert_p)_{p\in C}$, where
$$\lVert x\rVert_p = \inf\{\lambda \ge 0: -\lambda p \le x \le \lambda p\}.$$
\ind A connected Banach manifold with Finsler structure $(X, \nu)$ is called a \emph{symmetric Banach--Finsler manifold} (cf. \cite[Def 17.20]{Up85}) if for every $p\in X$ there exists an involutive $\nu$-isometry $s_p\colon X \to X$ having $p$ as an isolated fixed point. If $C$ is the open cone of a unital JB-algebra, then $(C, d_T)$ is a symmetric Banach--Finsler manifold; for example the $d_T$-symmetry at the identity is inversion. In the converse direction, Chu has proved in \cite{Chu21} that the open cone $C$ of a complete order unit space corresponds to a JB-algebra if and only if there exist an $\Aut(C)$-invariant Finsler metric $\nu$ and analytic $\nu$-symmetries turning $(C, \nu)$ into a symmetric Banach--Finsler manifold. We shall prove this result in case $\nu = d_T$ without making the strong assumption of analyticity of the symmetries.

\begin{theorem}\label{MAINTHM2}
Let $V$ be a complete order unit space and let $C = V_+^{\circ}$ be its open cone. Then the following are equivalent:
\begin{enumerate}[label={\normalfont(\arabic*)}]
    \item there exists a $d_T$-symmetry $S_p\colon C \to C$ at some point $p \in C$;
    \item $(C, d_T)$ is a symmetric Banach--Finsler manifold;
    \item there exists a JB-algebra structure on $(V, \lVert\cdot\rVert_v)$ with cone of squares $V_+$.
\end{enumerate}
\end{theorem}
We will prove this theorem as a consequence of \Cref{MAINTHM}, using the result in \cite[Thm.~3.5]{LRW25} that each $d_T$-symmetry is a gauge-reversing bijection; see \Cref{MAINTHMC2}. This establishes the second row in the following table of  correspondences between symmetric spaces and Jordan structures.
\ind 

\begin{center}
\renewcommand{\arraystretch}{2}
\begin{tabular}{ |c||c| }
 \hline
 \multicolumn{2}{|c|}{\bf{Correspondences between symmetric spaces and Jordan structures}} \\ \hline
 Self-dual and homogeneous cones & JH-algebras \\ \hline
 Thompson-symmetric cones &  JB-algebras\\ \hline
 Symmetric tube domains & JB$^*$-algebras \\ \hline
 Bounded symmetric domains &  JB$^*$-triples\\
 \hline
\end{tabular}
\end{center}
\vspace{0.5em}

\paragraph{Algebraic perspective.} Our method of constructing a Jordan algebra product in \Cref{MAINTHM} from the data of a gauge-reversing map is inspired by the works \cite{Springer73} and \cite{McC77} of Springer and McCrimmon, who proved a similar result for finite-dimensional Jordan algebras over a field. An extensive algebraic theory of Jordan algebras has been developed, see for example \cite{JacStruRep}, \cite{JacStrJA}, \cite{BraKoe}, \cite{ZelPII}, and, for a modern comprehensive account, \cite{McC04}. As an example, any associative algebra $A$ over a commutative ring $K$ with $\frac{1}{2}\in K$ gives rise to Jordan algebra $A^+$ with product 
\begin{equation}\label{circDef}
x \circ y = \frac{1}{2}(xy+yx).
\end{equation}
\ind Springer noted that the algebraic structure of a unital Jordan algebra, finite-dimensional over the field $K$, is completely contained in its Jordan inverse map. He captured the essential features of the Jordan inverse map in the following definition. A \emph{$j$-structure} on a finite-dimensional $K$-vector space $V$ is a birational involution that is homogeneous of degree $-1$ and satisfies the weak Hua identity $$j(v+x)+j(v+j(x))=v = j(v)$$ as well as a third axiom which we shall not state. Springer proved in \cite[Thm. 6.5]{Springer73} that the category of $j$-structures over $K$ is equivalent the category of finite-dimensional unital Jordan algebras over $K$, provided that $\frac{1}{2} \in K$.\\
\ind McCrimmon observed in \cite{McC66} that a satisfactory Jordan theory in characteristic $2$ cannot be modeled on the linear product \eqref{circDef}, but should be based on the quadratic map given by
\begin{equation}\label{UDef}
U(x)y = xyx.    
\end{equation}
He abstracted the properties of the quadratic map $U\colon A \to \End(A)$ in the definition of \emph{unital quadratic Jordan algebras}, see \Cref{QJADef}. When $\char K \neq 2$, the linear Jordan product can be recovered from the quadratic operators via
$$
x \circ y = \tfrac{1}{2}\bigl(U(x+y)-U(x)-U(y)\bigr)e,
$$
and the categories of linear and quadratic Jordan algebras are equivalent (we summarize the proof from \cite{Jac69} in \Cref{uqJASec}). A definition of JB-algebras based on the quadratic operators is given in \Cref{AltDef}. \\
\ind To extend Springer's theory of $j$-structures to characteristic $2$, McCrimmon modified it slightly, by removing the third axiom and instead strengthening Hua's identity to a version involving two indeterminates \cite[Sec. 3]{McC77}. He showed that such $h$\emph{-structures} are categorically equivalent to unital quadratic Jordan algebras over base fields of arbitrary characteristic.\\
\ind We will follow the same approach as McCrimmon towards constructing a Jordan product from an inversion-like map, barring two differences. First, our map is defined on the open cone of a, possibly infinite-dimensional, complete order unit space rather than a Zariski dense open subset of a finite-dimensional vector space. Second, instead of stipulating that the map satisfy Hua's identity, we require it to be gauge-reversing. Somewhat surprisingly, the former condition is a consequence of the latter (see \Cref{HuaPhi}).

\begin{theorem}[Hua's identity]\label{INTROHuaPhi}
Let $\Phi\colon C \to D$ be a gauge-reversing bijection between the open cones $C$ and $D$ of complete order unit spaces. Then $\Phi$ is continuously differentiable and for all $x,y \in C$ we have
\begin{equation}\label{INTROHua}
\Phi(x) - \Phi(x+y) = -D\Phi(x)(\Phi^{-1}(\Phi(x)+\Phi(y))).
\end{equation}
\end{theorem}

\textbf{Outline of proof.} Up to the proof of the main theorems in \Cref{MAINTHMsection}, we assume that $(V, V_+, v)$ and $(W, W_+, w)$ are two complete order unit spaces, with open cones $C = V_+^{\circ}$ and $D = W_+^{\circ}$, and that $\Phi\colon C \to D$ is a gauge-reversing bijection.\\
\ind In \Cref{ECSect} we define cones $C_{usc}$ and $C_{lsc}$ consisting of upper respectively lower semicontinuous, affine functions on the state space of $V$, which are obtained by adjoining to $C$ all limits of decreasing respectively increasing nets in $C$. The map $\Phi\colon C \to D$ is extended to a gauge-reversing bijection $\Phi_{usc}\colon C_{usc} \to D_{lsc}$. We establish that the cone $C_{usc}$ is \emph{strongly atomic}, i.e.\ each element in $C_{usc}$ is the supremum of the extremal vectors that it dominates. In fact, we show that under $\Phi_{usc}$ the extremal vectors of $C_{usc}$ correspond to the pure states on $W$, whose ample existence is guaranteed by the Krein--Milman theorem.\\
\ind In \Cref{evgdSec} several consequences of strong atomicity of $C_{usc}$ are deduced. First, we show in \Cref{DPhi} that $\Phi$ is Gateaux-differentiable in all positive directions. Second, we prove \Cref{HuaPhi} in \Cref{HuasubSec}. Third, in \Cref{ThmBSec} we
give a short new proof of \cite[Theorem B]{NollSch77, Sch78} asserting that a gauge-preserving bijection $\Psi\colon C \to D$ extends to a linear isomorphism from $V$ onto $W$.\\
\ind In \Cref{MTSec} we prove that $\Phi$ is Fréchet-differentiable and that for each point $x\in C$ minus the derivative at $x$ defines an isomorphism $-D\Phi(x)\colon (V, V_+, x) \to (W, W_+, \Phi(x))$ of order unit spaces. We show that $S_x := -[D\Phi(x)]^{-1} \circ \Phi$ is a $d_T$-symmetry of $C$ at $x$. In particular, we obtain our prospective inverse map as $j := S_v$. We define the quadratic representation of an element $x\in C$ by $P(x) := -[Dj(x)]^{-1}$ and show that $P$ extends to a quadratic polynomial map $P\colon V \to L(V)$. As a consequence we prove that every gauge-reversing bijection $\Phi\colon C \to D$ is bianalytic. The map $P$ is shown to satisfy be a quadratic JB-algebra (see \Cref{AltDef}), finishing the proof of \Cref{MAINTHM}.\\
\ind As a byproduct of the proof, we answer in the affirmative the question raised by Noll and Sch\"afer \cite{NollSch77} whether two, possibly non-complete, order unit spaces connected by a gauge-reversing bijection are necessarily isomorphic.  

\begin{theorem}\label{grevThm}
If $V$ and $W$ are order unit spaces such that there exists a gauge-reversing bijection $\Phi\colon C \to D$, then $V$ and $W$ are isomorphic.
\end{theorem}

Walsh proved the finite-dimensional case of \Cref{grevThm} in \cite{Walsh18}. We shall prove the general case as \Cref{ncThmAB} in \Cref{ncmplSec}.
 A version of \Cref{MAINTHM} for non-complete order unit spaces is included in that section as well.

\ind 
Walsh has strengthened \Cref{MAINTHM} in case $V$ is finite-dimensional in \cite{WalshSelecta}, by showing that the existence of an order-anti-isomorphism of $V^{\circ}_+$, not assumed homogeneous of degree $-1$, implies that there exists a gauge-reversing bijection--and hence a JB-algebra structure--on $V$. It would be interesting to see if this also holds in infinite dimensions.\\

The following section covers preliminary concepts and results concerning differentiable and analytic maps, order unit spaces, and JB-algebras. 

\section{Preliminaries}\label{prelSec}

\subsection{Differentiable maps}\label{diffSec}
In this section, we recall the definition of Gateaux- and Fréchet-differentiable maps, and prove a chain rule for the Gateaux-derivative.

\begin{definition} Let $(X, \lVert\cdot\rVert_X)$ and $(Y, \lVert\cdot\rVert_Y)$ be Banach spaces. Let $F \subset X$ be a subset and $f\colon F \to Y$ a function.
\begin{itemize}
\item[$(1)$] Given a point $x\in F$ and a vector $p \in X$, we say that $f$ is \emph{Gateaux-differentiable} at $x$ in the direction $p$ if for sufficiently small $\mu > 0$ we have $x + \mu p \in F$, and there is a vector $q\in Y$ such that
$$\lim_{\mu \downarrow 0} \mu^{-1}\lVert f(x+\mu p)-f(x)-\mu q\rVert_Y = 0.$$
Plainly, the vector $q$ is unique if it exists, and called the \emph{Gateaux-derivative} of $f$ at $x$ in the direction $p$. We emphasize that, due to the right limit in the definition, our notion of Gateaux-derivative is one-sided. Nonetheless, if $\lambda \ge 0$, then $\lambda q$ is the Gateaux-derivative of $f$ at $x$ in the direction $\lambda p$.
\item[$(2)$] Let $x$ belong to the interior of $F$. We say that $f$ is \emph{Fréchet-differentiable} at $x$ if there exists a continuous linear map $A\colon X \to Y$ such that
$$\lim_{h \to 0} \lVert h\rVert_X^{-1} \lVert f(x+h) - f(x) - A(h)\rVert_Y = 0.$$
Again, the linear map $A$ is unique if it exists. We call $A$ the \emph{Fréchet-derivative} of $f$ at $x$ and denote it $Df(x)$. Note that if $f$ is Fréchet-differentiable at $x$, then $f$ is continuous at $x$.
\end{itemize}
\end{definition}

We will use the following well-known chain rule for Fr\'{e}chet- and Gateaux-derivatives.

\begin{lemma}\label{chr} Let $X, Y, Z$ be Banach spaces. Let $f\colon F \to G$ and $g\colon G \to Z$ be functions, defined on subsets $F \subset X$ and $G \subset Y$. Let $x\in F$ and set $y := f(x) \in G$.
\begin{itemize} 
\item[$(1)$] If $f$ is Fréchet-differentiable at $x$, and $g$ is Fréchet-differentiable at $y$, then $g \circ f$ is Fréchet-differentiable at $x$ with Fréchet-derivative given by the chain rule
$$D(g \circ f)(x) = Dg(y) \circ Df(x).$$
\item[$(2)$] Let $p \in X$, $q \in Y$ and $r \in Z$ be vectors such that $q$ is the Gateaux-derivative of $f$ at $x$ in the direction $p$, and $r$ is the Gateaux-derivative of $g$ at $y$ in the direction $q$. Assume additionally that $g$ is Lipschitz on a neighbourhood of $y$ in $G$. Then $r$ is the Gateaux derivative of $g \circ f$ at $x$ in the direction $p$.
\end{itemize}
\end{lemma}

\begin{proof}
Part (1) is standard, so we only show part (2).\\
\ind Since $g$ is Lipschitz on a neighborhood of $y$ in $G$, there exist $\delta > 0$ and $L > 0$ such that $\lVert g(y_1) - g(y_2)\rVert_Z \le L\lVert y_1 - y_2\rVert_Y$ whenever $y_i \in G$ with $\lVert y_i - y\rVert_Y < \delta$ for $i=1,2$. For $\mu > 0$ sufficiently small we have $\lVert \mu q\rVert_Y < \delta/2$ and $\lVert f(x+\mu p) - f(x) - \mu q\lVert_Y < \delta/2$, hence $\lVert f(x + \mu p) - y\lVert < \delta$ and $ \lVert f(x) + \mu q - y\rVert < \delta$. Since $g$ is $L$-Lipschitz on $\{y'\in G: \lVert y'-y\rVert_Y < \delta\}$, for sufficiently small $\mu > 0$ we have
\begin{align*}
\lVert g(f(x+\mu p)) - g(f(x)+\mu q)\rVert_Z \le L \lVert f(x+\mu p)-f(x)-\mu q\rVert_Y
\end{align*}
and hence by the triangle inequality
\begin{align*}
\lVert (g \circ f)(x + \mu p) - (g \circ f)(x) - \mu r \rVert_Z &\le  L \lVert f(x+\mu p)-f(x)-\mu q\rVert_Y \\&\hskip .29cm + \lVert g(f(x)+\mu q)-g(f(x))-\mu r\rVert_Z.
\end{align*}
As $q$ is the Gateaux-derivative of $f$ at $x$ in the direction $p$ and $r$ is the Gateaux-derivative of $g$ at $y$ in the direction $q$, it follows that $r$ is the Gateaux-derivative of $g \circ f$ at $x$ in the direction $p$.
\end{proof}

\subsection{Analytic maps}
Let $(X, \lVert\cdot\rVert_X)$ and $(Y, \lVert\cdot\rVert_Y)$ be Banach spaces. The \emph{norm} of a multilinear map $\phi\colon X^k \to Y$ is given by
$$\lVert \phi \rVert := \sup \{\lVert \phi(x_1,x_2,\ldots,x_k)\rVert_Y: x_i \in X, \lVert x_i\rVert_X \le 1\} \in [0, \infty].$$
We say $\phi$ is \emph{continuous} if $\lVert \phi \rVert < \infty$.

%We denote $X^k$ the $k$-fold cartesian product $X \times X \times \ldots \times X$ endowed with the maximum norm $\lVert (x_1,\ldots,x_k)\rVert_{X^k} = \max_{i=1}^k \lVert x_i\rVert_X.$ 

\begin{definition} Let $k \in \Z_{\ge 0}$. We say a map $c\colon X \to Y$ is a \emph{(continuous) homogeneous polynomial map of degree} $k$ if it is defined by a (continuous) symmetric multilinear map $\overline{c}\colon X^k \to Y$ via 
\begin{equation*}
c(x) := \overline{c}(x,x,\ldots,x).
\end{equation*}
\end{definition}
The map $\overline{c}$ is uniquely determined by $c$ according to the following formula, taken from \cite[p. 274]{Bochnak}, which follows readily from the inclusion--exclusion principle:
\begin{align}\label{cbc}
\overline{c}(x_1,\ldots,x_k) = \frac{1}{k!}\sum_{\substack{\epsilon_1 = 0,1\\ \cdots\cdots \\ \epsilon_k = 0,1}} (-1)^{k-\epsilon_1-\ldots-\epsilon_k} c\left(\sum_{i=1}^k \epsilon_i x_i\right).
\end{align}

\begin{definition}\label{AnMap}
A map $f\colon U \to Y$ defined on a neighborhood $U$ of a point $x \in X$ is said to be \emph{analytic at }$x$ if there exist $r > 0$ with $\{z \in X: \lVert z-x\rVert_X < r\} \subset U$ and for each $k \in \Z_{\ge 0}$ a continuous homogeneous polynomial map $c_k\colon X \to Y$ of degree $k$ such that $\limsup_k \lVert \overline{c_k}\rVert^{1/k} \le r^{-1}$ and for all $h \in X$ with $\lVert h\rVert_X < r$ one has
\begin{equation}
    f(x+h) = \sum_{i=0}^{\infty} c_i(h).
\end{equation}
We say $f$ is \emph{analytic} if $f$ is analytic at each point of its domain $U$.
\end{definition}

Many classical results about analytic functions of a single variable extend to the Banach space setting, among which we record for later reference the identity principle.

\begin{theorem}\label{idprinc}
Let $U$ be a connected open subset of the Banach space $X$, and let $f,g\colon U \to Y$ be analytic maps. Suppose there exists a nonempty open subset $U_0$ of $U$ such that $f(x) = g(x)$ for each $x\in U_0$. Then $f$ and $g$ are identically equal on $U$.
\end{theorem}
\begin{proof}
See \cite[Corollary, p. 1080]{Whitt}.
\end{proof}

\begin{theorem}\label{invan} Let $B$ be a Banach algebra with identity. Then the set $G := \{b \in B: b \text{ is invertible}\}$ is open and the inverse map $G \to G$, $b \mapsto b^{-1}$ is analytic.
\end{theorem}
\begin{proof}
This is shown in \cite[Theorem VII.2.2]{Con90}.
\end{proof}

\subsection{Semi-continuous functions}\label{scfSec}
In this section we recall some facts concerning (upper or lower) semicontinuous functions, referring to \cite{Beer} for proofs.

\begin{definition}
Let $X$ be a topological space.
\begin{itemize}
\item[$(1)$] A function $g\colon X \to [-\infty, \infty)$ is called \emph{upper semicontinuous} if for every $\lambda \in \R$ the set $g^{-1}([-\infty,\lambda))$ is open in $X$.
\item[$(2)$] A function $h\colon X \to (-\infty, \infty]$ is called \emph{lower semicontinuous} if for every $\lambda \in \R$ the set $h^{-1}((\lambda,\infty])$ is open in $X$.
\item[$(3)$] A net $(g_{\alpha})_{\alpha}$ of functions $g_{\alpha}\colon X \to (-\infty, \infty]$ is called \emph{almost nondecreasing} if for every $\epsilon > 0$ there exists an index $\alpha_0$ such that for all indices $\alpha \ge \alpha' \ge \alpha_0$ one has $g_{\alpha} \ge g_{\alpha'} - \epsilon$.
\item[$(4)$] A net $(h_{\beta})_{\beta}$ of functions $h_{\beta}\colon X \to [-\infty, \infty)$ is called \emph{almost nonincreasing} if $(-h_{\beta})_{\beta}$ is an almost nondecreasing net.
\end{itemize}
\end{definition}

\begin{lemma}\label{sclim}
Let $X$ be a topological space.
\begin{enumerate}[label=\normalfont{(\arabic*)}]
\item Let $(g_{\alpha})_{\alpha}$ be an almost nonincreasing net which consists of upper semicontinuous functions $g_{\alpha}\colon X \to [-\infty, \infty)$. Then the function $g\colon X \to [-\infty, \infty)$ defined $g(x) := \liminf_{\alpha} g_{\alpha}(x)$ is upper semicontinuous and equal to the pointwise limit of the net $(g_{\alpha})_{\alpha}$.
\item Let $(h_{\beta})_{\beta}$ be an almost nondecreasing net which consists of lower semicontinuous functions $h_{\beta}\colon X \to (-\infty, \infty]$. Then the function $h\colon X \to (-\infty, \infty]$ defined by $h(x) := \limsup_{\beta} h_{\beta}(x)$ is lower semicontinuous and equal to the pointwise limit of the net $(h_{\beta})_{\beta}$.
\end{enumerate}
\end{lemma}
\begin{proof}
See \cite[Theorem 1.3.1]{Beer}
\end{proof}

\begin{lemma}\label{scatt}
Let $X$ be a compact topological space.
\begin{itemize}
\item[$(1)$] Let $g\colon X \to [-\infty, \infty)$ be an upper semicontinuous function. Then $g$ is bounded above and attains its supremum.
\item[$(2)$] Let $h\colon X \to (-\infty, \infty]$ be a lower semicontinuous function. Then $h$ is bounded below and attains its infimum.
\end{itemize}
\end{lemma}
\begin{proof}
See \cite[Theorem 1.3.4]{Beer}.
\end{proof}

\subsection{Order unit spaces}\label{ousSec}
In this section we introduce order unit spaces and state Kadison's theorem giving a functional representation of complete order unit spaces as \Cref{KadThm}. We conclude with a discussion of norm completions.\\
\ind Let $V_+$ be a proper cone in a real vector space $V$, i.e.\ $V_+ + V_+ \subset V_+$ and $\lambda V_+ \subset V_+$ for all $\lambda \ge 0$, as well as $V_+ \cap -V_+ = \{0\}$. A partial order $\le$ on $V$ is defined by
$$x \le y \iff y - x \in V_+.$$
We call such a pair $(V, V_+)$ a \emph{partially ordered vector space}. It is called \emph{archimedean} if for every $x\in V$ the set $\{\lambda x: \lambda \ge 0\}$ has an upper bound only if $x\le 0$. An element $v\in V_+$ is called an \emph{order unit} if for every $x\in V$ there exists $\lambda > 0$ such that $x\le \lambda v$. A triple $(V, V_+, v)$ is called an \emph{order unit space} if $(V, V_+)$ is an archimedean partially ordered vector space with order unit $v$.\\
\ind Let $(V, V_+, v)$ be an order unit space. Its \emph{order unit norm} is the norm $\lVert \cdot \rVert_v$ on $V$ defined by
$$\lVert x\rVert_v := \{\inf \lambda > 0: -\lambda v\le x\le \lambda v\}.$$
 Its \emph{open cone} is the $\lVert\cdot\rVert_v$-interior $V^{\circ}_+$ of $V_+$, given by 
 \begin{equation}\label{Vcirc}
 V^{\circ}_+ = \{x \in V: \epsilon v \le x\text{ for some } \epsilon > 0\}.
 \end{equation}
The elements of $V^{\circ}_+$ are precisely the order units of the partially ordered vector space $(V, V_+)$. If $v' \in V^{\circ}_+$, then the norm $\lVert \cdot\rVert_{v'}$ equivalent to $\lVert \cdot\rVert_v$. We call $(V, V_+, v)$ a \emph{complete order unit space} if $(V, \lVert\cdot\rVert_v)$ is a Banach space.\\

\begin{example}\label{AS}
Let $K$ be a compact convex subset of a locally convex space $E$. We write $A(K)$ for the real linear space of continuous affine functions $a\colon K \to \R$. We give $A(K)$ the pointwise ordering, so that $A(K)_+ = \{a\colon K \to [0,\infty) \text{ affine continuous}\}$. The triple $(A(K), A(K)_+; \mathbbm{1}_K)$ is a complete order unit space, where $\mathbbm{1}_K: K \to \R$ is identically equal to $1$.
\end{example}

We now recall a fundamental theorem of Kadison, stating that each complete order unit space is isomorphic to $A(K)$ for a certain compact convex subset $K$ of $V^*$, called the state space.\\
\ind Let $(V, V_+, v)$ be an order unit space. Its \emph{state space} is the set $K$ of positive linear functionals $\rho\colon V \to \R$ such that $\rho(v) = 1$. Since each state is $\lVert \cdot \rVert_v$-continuous, $K$ is a subset of the continuous dual $V^*$ of $(V, \lVert\cdot\rVert_v)$. We endow $V^*$ with the wk*-topology. It follows from the Banach--Alaoglu theorem that $K$ is a wk*-compact convex subset of $V^*$.

\begin{theorem}[Kadison]\label{KadThm}
Let $(V, V_+, v)$ be a complete order unit space. For $x\in V$ define  $\tilde{x}\colon K \to \R$ by $\tilde{x}(\rho) := \rho(x)$ for all $\rho \in K$. Then there exists an isomorphism of order unit spaces 
\begin{align}\label{funrep}
(V, V_+, v) &\stackrel{\sim}{\to} (A(K), A(K)_+, \1_K),\\
x &\mapsto \tilde{x}.\notag
\end{align}
\end{theorem}
\begin{proof}
We refer to \cite[Theorem II.1.8]{Alfsen71}.
\end{proof}

\begin{remark}\label{remAKE} If $V$ is a, possibly non-complete, order unit space, then the assignment $x \mapsto \tilde{x}$ defines an isometric isomorphism of $V$ onto a norm dense subspace of $A(K)$, namely the space $A(K, V^*)$ consisting of all functions on $K$ that extend to a continuous real-valued affine function on $V^*$ (see again  \cite[Theorem II.1.8]{Alfsen71}). In particular, the norm completion $\hat{V}$ of $V \cong A(K, V^*)$ is identified as  $\hat{V} \cong A(K)$.
\end{remark}

Of fundamental importance in the theory of compact convex sets is the Krein--Milman theorem, ensuring the existence of extremal points. Recall that if $L$ is a convex subset of a vector space $E$, then an \emph{extremal point of} $L$ is a point $\rho \in L$ such that for all $0 < \lambda < 1$ and $\sigma,\tau \in L$ with $\rho = (1-\lambda)\sigma + \lambda \tau$ one has $\sigma = \tau = \rho$, i.e., $\rho$ is not a proper convex combination of points in $L$ different from $\rho$. The set of extremal points of $L$ is denoted $\delta_e L$.

\begin{theorem}[Krein--Milman]\label{KrMm}
Let $L$ be a compact convex subset of a locally convex space. Then $L$ is the closed convex hull of the set $\delta_e L$ of extremal points of $L$.
\end{theorem}
\begin{proof}
We refer to \cite[Theorem V.7.4]{Con90}.
\end{proof}

An extremal point of the state space $K$ of an order unit space $(V, V_+, v)$ is called a \emph{pure state}. So, the set of pure states of $V$ is denoted $\partial_e K$. It follows from the Krein--Milman theorem that if $a,a' \in A(K)$ agree on $\delta_e K$, then $a = a'$ holds; see \Cref{pureIneq} for a more general statement.\\
\ind The following lemma on the completion of an order unit space will be applied in \Cref{ncmplSec}.

\begin{lemma}\label{VhatL}
Let $(V, V_+, v)$ be an order unit space. Let $\hat{V}$ be the completion of the normed space $(V, \lVert\cdot\rVert_v)$, and let $\hat{V}_+$ be the closure of $V_+$ in $\hat{V}$. Then $(\hat{V}, \hat{V}_+, v)$ is a complete order unit space and the order unit norm on $\hat{V}$ coincides with the extension of the norm $\lVert\cdot\rVert_v$ on $V$ to its completion.
\end{lemma}
\begin{proof}
By \Cref{remAKE}, we may assume that $V = A(K, E)$ for some locally convex space $E$ and compact convex subset $K$ of $E$, so that $\hat{V} = A(K)$. Let us verify that the closure of $A(K, E)_+$ inside $A(K)$ equals $A(K)_+$. Plainly the closure of $A(K,E)_+$ is contained in $A(K)_+$. Conversely, let $x \in A(K)_+$. Since $A(K, E)$ is norm dense in $A(K)$ we can find a sequence $(x_n)_n$ in $A(K, E)$ converging to $x$ in $\lVert\cdot\rVert_v$, i.e.\ uniformly on $K$. Therefore, the sequence $(\rho_n)_n$ defined by $\rho_n := \inf_K x_n$ converges to $\inf_K x \ge 0$. It follows that the sequence $(y_n)_n$ defined by $y_n := x_n + \max(0, -\inf x_n)v$ belongs to $A(K, E)_+$ and converges to $x + \max(0, -\inf_K x)v = x$. We conclude that $x$ is contained in the closure of $A(K, E)_+$, as desired.\\
\ind The final assertion about coincidence of the two norms on $\hat{V}_+$ holds because both norms agree with the supremum norm on $A(K)$. 
\end{proof}

In the forthcoming section, we will identify $x$ with $\tilde{x}$ and write $x(\rho)$ instead of $\tilde{x}(\rho)=\rho(x)$, for $x \in V$ and $\rho \in K$.

\subsection{Gauges and Thompson's metric}
%XXX: Mention Finsler metric
Let $(V, V_+, v)$ be an order unit space. On the open cone $C = V^{\circ}_+$ there exists a natural metric $d_T$ called Thompson's metric. We will show that $d_T$ is locally bi-Lipschitz to the metric associated with the order unit norm $\lVert\cdot\rVert_v$. Moreover, we define gauge functions, which measure or gauge how small or large one element of $C$ is compared to another.\\ \ind In the following definition, observe that since $C$ consists of the order units of $(V, V_+)$, for any two $x,y\in C$ there exists $\lambda > 0$ such that $x \le \lambda y$.

\begin{definition}\label{GaugeDef}
Let $(V, V_+, v)$ be an order unit space and $C = V_+^{\circ}$ its open cone.
\begin{enumerate}[label=\normalfont{(\arabic*)}]
\item We define \emph{Thompson's metric} $d_T\colon C \times C \to \R_{\ge 0}$ by
\begin{equation}
    d_T(x, y) = \inf\{\lambda > 0: x \le \lambda y \text{ and }y\le \lambda x\}.
\end{equation}
\item We define the gauge-functions $L\colon C\times C\to \R_{>0}$ and $M\colon C \times C \to \R_{>0}$ by
\begin{equation}
m(x,y) := \sup \{\lambda > 0: \lambda y \le x\} \in \R_{>0},\label{ldef}
\end{equation}
and
\begin{equation}
M(x,y) := \inf \{\mu  > 0: x \le \mu y\} \in \R_{>0}.\label{mdef}
\end{equation}
\end{enumerate}
%with the convention that $M(x,y) = \inf \emptyset = \infty$ if $x \le \mu y$ for no $\mu \in (0, \infty)$.
\end{definition}

In terms of the gauge functions, Thompson's metric is given by
\begin{equation}\label{dTeq}
d_T(x, y) = \log \max\{M(x,y), M(y,x)\}.
\end{equation}

One deduces that $d_T$ satisfies the triangle inequality and is nonnegative from the inequality
\begin{equation}\label{dMeq}
M(x,z) \le M(x,y)M(y,z),
\end{equation}
valid for all $x,y,z \in C$.

\begin{example}
Using Kadison's \Cref{KadThm}, we may view elements $x, y \in C$ as continuous functions from the state space $K$ to $\R_{>0}$. Their quotient $x/y\colon K\to \R_{>0}$, $\psi \mapsto x(\psi)/y(\psi)$ is a continuous function on $K$ as well. The minimum and maximum values of $x/y$ over $K$ are given by the gauge function values $m(x,y)$ resp. $M(x,y)$, that is 
\begin{equation}
m(x,y) = \min_{\psi \in K} \frac{x(\psi)}{y(\psi)} \text{ and }
M(x,y) = \max_{\psi \in K} \frac{x(\psi)}{y(\psi)}.
\end{equation}

\end{example}

The gauge functions have the following elementary properties. 
\begin{lemma} Let $x,y \in C$ and $\lambda > 0$. Then
\begin{align}\label{LM}
m(y,x) &= M(x,y)^{-1},\\
M(\lambda x,y) &= \lambda M(x,y)\\
M(x,\lambda y)&=\lambda^{-1}M(x,y)\\
m(y,\lambda x)&=\lambda^{-1}m(y,x)\\
m(\lambda y,x)&=\lambda m(y,x)&\\
d_T(\lambda x,\lambda y)&=d_T(x, y).
\end{align}
\end{lemma}
\begin{proof}
Left to the reader.
\end{proof}

%Besides its completeness Thompson also showed that the topology of the Thompson metric on 
% agreed with the relative topology from the Banach space topology on V. In terms of the preceding notation, the order unit norm for 
% is given by

\begin{lemma}\label{dTlN}
Let $\lambda > 0$. For all $x,y \in C$ with $x,y \ge \lambda^{-1} v$ we have
$$d_T(x,y) \le \lambda\lVert x-y\rVert_v.$$
\end{lemma}
\begin{proof}
For each state $\psi$ on $(V, v)$ we have, because $y(\psi)\ge m(y,v)\ge \lambda^{-1}$, that
$$\frac{x(\psi)}{y(\psi)} = 1 + \frac{x(\psi) - y(\psi)}{y(\psi)} \le 1 + \frac{\lVert x-y\rVert_v}{m(y,v)} \le 1 + \lambda \lVert x-y\rVert_v.$$
This shows that $M(x,y) \le 1 + \lambda\lVert x-y\rVert_v,$ and the same holds with $x$ and $y$ interchanged. We conclude that
$$d_T(x,y) \le e^{d_T(x,y)} - 1 = \max\{M(x,y), M(y,x)\} - 1 \le \lambda\lVert x-y\rVert_v.$$
\end{proof}

\begin{lemma}\label{NldT}
Let $\lambda > 0$. For all $x,y \in C$ with $x,y \le \lambda v$ we have
$$\lVert x-y\rVert_v \le \lambda d_T(x,y).$$
\end{lemma}
\begin{proof} Let $\psi$ be a state on $(V, v)$. We have $y(\psi) \ge e^{-d_T(x,y)}x(\psi)$ because $e^{d_T(x,y)}y \ge M(x,y)y \ge x$. Using that $1-e^{-a}\le a$ for each $a \in \R$, we obtain
$$x(\psi) - y(\psi) \le (1-e^{-d_T(x,y)})x(\psi) \le \lambda d_T(x,y).$$
By interchanging $x$ and $y$ we find that $|x(\psi) - y(\psi)| \le \lambda d_T(x,y)$. The desired inequality follows since $\psi$ is an arbitrary state on $(V, v)$.
\end{proof}

As a consequence of the two preceding lemmata, Thompson's metric $d_T$ and the metric associated with the order unit norm $\lVert\cdot\rVert_v$ are bi-Lipschitz with constant $\lambda$ on the Thompson's metric ball $\{x\in V: d_T(v,x)\le e^{\lambda}\} = [\lambda^{-1}v, \lambda v]$ around $v$ of radius $e^{\lambda}$ for each $\lambda > 0$, hence these metrics are locally bi-Lipschitz. In particular, $d_T$ and $\lVert\cdot\rVert_v$ induce the same topology on $C$.

\subsection{Gauge-reversing bijections}\label{GrbSec}

In this section we define gauge-preserving bijections and gauge-reversing bijections. We show they are isometries for Thompson's metric and, consequently, are locally bi-Lipschitz with respect to the order unit norms. Moreover, we introduce $d_T$-symmetries, and state that these are gauge-reversing bijections.

\begin{definition}\label{opordef} Let $(X, \le)$ and $(Y, \le)$ be partially ordered sets, and let $\Phi\colon X \to Y$ be a map. We say that $\Phi$ is \emph{order-preserving} (resp. \emph{order-reversing}) if for all $x_1, x_2 \in X$ with $x_1 \le x_2$ one has $\Phi(x_1) \le \Phi(x_2)$ (resp. $\Phi(x_1) \ge \Phi(x_2)$). We say $\Phi$ is an \emph{order-isomorphism} (resp. \emph{order-anti-isomorphism} if $\Phi$ is bijective and both $\Phi$ and $\Phi^{-1}$ are order-preserving (resp. order-reversing).
\end{definition}

In the remainder of this section, $(V, V_+, v)$ and $(W, W_+, w)$ denote order unit spaces, and $C = V_+^{\circ}$ and $D = W_+^{\circ}$ their open cones.

\begin{definition}\label{gpgrdef} Let $\Phi\colon C \to D$ be a bijection. We say that $\Phi$ is \emph{homogeneous of degree $1$} (resp. \emph{homogeneous of degree $-1$}) if for all $x \in C$ and $\lambda \in (0, \infty)$ one has $\Phi(\lambda x) = \lambda \Phi(x)$ (resp. $\Phi(\lambda x) = \lambda^{-1} \Phi(x)$). We call  $\Phi$ \emph{gauge-preserving} if for all $x, y \in C$ one has
\begin{equation}\label{gpres}
    m(\Phi(x),\Phi(y)=m(y,x)
\end{equation}
and \emph{gauge-reversing} if all $x, y \in C$ satisfy
\begin{equation}\label{grev}
    m(\Phi(x),\Phi(y))=m(x,y).
\end{equation}
\end{definition}
Since the functions $M$ and $m$ are multiplicatively inverse by \eqref{LM}, one arrives at an equivalent definition if condition \eqref{gpres} is replaced by $M(\Phi(x),\Phi(y))=M(x,y)$ and condition \eqref{grev} by $M(\Phi(x),\Phi(y))=M(y,x)$.

\begin{proposition}\label{graor}For a bijection $\Phi\colon C \to D$, the following statements are equivalent:
\begin{enumerate}[label={\normalfont(\arabic*)}]
\item $\Phi$ is gauge-reversing;
\item $\Phi$ is homogeneous of degree $-1$ and an order-anti-isomorphism
\item $\Phi$ is homogeneous of degree $-1$ and $y - x \in C \iff \Phi(x) - \Phi(y) \in C$ for all $x,y \in C$.
\end{enumerate}
\end{proposition}
\begin{proof}
For (1)$\iff$(2), see \cite[Lemma 2.5]{LRW25}, and for (2)$\iff$(3), see \cite{NollSch77}[Prop. 7.2].
\end{proof}

\begin{proposition}\label{gphop}
For a bijection $\Psi\colon C \to D$, the following statements are equivalent:
\begin{enumerate}
\item[$(1)$] $\Psi$ is gauge-preserving;
\item[$(2)$] $\Psi$ is homogeneous of degree $1$ and an order-isomorphism.
\item[$(3)$] $\Phi$ is homogeneous of degree $1$ and $y - x \in C \iff \Phi(y) - \Phi(x) \in C$ for all $x,y \in C$.
\end{enumerate}
\end{proposition}
\begin{proof}
The proof is similar to \Cref{graor} and therefore omitted.
\end{proof}

\begin{theorem}\label{NSThmB}
Let $\Psi\colon C \to D$ be a gauge-preserving bijection with $\Psi(v) = w$. Then $\Psi$ extends to an isomorphism between the order unit spaces $(V, V_+, v)$ and $(W, W_+, w)$.
\end{theorem}
\begin{proof}
This is proved as Theorem B in \cite{NollSch77} and \cite{Sch78}.
\end{proof}

We will only have need for the previous theorem in the special case that $(V, V_+, v)$ of $(W, W_+, w)$ are complete and there exists a gauge-reversing bijection $\Phi\colon C \to D$. In order to make our exposition more self-contained, we will provide in \Cref{ThmBSec} an independent proof of \Cref{NSThmB} under an additional assumption (namely that the cone $C_{usc}$ be strongly atomic, cf. \Cref{evsubSec}), which is satisfied in our special case by \Cref{supAtom}.
\begin{proposition}\label{dTisom}
Let $\Phi\colon C \to D$ be a gauge-reversing or gauge-preserving bijection. Then $\Phi$ is a $d_T$-isometry.
\end{proposition}
\begin{proof}
Suppose first that $\Phi$ is gauge-reversing. Then for all $x,y \in C$ we have that
\begin{align*}d_T(\Phi(x),\Phi(y)) &= \log \max\{M(\Phi(x),\Phi(y)), M(\Phi(y),\Phi(x))\}\\ &= \log \max \{M(y,x), M(x,y)\}\\
&= d_T(x,y),
\end{align*}
hence $\Phi$ is a $d_T$-isometry. The proof for a gauge-preserving bijection is similar.
\end{proof}

\begin{definition}\label{dTsymDef}
We define a $d_T$\emph{-symmetry}
$S_x\colon C\to C$ at a point $x\in C$ to be an involutive $d_T$-isometry having $x$ as an isolated fixed point.
\end{definition}

\begin{theorem}\label{dTsymgr}
Let $S_x: C \to C$ be a $d_T$-symmetry at some point $x \in C$. Then $S_x$ is a gauge-reversing bijection.
\end{theorem}
\begin{proof}
This is shown in \cite[Theorem 3.5]{LRW25}.
\end{proof} 

\begin{proposition}\label{jcont}
Let $\Phi\colon C \to D$ be a gauge-reversing bijection with $\Phi(v) = w$. Then for all $\lambda > 0$ and $x, y \in C$ with $x, y \ge \lambda^{-1}v$ one has
\begin{equation}\label{jlip}
    \lVert \Phi(x)-\Phi(y)\rVert_w \le \lambda^2 \lVert x-y\rVert_v.
\end{equation}
\end{proposition}
\begin{proof}
In \Cref{dTlN} we showed that $d_T(x,y) \le \lambda \lVert x-y\rVert_v$. Likewise, since $\Phi(x), \Phi(y) \ge \lambda w$ we have by \Cref{NldT} that $\lVert \Phi(x) - \Phi(y)\rVert_w \le \lambda d_T(\Phi(x), \Phi(y))$. Because $\Phi$ is a $d_T$-isometry by \Cref{dTisom}, we conclude that
\begin{equation*}
\lVert \Phi(x) - \Phi(y)\rVert_w \le \lambda d_T(\Phi(x), \Phi(y)) = \lambda d_T(x, y) \le \lambda^2 \lVert x-y\rVert_v.
\end{equation*}
\end{proof}

\begin{theorem}\label{biLip}
Let $\Phi$ be as in \Cref{jcont}. Then $\Phi\colon (C, \lVert\cdot\rVert_v) \to (D, \lVert \cdot \rVert_w)$ is bi-Lipschitz on Thompson's metric balls around $v$ and $w$. In particular, $\Phi$ is locally bi-Lipschitz.
\end{theorem}
\begin{proof}
Fox fixed $\lambda > 0$, the maps $\Phi$ and $\Phi^{-1}$ restrict to bijections between the Thompson's metric balls $[\lambda^{-1}v, \lambda v]$ and $[\lambda^{-1}w, \lambda w]$, and these restrictions are Lipschitz continuous with constant $\lambda^2$ by \Cref{jcont}. Since the norm interiors of $[\lambda^{-1}v, \lambda v]$ for $\lambda > 0$ cover $C$, the corollary follows.
\end{proof}

\subsection{Jordan algebras}\label{uqJASec}
%XXX: Mention Jordan algebras are unital
In this section we first define classical, i.e.\ linear, Jordan algebras. We then provide an alternative axiomatic framework devised by McCrimmon, and show each quadratic Jordan algebra is a linear Jordan algebra. The theory in this section is purely algebraic and is valid over any field of characteristic $\neq 2$ with at least four elements. However, for simplicity we work over the field of real numbers $\R$. Throughout this paper, all our (Jordan) algebras are assumed to have an identity element.

\begin{definition}
A \emph{Jordan algebra} is a real linear space $A$ equipped with a commutative, but not necessarily associative, $\R$-bilinear product that satisfies the so-called \emph{Jordan-identity}
\begin{align}
x \circ (y \circ x^2) &= (x \circ y) \circ x^2 \tag*{(J)}\label{J2}
\end{align}
and has a unit element $e$, i.e.\ $x \circ e = x$, for all $x,y \in A$.
\end{definition}

\begin{example}
Let $(A, \cdot)$ be an associative real algebra with identity $e$. Define on $A$ the following \emph{Jordan product} $\circ$ by
\begin{equation*}
x \circ y := \tfrac{1}{2}(x\cdot y+y\cdot x).
\end{equation*}
Then $(A, \circ)$ is a Jordan algebra with identity element $e$.
\end{example}

For a real vector space $V$, we denote $L(V)$ the real algebra of $\R$-linear endomorphisms of $V$.

\begin{definition}\label{qrep}
Let $(A, \circ)$ be a Jordan algebra, and let $x\in A$ be an element.
\begin{itemize}
\item[$(1)$]We define the \emph{linear representation} of $x$ to be the endormorphism $T(x) \in L(A)$ given by
$$T(x)(y) := x \circ y = y \circ x.$$

\item[$(2)$]The \emph{quadratic representation} of $x$ is defined in terms of the linear representation as
\begin{equation}\label{qrepUdef}
U(x) := 2T(x)^2 - T(x^2) \in L(A).
\end{equation}
\end{itemize}
\end{definition}

The map $T\colon A \to L(A)$ is linear, but the quadratic representation map $U\colon A \to L(A)$ is a quadratic map in the sense of the following definition.

\begin{definition}
Let $A$ and $B$ be real vector spaces. We say a map $U\colon A \to B$ is \emph{quadratic} if it is homogeneous of degree $2$, i.e.\ $U(\lambda x)=\lambda^2 U(x)$ for all $\lambda \in \R$ and $x \in A$, and the map $U: A \times A \to B$ defined by $$U(x,y) :=\tfrac{1}{2}(U(x+y)-U(x)-U(y)) \in B$$
is bilinear.
\end{definition}
Since $U(x,e)y=x \circ y$ for all $x, y \in A$, the bilinear product $x\circ y$ can be reconstructed from the quadratic product $U(x)y = 2x \circ (x \circ y) - x^2 \circ y$. In fact, McCrimmon has expressed the axioms for a Jordan algebra in terms of the quadratic representation.

\begin{definition}\label{QJADef}
A \emph{(unital) quadratic Jordan algebra} is a triple $(A, U, e)$ consisting of a real vector space $A$, a vector $e \in A$ and a quadratic map $U\colon A \to \End(A)$ satisfying the following axioms. Let $U(\cdot,\cdot)\colon A \times A \to \End(A)$ be the symmetric bilinear map given by
\begin{equation}
U(x,y)=\tfrac{1}{2}(U(x+y)-U(x)-U(y)).
\end{equation}
Then for all $x,y,z\in A$, we impose that
\begin{align*}
U(e) &= \Id_A, \tag*{(QJ1)}\\
U(x)U(y,z)x &=U(U(x)y,x)z, \tag*{(QJ2)}\\
U(U(x)y) &=U(x)U(y)U(x). \tag*{(QJ3)}
\end{align*}
\end{definition}

If $A$ is a classical Jordan algebra with unit element $e$ and quadratic  map $U: A \to L(A)$, then $(A, U, e)$ is a quadratic Jordan algebra, as is shown in \cite[§1.2]{Jac69}. Conversely, a quadratic Jordan algebra $(A, U, e)$ gives rise to a classical Jordan algebra with identity element $e$ and quadratic representation map $U$. This result is proved in \cite[§1.3 and §1.4]{Jac69}, among a list of other identities. To make our exposition more self-contained, we reproduce his derivation in the following theorem.

\begin{theorem}\label{QJAThm}
Let $(A, U, e)$ be a quadratic Jordan algebra. Then $A$ is a Jordan algebra with unit element $e$ in the product $a \circ b := U(a,b)e$.
\end{theorem}
\begin{proof}
It is clear that the product $\circ$ is bilinear and commutative. For every $a\in A$ set $$T(a) = U(a, e) \in \End(A).$$
Take $x = e$ in (QJ2) to find using (QJ1) that 
\begin{equation}\label{Tcirc}
T(y)z=U(y,e)z=U(y,z)e = y \circ z,
\end{equation}
so $T(y)$ is the operation of multiplication by $y$. In particular, $T(e) = U(e, e) = U(e) = \Id_A$, so $e$ is the multiplicative identity element. Moreover, $U(x)e = U(x, x)e = x\circ x =: x^2$.\\
\ind Next we show that for every $a\in A$ we have $$U(a) = 2T(a)^2 - T(a^2).$$ Let $\lambda$ be real number, and take $x = a + \lambda e$ and $y = e$ in (QJ3). This gives, since $U(e) = \id_A$, that
$$U(a+\lambda e)^2 = U((a+\lambda e)^2),$$
which expands, using (QJ1) again, as 
$$(U(a)+2\lambda U(a,e)+\lambda^2\Id_A)^2 = U(a^2 + 2\lambda a + \lambda^2 e).$$
Comparison of coefficients at $\lambda^2$ gives
$$2U(a)+4U(a,e)^2 = 4U(a)+2U(a^2,e).$$
Dividing by $2$ and substitute $T(a) = U(a, e)$ and $T(a^2) = U(a^2, e)$ to obtain $U(a) = 2T(a)^2 - T(a^2)$.\\
\ind Finally, we prove the Jordan identity by showing that $T(x)$ and $T(x^2)$ commute. Take $z = e$ in (QJ2) to find using \eqref{Tcirc} that
$$U(x)U(x,e)y = U(x)U(y, e)x = U(U(x)y, x)e = U(x, e)U(x)y,$$
i.e.\ $U(x)$ and $U(x, e)=T(x)$ commute. Since we have shown that $T(x^2) = 2T(x)^2 - U(x)$, it follows that $T(x)$ and $T(x^2)$ commute.
\end{proof}

If $A$ and $B$ are Jordan algebras, then a linear map $\phi: A \to B$ is a Jordan homomorphism if and only if $\phi(1_A) = 1_B$ and $U_B(\phi(x)) \circ \phi = \phi \circ U_A(x)$.
It follows that the categories of Jordan algebras and quadratic Jordan algebras are equivalent (see also \cite[§1.4]{Jac69}).

\begin{definition}\label{invDef}
An element $x\in A$ is called \emph{invertible} if its quadratic representation $U(x)$ is invertible in the Banach algebra $L(A)$.    
\end{definition} 

\begin{remark}\label{invRem}
This notion of invertibility agrees with the one commonly used in the theory of linear Jordan algebras, where an element $x \in A$ is called invertible if there exists $y \in A$ such that $x \circ y = e$ and $x^2 \circ y = x$. Indeed, such element $y$ exists if and only if $U(x)$ is invertible, in which case $y = U(x)^{-1}x$. The element $x^{-1} := U(x)^{-1}x$ is called the \emph{inverse} of $x$, and has quadratic representation $U(x^{-1})=U(x)^{-1}$.
\end{remark}

\subsection{JB-algebras}\label{JBprelSec}
In this section we define JB-algebras, which are Jordan algebras with a compatible norm. The compatibility conditions are customarily stated in terms of the linear product, however, may equivalently be formulated in terms of the quadratic representation operators. This observation allows to present a definition of JB-algebra referencing only the quadratic product.

\begin{definition}\label{JBDef}
A \emph{JB-algebra} is a Jordan algebra $(A, \circ, e)$ that is simultaneously a Banach space in a norm $\lVert \cdot\rVert$ satisfying for all $x,y \in A$ the conditions 
\begin{align}
\lVert x \circ y\rVert &\le \lVert x\rVert \lVert y\rVert,\label{NC1}\\
\lVert x^2\rVert &= \lVert x\rVert^2,\label{NC2}\\
\lVert x^2\rVert &\le \lVert x^2 + y^2\rVert.\label{NC3}
\end{align}
\end{definition}
It follows from \eqref{NC2} that $\lVert e\rVert = 1$ provided $A \neq 0$. Each JB-algebra is an order unit space per the following proposition.

\begin{proposition}\label{JBProp}
Let $A$ be a JB-algebra with norm $\lVert\cdot\rVert$ and identity element $e$.
\begin{itemize}
\item[$(1)$] The set $A_+ := \{x^2: x \in A\}$ of squares in $A$ is the positive cone of an order unit space $(A, A_+)$ with order unit $e$. Moreover, the order unit norm $\lVert \cdot \rVert_e$ coincides with the given norm $\lVert\cdot\rVert$.
\item[$(2)$] For every $x \in A$ the map $U(x): A \to A$ is positive and has norm $\lVert U(x) \rVert = \lVert x\rVert^2_e$.
\end{itemize}

\end{proposition}
\begin{proof}
For (1) see \cite[Lemma 3.3.7 and 3.3.10]{HOSt84}; for 
(2) see \cite[Thm. 1.25]{AS03}.
\end{proof}
  
\ind The following theorem provides a converse.

\begin{theorem}\label{HOSAS}
Let $(A, A_+, e)$ be a complete order unit space which is simultaneously a Jordan algebra that has $e$ as its identity element. Suppose for each $x\in A$, the quadratic representation $U(x) \in L(A)$ of $x$ is positive and has $\lVert\cdot\rVert_e$-operator norm $\lVert U(x)\rVert_{e} = \lVert x\rVert^2$. Then $A$ is a JB-algebra in the order unit norm $\lVert\cdot\rVert_e$ with cone of squares $A_+ = \{x^2: x \in A\}$.
\end{theorem}
\begin{proof}
From \cite[Lemma 1.80]{AS01} it follows that $\lVert x \circ y\rVert_e \le \lVert x\rVert_e \lVert y\rVert_e$ for all $x,y\in A$ and that all squares are positive. Conversely, by [ibidem, Lemma 1.73] each positive element is a square. Since $\lVert\cdot\rVert_e$ is monotone on $A_+ = \{x^2: x \in A\}$, we have $\lVert x^2\rVert_e \le \lVert x^2 + y^2\rVert_e$ for all $x,y \in V$. Finally, according to [ibidem, Theorem 1.81] we have $\lVert x^2\rVert_e = \lVert x\rVert_e^2$ for each $x \in A$. We have shown that the norm $\lVert\cdot\rVert_e$ satisfies \eqref{NC1}, \eqref{NC2} and \eqref{NC3}. We conclude that $A$ is a JB-algebra in the order unit norm $\lVert\cdot\rVert_e$ with cone of squares $A_+$.
\end{proof}

\Cref{QJAThm}, \Cref{JBProp} and \Cref{HOSAS} combine to show the equivalence of the standard definition of JB-algebras encountered in the literature, to the following alternative definition, which we will employ exclusively in the remainder of this article.

\begin{definition}\label{AltDef}
A \emph{JB-algebra} is a complete order unit space $(A, A_+, e)$ together with a quadratic map $U\colon A \to L(A)$ satisfying (QJ1-QJ3) such that for every $x\in A$ the quadratic representation $U(x) \in L(A)$ is positive and has $\lVert\cdot\rVert_e$-operator norm $\lVert U(x) \rVert_{e} = \lVert x\rVert^2_e$.
\end{definition}

The only if direction of \Cref{MAINTHM} is a known result in the theory of JB-algebras.

\begin{theorem}\label{JBigr}
Let $(A, A_+, e, U)$ be a quadratic JB-algebra, and let $\Omega := A_+^\circ$ be its open cone. For $x\in \Omega$ denote $i(x) := x^{-1} (= U(x)^{-1}x)$ the inverse of $x$. Then the map $i\colon \Omega \to \Omega$ is Fréchet-differentiable with derivative $$Di(x)=-U(x)^{-1}.$$
Moreover, $i$ is an involutive gauge-reversing bijection satisfying $i(e)=e$ and $Di(e) = -\Id_A$.
\end{theorem}
\begin{proof}
The Fréchet-derivative of $i$ is determined in \cite[Prop. II.3.3(i)] {FaKor} for Euclidean Jordan algebras; the proof is also valid for JB-algebras.\\
\ind % By \Cref{invRem} we have $U(x^{-1})=U(x)^{-1}$, so $i(i(x))=U(x^{-1})^{-1}x^{-1}=U(x)U(x)^{-1}x=x$ and $i$ is involutive. We have $i(\lambda x)=U(\lambda x)(\lambda x)=(\lambda^2 U(x))^{-1}(\lambda x)=(\lambda^{-2}U(x)^{-1})(\lambda x)=\lambda^{-1}U(x)^{-1}x=\lambda^{-1}i(x)$, hence $i$ is homogeneous of degree $-1$.
It is clear that $i$ is involutive and homogeneous of degree $-1$. In \cite[Lemma 1.31]{AS03} it is shown that $i$ is order-reversing. Since $i=i^{-1}$ this entails that $i$ is an order-anti-isomorphism. We conclude that $i$ is an involutive gauge-reversing bijection. Finally, $U(e)=\Id_A$ yields that $i(e)=U(e)^{-1}e=e$ and $Di(e)=-U(e)^{-1}=-\Id_A$.
\end{proof}

\section{Extended cones of semicontinuous affine functions}\label{ECSect}
It will be useful to extend the given gauge-reversing bijection to certain larger cones of semicontinuous affine functions. We will introduce these cones in \Cref{ecSec}, and construct the extension in \Cref{extSec}. The technical advantage of these extended cones is that they possess a wealth of extremal structure, as will be discussed in \Cref{evsubSec}. 

\subsection{Extension of the cones}\label{ecSec}

Throughout this section, we let $(V, V_+, v)$ be a complete order unit space. We let $K$ be its state space, always considered in the wk*-topology. Recall Kadison's \Cref{KadThm} stating that there exists an isomorphism $(V, V_+, v) \cong (A(K), A(K)_+, \1_K)$ sending $x\in V$ to the function $\tilde{x}: K \to \R$ given by $\tilde{x}(\rho) := \rho(x)$. We will identify each element $x \in V$ with the affine continuous function $\tilde{x}$ on $K$ it defines, thus dropping the tilde in $\tilde{x}$. Hence, we will write $x(\rho)$ for the value $\rho(x) = \tilde{x}(\rho)$.\\
\ind We denote the open cone by $C := V_+^{\circ}$. Since $C = \{x\in V: x \ge \delta v\text{ for some }\delta > 0\}$ and $K$ is compact, we have
$$C = \{a\colon K \to (0, \infty)\text{ affine continuous}\}.$$
\ind The set of all $[0, \infty]$-valued affine functions on $K$ is denoted
\begin{equation*}
A(K)_+^{\infty} := \{g\colon K \to [0, \infty]\text{ affine}\}.
\end{equation*}
We endow $A(K)_+^{\infty}$ with the pointwise ordering given by $a \le b$ if $a(\rho) \le b(\rho)$ for every $\rho \in K$. 
We remark that $A(K)_+^{\infty}$ is a \emph{monotone complete} partially ordered set: each nonincreasing net in $A(K)^{\infty}_+$ has an infimum, and each nondecreasing net in $A(K)_+^{\infty}$ has a supremum.\\
\ind Let us also define a transitive relation $\prec$ on $A(K)^\infty_+$ by $a \prec b$ if $a(\rho) < b(\rho)$ for every $\rho \in K$. Note that for $a,b \in A(K)_+$ we have $a \prec b$ if and only if $b - a \in C$, whereas $a \le b$ if and only if $b - a \in A(K)_+$.\\
\ind Moreover, on $A(K)_+^{\infty}$ we define a pointwise addition and multiplication by scalars in $[0, \infty]$, using the convention that $0 \cdot \infty = 0$. In this way, $A(K)_+^{\infty}$ becomes a semimodule over the semiring $[0, \infty]$.\\
\ind We now introduce two cones in $A(K)_+^{\infty}$ containing $C$, which will figure prominently in the sequel:
\begin{equation}\label{Cusc}
C_{usc} := \{g\colon K \to [0, \infty) \mid \text{ affine upper semicontinuous}\}
\end{equation}
and
\begin{equation}\label{Clsc}
C_{lsc} := \{h\colon K \to (0, \infty] \mid \text{ affine lower semicontinuous}\}.
\end{equation}
We remark that $C_{lsc}$ is closed under addition and multiplication by scalars in $(0, \infty]$ and that $C_{usc}$ is closed under addition and multiplication by scalars in $[0, \infty)$. Furthermore, we have that $$C_{usc} \cap C_{lsc} = \{g\colon K \to (0, \infty)\text{ affine continuous}\} = C.$$

\begin{example}\label{JF}
Let $F$ be a nonempty closed face of $K$. Define the function $I^{\infty}_F\colon K \to [0, \infty]$ by
\begin{equation}
I^{\infty}_F(\rho) = \begin{cases}
1 & \text{if }\rho \in F,\\
\infty &\text{if }\rho \not\in F.
\end{cases}
\end{equation}
We have $I^{\infty}_F \in C_{lsc}$ with $\min I^{\infty}_F = 1$. In particular, if $\psi \in \delta_e K$ is a pure state, then $I^{\infty}_{\{\psi\}} \in C_{lsc}$. 
\end{example}

\begin{remark}\label{Cndv}
It would have been equivalent to define $C_{usc}$ and $C_{lsc}$ as cones of affine semicontinuous functions on the dual cone of $V_+$. Indeed, let $V^*_+$ be the dual cone of $V_+$ in $V^*$, consisting of all positive linear functions on $V$, so that $K = \{\rho \in V^*_+: \rho(v) = 1\}$. Restriction defines an order- and $[0,\infty]$-semimodule-isomorphism
\begin{align}\label{dciso}
\{\tilde{g}\colon V^*_+\to [0,\infty]\text{ affine}\} &\stackrel{\sim}{\to} A(K)^{\infty}_+,\\
\tilde{g} &\mapsto \tilde{g}|_K,
\end{align}
because each $g\in A(K)^{\infty}_+$ has a unique affine extension $\tilde{g}\colon V^*_+ \to [0, \infty]$ given by 
\begin{align*}
\tilde{g}(\rho) =  \begin{cases}
0 &\text{if }\rho = 0,\\
\rho(v) g(\rho(v)^{-1}\rho) &\text{if }\rho \neq 0.
\end{cases}  
\end{align*}
Moreover, $g$ is upper or lower semicontinuous if and only if $\tilde{g}$ is so, whence the isomorphism \eqref{dciso} given by $\tilde{g} \mapsto \tilde{g}|_K$ restricts to order-isomorphisms
\begin{equation*}
\{\tilde{g}: V^*_+ \to [0,\infty) \text{ affine upper semicontinuous}\} \stackrel{\sim}{\to} C_{usc}
\end{equation*}
and
\begin{equation*}
\{\tilde{g}: V^*_+ \to (0,\infty] \text{ affine lower semicontinuous}\} \stackrel{\sim}{\to} C_{lsc}.
\end{equation*}
It follows from this alternative description of $C_{usc}$ and $C_{lsc}$, that if $v' \in C$ is a second order unit and $C'_{lsc}$ and $C'_{usc}$ are the cones defined by \eqref{Clsc} and \eqref{Cusc} but with $K$ replaced by the state space $K' := \{\rho \in V^*_+: \rho(v') = 1\}$ of $(V, V_+, v')$, then there exist canonical cone isomorphisms $C_{usc} \to C'_{usc}$ and $C_{lsc} \to C'_{lsc}$ given by $g \mapsto \tilde{g}|_{K'}$ which respect all structure.
In conclusion, the cones $C_{usc}$ and $C_{lsc}$ are independent of the choice of the order unit $v \in C$ for $(V, V_+)$.
\end{remark}

Our next theorem shows that $C_{lsc}$ is the closure of $C$ inside $A(K)_+^{\infty}$ under the formation of suprema of nondecreasing nets. Similarly, $C_{usc}$ consists precisely of the infima of nonincreasing nets in $C$. We will deduce these approximation results from two lemmata.

\begin{lemma}\label{AlfAf} 
Let $K$ be the state space of a complete order unit space $(V, V_+, v)$. Consider an affine upper semicontinuous function $g\colon K \to [-\infty, \infty)$, and set
$$\Af := \{a \in A(K) :  a \succ g\}.$$
\begin{enumerate}[label={\normalfont(\arabic*)}]
    \item For all $a_1, a_2 \in \Af$ there exists $a \in \Af$ with $a \le a_i$ for $i\in \{1,2\}$.
    \item For every $\rho \in K$ we have $g(\rho) = \inf \{a(\rho): a \in \Af\}$.
\end{enumerate}
\end{lemma}
\begin{proof}
Part (1) follows from \cite[Cor. I.1.4]{Alfsen71} and part (2) follows from [ibidem, Cor. I.1.2].
\end{proof}

We now make a remark about our language concerning nets. Let $E$ be a nonempty subset of $A(K)^{\infty}_+$ such that for any two $h_1, h_2 \in E$ there exists $h \in E$ with $h \ge h_i$ for $i\in \{1,2\}$. Then the inclusion map $E \to A(K)^\infty_+$ is a nondecreasing net indexed by the directed partially ordered set $(E, \le)$. By abuse of language, we will speak of the nondecreasing net $E$. Similarly, if $D$ is a nonempty subset of $A(K)^{\infty}_+$ such that for any two $g_1, g_2 \in D$ there exists $g \in D$ with $g \le g_i$ for $i\in \{1,2\}$, then we will view $D$ as nonincreasing net indexed by itself.

\begin{lemma}\label{SupInfProp}
Let $(V, V_+, v)$ be a complete order unit space, with open cone $C := V_+^{\circ}$.
\begin{enumerate}[label={\normalfont(\arabic*)}]
\item Each nonincreasing net in $C_{usc}$ has an infimum in $C_{usc}$, which is given pointwise.
\item Each nondecreasing net in $C_{lsc}$ has a supremum in $C_{lsc}$, which is given pointwise.
\end{enumerate}
\end{lemma}
\begin{proof}
(1) Note that the empty set (or net) has infimum $0$ in $C_{usc}$. Let $(g_\alpha)_{\alpha}$ be a nonempty nonincreasing net in $C_{usc}$, and denote by $g: K \to [0, \infty)$ its pointwise infimum, which is given by $g(\rho) = \inf_\alpha g_\alpha(\rho)$ for each $\rho \in K$. Then $g$ is affine and upper semicontinuous by \Cref{sclim}(1), whence $g \in C_{usc}$. It is then clear that $g$ is the infimum of the net $(g_\alpha)_{\alpha}$ in $A(K)^\infty_+$, hence in $C_{usc}$.\\
\ind (2) This is shown in a similar way using \Cref{sclim}(2).
\end{proof}

\begin{theorem}\label{scapprox}
Let $(V, V_+, v)$ be a complete order unit space, with open cone $C := V_+^{\circ}$.
\begin{enumerate}[label={\normalfont(\arabic*)}]
\item Each $g \in C_{usc}$ is the pointwise infimum of the nonincreasing net $\{a \in C: g \prec a\}$.
\item  Each $h \in C_{lsc}$ is the pointwise supremum of the nondecreasing net $\{b \in C: b \prec h\}$.
\end{enumerate}
\end{theorem}
\begin{proof}
{\normalfont (1)} In \Cref{AlfAf} we have seen that $g$ is the pointwise infimum of the nonincreasing net $\Af := \{a \in A(K): a \succ g\}$. Now for $a \in A(K)$ we have that $a \succ g$ and $g \ge 0$ imply $a \succ 0$. Therefore $\Af \subset C$ and the proof is complete.\\
\ind {\normalfont (2)} Let $\Bf := \{b \in A(K): b \prec h\}$. Applying \Cref{AlfAf}(1) with $g = -h$ yields that for all $b_1, b_2 \in \Bf$ there exists $b\in \Bf'$ with $b\ge b_i$ for $i\in \{1,2\}$. Since $0 \in \Bf'$ holds, the subnet $\Bf' := \{b\in C: b \prec h\} = \{b \in A(K): 0 \prec b \prec h\} $ is cofinal in $\Bf$. Therefore by part (2) of that theorem, $h(\rho) = \sup\{b(\rho): b \in \Bf\} = \sup\{b(\rho): b \in \Bf'\}$ for every $\rho \in K$. 
\end{proof}

Part (1) of the following lemma asserts that any nonincreasing net in $C_{usc}$ with infimum, say $g \in C_{usc}$, is cofinal in a suitable sense with respect to the net $\Af = \{x \in C: x \prec g\}$  considered in the previous. Part (2) contains the analogous cofinality assertion for a nondecreasing net in $C_{lsc}$. The lemma is a version of Dini's theorem for semicontinuous functions.

\begin{lemma}\label{ocga}
Let $(V, V_+, v)$ be a complete order unit space, with open cone $C := V_+^{\circ}$.
\begin{enumerate}[label={\normalfont(\arabic*)}]
\item Let $(g_{\alpha})_{\alpha}$ be any nonincreasing net in $C_{usc}$ with infimum $g \in C_{usc}$. Let $f \in C$ with $f \succ g$. Then there exist an index $\alpha_0$ and $\epsilon > 0$ with $g_{\alpha_0} \prec f - \epsilon$.
\item Let $(h_{\alpha})_{\alpha}$ be any nondecreasing net in $C_{lsc}$ with supremum $h \in C_{lsc}$. Let $f \in C$ with $f \prec h$. Then there exist an index $\alpha_0$ and $\epsilon > 0$ with $h_{\alpha_0} \succ f + \epsilon$.
\end{enumerate}
\end{lemma}
\begin{proof}
(1) Since $f$ is continuous, the functions $g_{\alpha}-f$ and $g-f$ are upper semicontinuous. From $g \prec f$ it follows by \Cref{scatt}(1) that $\max(g-f)<0$, so there exists $\epsilon > 0$ with $g-f \prec -\epsilon$. For every index $\alpha$ the set $U_{\alpha} := \{\rho \in K: g_{\alpha}(\rho) - f(\rho) < -\epsilon\}$ is open since $U_{\alpha} = (g_{\alpha}-f)^{-1}((-\infty,\epsilon))$. We claim that $(U_{\alpha})_{\alpha}$ is an open cover of $K$. Indeed, because for each $\rho \in K$ we have that $\inf_\alpha (g_{\alpha}(\rho) - f(\rho)) = g(\rho) - f(\rho) < -\epsilon$, there exists an index $\alpha$ such that $g_{\alpha}(\rho) - f(\rho) < -\epsilon$, that is, $\rho \in U_{\alpha}$. By compactness of $K$ we can find finitely many indices $\alpha_1, \ldots, \alpha_n$ such that $K = \bigcup_{i=1}^n U_{\alpha_i}$. Since the net $(g_{\alpha})_{\alpha}$ is downward directed, there exists an index $\alpha_0$ such that for each $i \in \{1,2,\ldots,n\}$ we have $g_{\alpha_i} \ge g_{\alpha_0}$, in particular, $U_{\alpha_i} \subset U_{\alpha_0}$. It follows that $K = U_{\alpha_0}$. We conclude that $g_{\alpha_0} \prec f - \epsilon$, as desired.\\
\ind (2) This is shown analogously to part (1).
\end{proof}

\begin{lemma}\label{scface}
Let $K$ be the state space of the complete order unit space $(V, V_+, v)$.
\begin{enumerate}[label={\normalfont(\arabic*)}]
\item Let $g\colon K \to [-\infty, \infty)$ be affine and upper semicontinuous. Then there exists a pure state $\psi \in \delta_e K$ such that $\psi(g) = \sup g$.
\item Let $h\colon K \to (-\infty, \infty]$ be affine and lower semicontinuous. Then there exists a pure state $\psi \in \delta_e K$ such that $\psi(h) = \inf h$.
\end{enumerate}
\end{lemma}
\begin{proof}
(1) Let $F = \{\rho \in K: g(\rho) = \sup g\}$. \Cref{scatt} asserts that $g$ attains its supremum, so $F$ is nonempty. Moreover, $F$ is a face of $K$. Because $F = g^{-1}([\sup g, \infty))$ and $g$ is upper semicontinuous, $F$ is closed.  We conclude that $F$ is a nonempty, closed face of $K$. The Krein--Milman \Cref{KrMm} gives that $F$ is the closed convex hull of $\delta_e F$; in particular $\delta_e F \neq \emptyset$. Since $F$ is a face of $K$, we have $\delta_e K \cap F = \delta_e F \neq \emptyset,$ as was to be shown.\\
\ind (2) Just like in part (1) one shows that $F = \{\rho \in K: h(\rho) = \inf h\}$ is a nonempty, closed face of $K$, which therefore contains an extremal point of $K$.
\end{proof} %XXX: KrMm -> real reference

\begin{lemma}\label{pureIneq} Let $h, h'\colon K \to (-\infty, \infty]$ be lower semicontinuous affine functions. If $h(\psi) \le h'(\psi)$ for each pure state $\psi \in \delta_e K$, then $h \le h'$.
\end{lemma}
\begin{proof}
\Cref{scapprox} furnishes an increasing net $(h_{\alpha})_{\alpha}$ of continuous affine functions on $K$ such that $h(\rho) = \sup_\alpha h_{\alpha}(\rho)$ for each $\rho \in K$. Now by replacing $h$ with an arbitrary member $h_{\alpha}$ of the net, which again satisfies $h_{\alpha}(\psi) \le h'(\psi)$ for each $\psi \in \delta_e K$, we may and do assume that $h$ is continuous.\\
\ind The function $h_0 = h' - h$ is affine lower semicontinuous and satisfies $h_0(\psi) \ge 0$ for all $\psi \in \delta_e S$. It follows from \Cref{scface}(2) that $\inf h_0 \ge 0$, whence $h' \ge h$. 
\end{proof}

We close this section by extending the gauge functions $L$ and $M$ defined in \Cref{GaugeDef} on $C$ to $A(K)^\infty_+$, and establishing some of their elementary properties.

\begin{definition}
Let $x, y\in A(K)_+^{\infty} \setminus \{0,\infty\}$. We define the quantities 
\begin{align}
m(y,x) := \sup \{\lambda \in [0,\infty): \lambda x \le y\} \in [0, \infty),\label{ldefExt}
\end{align}
and
\begin{align}
M(x,y) := \inf \{\mu \in (0, \infty]: x \le \mu y\} \in (0, \infty].\label{mdefExt}
\end{align}
%with the convention that $M(x,y) = \inf \emptyset = \infty$ if $x \le \mu y$ for no $\mu \in (0, \infty)$.
\end{definition}
As before it holds that
\begin{equation}\label{LMext}
m(y,x) = M(x,y)^{-1},
\end{equation}
which is understood to include the assertion that $m(y,x)=0$ if and only if $M(x,y)=\infty.$

\begin{lemma}\label{PosCheck} 
Let $x, y \in A(K)_+^{\infty}$ such that $x \succ 0$. If $\lambda \in [0, m(x,y))$, then $x - \lambda y \succ 0$.
\end{lemma}
\begin{proof}
It is to be shown for every $\rho \in K$ that $x(\rho) - \lambda y(\rho) > 0$. If $y(\rho) = 0$ then we have $x(\rho) - \lambda y(\rho) = x(\rho) > 0,$ since $x \succ 0$. Suppose next that $y(\rho) > 0$. Then $m(x,y)y \le x$ implies $x(\rho) \ge m(x,y)y(\rho)$. Therefore $x(\rho) - \lambda y(\rho) \ge (m(x,y) - \lambda)y(\rho) > 0$, as $m(x,y) - \lambda > 0$ by assumption.
\end{proof}

\begin{lemma}\label{Maxby}
Let $x,y \in A(K)_+^{\infty}$ with $0 \neq x \neq \infty$ and $y \prec \infty$. Let $\alpha \in (0, \infty)$, $\beta \in [0, \infty)$ and $\gamma \in [0, \alpha m(x,y))$. Then we have
\begin{align}
M(\alpha x + \beta y,x) &= \alpha + \beta M(y,x), \label{Maxby1}\\
m(\alpha x+\beta y,x)&=\alpha +\beta m(y,x), \label{Maxby2}\\
m(\alpha x-\gamma y,x) &= \alpha-\gamma M(y,x), \label{Maxby3}\\
M(\alpha x-\gamma y,x) &= \alpha-\gamma m(y,x). \label{Maxby4}
\end{align}
%with the conventions that $0\cdot \infty = 0$ and $\infty/(a\infty+b)=1/a$, if both sides are defined.
\end{lemma}
\begin{proof}
We content ourselves with proving formula \eqref{Maxby3}, because the other formulae are obtained by similar arguments. Since $0 \neq x \neq \infty$ we have $m(\alpha x,x)=\alpha$ so \eqref{Maxby3} holds if $\gamma = 0$. Now if $0 < \gamma < \alpha m(x,y)=m(\alpha x,y)$, then $\alpha x -  \gamma y \in A(K)^\infty_+$ by \Cref{PosCheck}. Now for $\lambda \in [0, \alpha / \gamma]$ we have
$$\alpha - \gamma \lambda \le m(\alpha x - \gamma y,x) \iff (\alpha - \gamma \lambda)x \le \alpha x - \gamma y \iff y \le \lambda x \iff \lambda \ge M(y,x).$$
Using (the contrapositive of) this equivalence and the fact that $m( \alpha x-\gamma y,x) \le m(\alpha x,x) = \alpha$ and $\alpha/\gamma > m(x,y)^{-1} = M(y,x),$ we conclude that
\begin{align*}
m(\alpha x - \gamma y,x) &= \inf\{\alpha - \gamma \lambda: \lambda \in [0, \alpha/\gamma], \alpha - \gamma \lambda > m(\alpha x - \gamma y,x)\}\\
&= \alpha - \gamma \sup\{\lambda \in [0, \alpha/\gamma]: \lambda < M(y,x)\}\\
&= \alpha - \gamma M(y,x),
\end{align*}
which is this formula \eqref{Maxby3}.
\end{proof}

\subsection{Extension of the gauge-reversing bijection}\label{extSec}
The aim of this section is to show that each gauge-preserving or gauge-reversing bijection has an extension to the extended cones considered in \Cref{ecSec} with the same properties.\\
\ind Let $V$ and $W$ denote complete order unit spaces with open cones $C = V_+^\circ$ and $D = W_+^\circ$. Recall that the extended cones $C_{usc}$ and $C_{lsc}$ have been defined
by \eqref{Cusc} resp. \eqref{Clsc}, and we similarly define the extended cones $D_{usc}$ and $D_{lsc}$. Definitions \ref{opordef} and \ref{gpgrdef} can naturally be generalized to maps between the extended cones, provided we impose an additional condition regarding the image of $0$ or $\infty$. For example, a map $\Phi_{usc}\colon C_{usc} \to D_{lsc}$ is called \emph{gauge-reversing} if $\Phi_{usc}(0)=\infty$ and $M(\Phi(x),\Phi(y))=m(x,y)$ for all $x, y \in C_{usc} \setminus \{0\}$.

\begin{lemma}\label{graorc} For a bijection $\Phi$ from $C_{usc}$ to $D_{lsc}$, or from $C_{lsc}$ to $D_{usc}$, the following statements are equivalent:
\begin{enumerate}[label={\normalfont(\arabic*)}]
\item $\Phi$ is gauge-reversing;
\item $\Phi$ is homogeneous of degree $-1$ and an order-anti-isomorphism.
\end{enumerate}
\end{lemma}
\begin{proof}
The proof, being analogous to \Cref{graor}, is omitted.
\end{proof}

\begin{lemma}\label{gphopc}
For a bijection $\Psi$ from $C_{usc}$ to $D_{usc}$, or from $C_{lsc}$ to $D_{lsc}$, the following statements are equivalent:
\begin{enumerate}[label={\normalfont(\arabic*)}]
\item $\Psi$ is gauge-preserving;
\item $\Psi$ is homogeneous of degree $1$ and an order-isomorphism.
\end{enumerate}
\end{lemma}
\begin{proof}
The proof, being analogous to \Cref{gphop}, is omitted.
\end{proof}

Our aim now is to show that each gauge-reversing bijection $\Phi\colon C \to D$ extends to gauge-reversing bijections $\Phi_{usc}\colon C_{usc} \to D_{lsc}$ and $\Phi_{lsc} \colon C_{lsc} \to D_{usc}$. Observe that by criterion (2) of \Cref{graorc}, these extensions are to be order-anti-isomorphisms, so are required to interchange suprema and infima. Because  the members of the extended cones are limits of monotone nets in the original cones by \Cref{scapprox}, one arrives at the formulae \eqref{PhiuscDef} and \eqref{PhilscDef} for $\Phi_{usc}$ resp. $\Phi_{lsc}$ in the following theorem.

\begin{theorem}\label{ghProp}
Let $V$ and $W$ be complete order unit spaces, and let $\Phi\colon C\to D$ be a gauge-reversing bijection between their open cones $C = V^\circ_+$ and $D = W^\circ_+$. Then $\Phi$ has a unique extension to a gauge-reversing bijection
$$\Phi_{usc}\colon  C_{usc} \to D_{lsc}$$
and to a gauge-reversing bijection
$$\Phi_{lsc}\colon C_{lsc} \to D_{usc}.$$
\end{theorem}
\begin{proof}
We recall \Cref{graor} that a map between (extended) cones is a gauge-reversing bijection if and only if it is homogeneous of degree $-1$ and an order-anti-isomorphism, i.e.\ an order-reversing bijection with an order-reversing inverse.\\
\ind We define the extension $\Phi_{usc}: C_{usc} \to D_{lsc}$ as follows. Let $g \in C_{usc}$. By \Cref{scapprox}(1) the set $\{x \in C: x \succ g\}$ defines a nonincreasing net in $C$ with infimum $g$. It follows that $\{\Phi(x): x \succ g\}$ is a nondecreasing net in $D$, which has a supremum in $D_{lsc}$ by \Cref{SupInfProp}. We define
\begin{equation}\label{PhiuscDef}
\Phi_{usc}(g) := \sup \{\Phi(x): x \succ g\} \in D_{lsc}.  
\end{equation}
\ind Let us check that if $g \in C$ then $\Phi_{usc}(g) = \Phi(g)$. Since $\Phi$ is order-reversing, it is clear that $\Phi_{usc}(g) \le \Phi(g)$. On the other hand, for every $\lambda > 1$ we have $\lambda g \succ g$ and $\Phi(\lambda g) = \lambda^{-1}\Phi(g)$, so \eqref{PhiuscDef} gives $\Phi_{usc}(g) \ge \sup \{\lambda^{-1} \Phi(g): \lambda > 1\} = \Phi(g)$. We conclude that $\Phi_{usc}(g) = \Phi(g)$.\\ 
%belongs to the set $\{\Phi(x): x \succ g\}$ of which $\Phi_{usc}(g)$ is the supremum by definition. It follows that %Thus we may drop the hat, and will write simply $\Phi(g)$ for $\Phi_{usc}(g)$. \\
\ind Let us check that $\Phi_{usc}$ is order-reversing. Suppose that $g,g' \in C_{usc}$ satisfy $g\le g'$. Then $\{x \in C: x \succ g'\} \subseteq \{x\in C: x \succ g\}$, hence $$\Phi_{usc}(g') = \sup \{\Phi(x): x \in C, x \succ g'\} \le \sup \{\Phi(x): x \in C, x \succ g\} = \Phi_{usc}(g).$$
\ind Let us check that $\Phi$ is homogeneous of degree $-1$. For every $\lambda \in (0, \infty)$ and $g \in C_{usc}$ we have $\{x \in C: x \succ \lambda g\} = \{\lambda x: x \in C, x \succ g\}$, so that
\begin{align*}
\Phi_{usc}(\lambda g) &= \sup\{\Phi(\lambda x): x \in C, x \succ g\} = \sup\{\lambda^{-1}\Phi(x): x \in C, x \succ g\}\\
&= \lambda^{-1}\sup\{\Phi(x): x \in C, x \succ g\}=\lambda^{-1}\Phi_{usc}(g).
\end{align*}

It remains to check that $\Phi_{usc}(0) = \infty$. For every $\epsilon > 0$ we have $0 \prec \epsilon v$, so that $\Phi(0) \ge \Phi(\epsilon v)=\epsilon^{-1}\Phi(v)$, from where $\Phi_{usc}(0) = \infty$.\\
\ind Let us now show that $\Phi_{usc}$ is order-continuous with respect to nonincreasing nets. Let $(g_{\alpha})_{\alpha}$ be a nonincreasing net in $C_{usc}$, and let $g\in C_{usc}$ be its infimum. Since $\Phi_{usc}$ is order-reversing, the net $(\Phi_{usc}(g_{\alpha}))_{\alpha}$ in $D_{lsc}$ is nondecreasing. By definition we have $\Phi_{usc}(g_{\alpha}) = \sup\{\Phi(x): x \in C, x \succ g_{\alpha} \}$, and $\Phi_{usc}(g) = \sup \{\Phi(x): x \in C, x \succ g\}$. As a consequence of \Cref{ocga} we obtain that $\{x \in C: x \succ g\} = \bigcup_{\alpha} \{x\in C: x \succ g_{\alpha}\}$. Now take the supremum to see that $\Phi_{usc}(g) = \sup_{\alpha} \Phi(g_{\alpha})$, as desired.\\
\ind Similarly, using \Cref{scapprox}(2) that $\{x \in C: x \prec g\}$ is a nondecreasing net with supremum $g \in C_{lsc}$, we define 
\begin{equation}\label{PhilscDef}
\Phi_{lsc}(g) = \inf \{\Phi(x): x \in C, x \prec g\}.
\end{equation}
The verification that $\Phi_{lsc}$ is order-reversing, homogeneous of degree $-1$, and order-continuous with respect to nondecreasing nets goes along the same line.\\
\ind Let us now show that $\Phi_{usc}\colon C_{usc} \to D_{lsc}$ is a bijection with inverse $(\Phi^{-1})_{lsc}\colon D_{lsc} \to C_{usc}$. Let $g \in C_{usc}$, and choose some nonincreasing net $(g_{\alpha})$ in $C$ with infimum $g$. Because $\Phi_{usc}$ is order-continuous with respect to nonincreasing nets, $\Phi(g_{\alpha})$ is a nondecreasing net in $D$ with supremum $\Phi_{usc}(g)$ in $D_{usc}$. Similarly, using that $(\Phi^{-1})_{lsc}$ is order-continuous with respect to nondecreasing nets, we obtain that $(\Phi^{-1})_{lsc}(\Phi_{usc}(g)) = \inf_{\alpha} \Phi^{-1}(\Phi(g_\alpha)) = \inf_{\alpha} g_{\alpha} = g$. This shows that $(\Phi^{-1})_{lsc} \circ \Phi_{usc} = \Id_{C_{usc}}$; the proof that $\Phi_{usc} \circ (\Phi^{-1})_{lsc} = \Id_{D_{lsc}}$ is similar. Since $\Phi_{usc}$ and $(\Phi^{-1})_{lsc}$ are mutually inverse order-reversing bijections, they are order-anti-isomorphisms, and also homogeneous of degree $-1$. We conclude by \Cref{graor} that $\Phi_{usc}$ and $(\Phi^{-1})_{lsc}$ are gauge-reversing bijections. Considering the gauge-reversing map $\Phi^{-1}$ instead, we obtain that $\Phi_{lsc}$ and $(\Phi^{-1})_{usc}$ are mutually inverse bijections.\\
\ind Since an order-anti-isomorphism sends the infimum of a set to the supremum of the image of that set, any gauge-reversing bijection $\Phi_{usc}\colon C_{usc} \to D_{lsc}$ extending $\Phi$ is necessarily given by \eqref{PhiuscDef}. This shows the asserted uniqueness of $\Phi_{usc}$. Similarly, the gauge-reversing bijective extension $\Phi_{lsc}$ of $\Phi$ is uniquely determined by \eqref{PhilscDef}.
% Show that $g(K) \subset (0, \infty)$ if and only if $\Phi(g)(S) \subset (0, \infty)$??
%\ind Since $\Phi_{usc}$ is anti-homogeneous and order-reversing, as in \Cref{ghProp} one readily proves that $M(g,g')=M(\Phi_{usc}(g')/\Phi_{usc}(g))$ and $m(g',g)=L(\Phi_{usc}(g')/\Phi_{usc}(g))$ for all $g,g' \in C_{usc}$. Similar assertions hold for $\Phi_{lsc}$.
\end{proof}

Now suppose $\Psi\colon C \to D$ that is a gauge-preserving bijection. By the same method as in \Cref{ghProp}, one constructs extensions $\Psi_{usc}\colon C_{usc} \to D_{usc}$ and $\Psi_{lsc}\colon C_{lsc} \to D_{lsc}$. 

\begin{proposition}\label{ggProp}
Let $V$ and $W$ be complete order unit spaces, and let $\Psi\colon C\to D$ be a gauge-preversing bijection between their open cones $C = V^\circ_+$ and $D = W^\circ_+$. Then $\Psi$ has a unique extension to a gauge-preversing bijection
\begin{equation*}
\Psi_{usc}\colon C_{usc} \to D_{usc}\quad\text{ and }
\end{equation*}
and to a gauge-preserving bijection
\begin{equation*}
\Psi_{lsc}\colon C_{lsc} \to D_{lsc}.
\end{equation*}
\end{proposition}
\begin{proof}
Modify the proof of \Cref{ghProp} appropriately.
\end{proof}

\subsection{Strong atomicity}\label{evsubSec}

\begin{definition}\label{extD}
Let $(X, X_+)$ be an ordered vector space, and let $E \subset X_+$ be a subcone containing $0$. An \emph{extremal vector} of $E$ is a nonzero vector $p \in E$ such that $\{x \in E: x \le p\} = \{tp: 0 \le t \le 1\}$. The set of extremal vectors of $E$ is denoted $\ext E$. We say that the cone $E$ is \emph{strongly atomic} if for each $x\in E$ one has
$$x = \sup \{p \in \ext E \colon p \le x\}\text{ in }E.$$
\end{definition}

We continue with our assumption that $(V, V_+, v)$ and $(W, W_+, w)$ be complete order unit spaces, with state spaces $K$ resp. $S$, and open cones $C = V^\circ_+$ and $D = W^\circ_+$. The set of extreme vectors $p$ in $C_{usc}$ normalized so that $M(p,v) = 1$, or equivalently so that $m(v,p) = 1$, is denoted by
\begin{equation*}
\ext_v C_{usc} := \{p \in \ext C_{usc}: M(p,v) = 1\}.
\end{equation*}

\begin{lemma}\label{extEq}
Let $(V, V_+, v)$ be an order unit space. Then the following are equivalent:
\begin{enumerate}[label=\normalfont(\arabic*)]
\item the cone $C_{usc}$ is strongly atomic;
\item for each $g \in C_{usc}$ we have
\begin{equation}
g = \sup\{m(g,p)p: p \in \ext_v C_{usc}\}
\end{equation}
\item for all $g, g' \in C_{usc}$ we have
\begin{equation}
m(g,p) \le m(g',p)\text{ for each }p \in \ext_v C_{usc} \iff g \le g'.
\end{equation}
\end{enumerate}
\end{lemma}
\begin{proof}
As to (1) $\iff$ (2), note that $\{r \in \ext C_{usc}: r \le g\} = \{m(g,p)p: p \in \ext_v C_{usc}\}.$ The equivalence (2) $\iff$ (3) holds because for each $p \in \ext_v C_{usc}$ we have that $m(g',p) \ge m(g,p)$ if and only if $m(g,p)p \le g'$.
\end{proof}

Recall that we defined $\delta_e S$ as the set of pure states in $S$, and for each $\psi \in \delta_e S$ the function $I^{\infty}_{\{\psi\}}$ by $I^{\infty}_{\{\psi\}}(\psi)=1$ and $I^{\infty}_{\{\psi\}}(\rho) = \infty$ if $\rho \neq \psi$, see \Cref{JF}.

\begin{proposition}\label{extdeBij} Let $\Phi\colon C \to D$ be a gauge-reversing bijection such that $\Phi(v) = w$. Then we have a bijection
\begin{align}\label{extBij}
\ext_v C_{usc} &\to \delta_e S,\\
p &\mapsto \psi,
\end{align}
where $\psi$ is determined by $I^{\infty}_{\{\psi\}} = \Phi_{usc}(p)$. Moreover, for all $g \in C_{usc}$ we have
\begin{equation}\label{Mppsi}
M(p,g) = \Phi_{usc}(g)(\psi).
\end{equation}
\end{proposition}
\begin{proof}
Let $p \in C_{usc}$ and $f \in D_{lsc}$ be such that $f = \Phi_{usc}(p)$, hence $p = (\Phi^{-1})_{lsc}(f)$. We need only show that $p \in \ext_v C_{usc}$ if and only if $f = I^{\infty}_{\{\psi\}}$ for some $\psi \in \delta_e S$. Since the bijection $\Phi_{usc}$ is an order-anti-isomorphism and homogeneous of degree $-1$, it restricts to bijections
% https://q.uiver.app/#q=WzAsNSxbMCwwLCJcXHt0cHwgMCA8IHQgXFxsZSAxXFx9Il0sWzAsMSwiXFx7ZyBcXGluIENfe3VzY30gfCAwIFxcbGUgZyBcXGxlIHBcXH0iXSxbMiwwXSxbMSwwLCJcXHt0ZnwgMSBcXGxlIHQgPCBcXGluZnR5XFx9Il0sWzEsMSwiXFx7aCBcXGluIERfe2xzY30gfCBmIFxcbGUgaCBcXGxlIFxcaW5mdHlcXH0iXSxbMCwzLCJcXHNpbSJdLFsxLDQsIlxcc2ltIl0sWzAsMSwiXFxzdWJzZXQiLDFdLFszLDQsIlxcc3Vic2V0IiwxXV0=
\[\begin{tikzcd}
	{\{\lambda p \colon 0 \le \lambda \le 1\}} & {\{\lambda f \colon 1 \le \lambda \le \infty\}} & {} \\
	{\{g \in C_{usc} \colon g \le p\}} & {\{h \in D_{lsc} \colon h \ge f\}}
	\arrow["\sim", from=1-1, to=1-2]
	\arrow[phantom, sloped, "\subseteq", from=1-1, to=2-1]
	\arrow[phantom, sloped, "\subseteq", from=1-2, to=2-2]
	\arrow["\sim", from=2-1, to=2-2]
\end{tikzcd}\]
Moreover, we have $m(v,p) = m(f,w) = \inf f$, so we may assume that $m(v,p) = 1 = \inf f$. Then $p \in \ext C_{usc}$ if and only if the left inclusion is an equality, which happens only if the right inclusion is an equality. We will finish the proof by showing that we have equality on the right if and only if $f = I^{\infty}_{\{\psi\}}$ for some $\psi \in \delta_e S$.\\
\ind The if direction being clear, suppose we have equality on the right. By \Cref{scface} there exists $\psi \in \delta_e S$ with $f(\psi) = \inf f = 1$. Then we have $I^{\infty}_{\psi} \ge f$, hence there exists $\lambda \in [1, \infty]$ with $I^{\infty}_{\psi} = \lambda f$. Then $\lambda = \inf(\lambda f) = \inf I^{\infty}_{\psi} = 1$, so $f = I^{\infty}_{\psi}$ as desired.\\
\ind Finally, we turn to proving \eqref{Mppsi}. Note that for each $h \in D_{lsc}$ and $\psi \in \delta_e S$ we have that $h(\psi) = \inf\{\lambda > 0: h \le \lambda I^{\infty}_{\{\psi\}}\} = M(h,I^{\infty}_{\{\psi\}})$. Therefore if $p \in \ext_v C_{usc}$ and $g \in C_{usc}$ are such that $\Phi_{usc}(p)=I^{\infty}_{\{\psi\}}$ and $\Psi_{usc}(g) = h$, then we have
$$\Phi_{usc}(g)(\psi) = h(\psi) = M(h,I^{\infty}_{\{\psi\}}) = M(\Phi_{usc}(g),\Phi_{usc}(p))=M(p,g),$$
where in the last step we used that $\Phi_{usc}$ is gauge-reversing.
\end{proof}

\begin{theorem}\label{Phiconv} Let $\Phi\colon C \to D$ be a gauge-reversing bijection. Then $\Phi$ is convex, i.e.\ for all $x_1,x_2 \in C$ and $\lambda \in [0,1]$ one has
\begin{equation}\label{Convex}
\Phi((1-\lambda)x_1+\lambda x_2) \le (1-\lambda) \Phi(x_1) + \lambda \Phi(x_2).
\end{equation}
\end{theorem}
\begin{proof}
Let $\psi \in \delta_e S$ be a pure state of $W$. We will show that the function $\psi \circ \Phi\colon C \to (0, \infty)$ is convex. \Cref{extdeBij} grants us an extremal vector $p \in \ext_v(C_{usc})$ such that $\psi(\Phi(z)) = M(p,z) = m(z,p)^{-1}$ for every $z \in C$. It suffices to show that the map $m(\cdot,p)\colon C \to (0, \infty)$ is concave, because the inverse of a strictly positive concave function is convex.\\
\ind We claim that $m(\cdot,p)$ is homogeneous of degree $1$ and superadditive. Plainly, for all $\lambda > 0$ and $z\in C$ we have $m(\lambda z,p)=\lambda m(z,p)$. Let now $z_1, z_2 \in C$. Then $m(z_1,p)p \le z_1$ and $m(z_2,p)p \le z_2$ hold, so $(m(z_1,p)+m(z_2,p))p \le z_1+z_2$, whence $m(z_1,p)+m(z_2,p)\le m(z_1+z_2,p)$. This proves our claim. It follows that the function $m(\cdot,p)$ is concave, since
\begin{equation*}
m((1-\lambda)x_1+\lambda x_2,p) \ge m(1-\lambda)x_1,p)+m(\lambda x_2,p) =
(1-\lambda) m(x_1,p)+\lambda m(x_2,p).
\end{equation*}
%Since the function $(0, \infty) \mapsto (0, \infty)$, $t \mapsto t^{-1}$ is decreasing and convex, it follows that
%\begin{equation*}
%L(p/(1-\lambda)x_1+\lambda x_2)^{-1} \le ((1-\lambda) m(x_1,p) + \lambda m(x_2,p))^{-1} \le (1-\lambda)m(x_1,p)^{-1}+m(x_2,p)^{-1}.
%\end{equation*}
%Since $\psi \circ \Phi=M(p,\cdot)=m(\cdot,p)^{-1}$, this shows that $\psi \circ \Phi$ is convex. Since $\psi \in \delta_e S$ was arbitrary, it follows that $\Phi$ is convex.
Consequently, $\psi \circ \Phi = m(\cdot,p)^{-1}$ is convex. Since $\psi \in \delta_e S$ was arbitrary, we conclude that $\Phi$ is convex by \Cref{pureIneq}.
\end{proof}

\begin{theorem}\label{supAtom} Let $(V, V_+, v)$ and $(W, W_+, w)$ be complete order unit spaces such that there exists a gauge-reversing bijection $\Phi\colon C \to D$. Then the cone $C_{usc}$ is strongly atomic.
\end{theorem}
\begin{proof}
Let us show that condition (3) of \Cref{extEq} is satisfied. Let $g, g' \in C_{usc}$ be such that $m(g,p) \le m(g',p)$, equivalently $M(p,g) \ge M(p,g')$, for each $p \in \ext_v C_{usc}$. Set $h = \Phi_{usc}(g)$ and $h' = \Phi_{usc}(g')$. Let $\psi \in \delta_e S$ be arbitrary. By \Cref{extdeBij} there exists a unique $p \in \ext_v C_{usc}$ such that $\Phi_{usc}(p) = I^{\infty}_{\{\psi\}}$. Moreover, according to (\ref{Mppsi}) we have that $M(p,g)=h(\psi)$ and $M(p,g')=h'(\psi)$.
It follows that $h(\psi) = M(p,g) \ge M(p,g') = h'(\psi).$ Now \Cref{pureIneq} implies that $h \ge h'$. Since $\Phi_{usc}$ is an order-anti-isomorphism, we conclude that $g \le g'$. 
\end{proof}

\begin{corollary}\label{myz}
Let $y,z \in C_{usc}$. Then $$\lVert y-z\rVert_v \le \sup \{|m(y,p)-m(z,p)|: p \in \ext_v C_{usc}\}.$$
\end{corollary}
\begin{proof}
Denote the supremum on the right by $\nu$. We claim that $y \le z + \nu v$. Indeed, for each $p \in \ext_v C_{usc}$ we have $m(z+\nu v,p) \ge m(z,p) + m(\nu v,p)=m(z,p)+\nu \ge m(y,p)$, so by \Cref{supAtom} it follows that $z + \nu v \ge y$. Now $y + \nu v \ge z$ follows by symmetry, whence $-\nu v \le y - z \le \nu v$. This shows that $\lVert y-z\rVert_v \le \nu$, as desired.
\end{proof}

\section{Consequences of strong atomicity}\label{evgdSec}
In \Cref{supAtom} of the previous section we have established that the extended cone $C_{usc}$ is strongly atomic. Each of the three subsections of this section derives a distinct consequence of strong atomicity of $C_{usc}$. In the first subsection we prove that the gauge-reversing bijection $\Phi\colon C \to D$ is Gateaux-differentiable in positive directions, and that the negative of the Gateaux derivative at a point $x \in C$ extends to an isomorphism $(V, V_+, x) \cong (W, W_+, \Phi(x))$. In the second subsection, we give an alternative proof of \Cref{NSThmB}, stating that a gauge-preserving bijection $\Psi\colon C\to D$ extends to a linear isomorphism between $V$ and $W$, under the additional assumption of strong atomicity of $C_{usc}$. Finally, in the third subsection, we prove a strong form of Hua's identity, which provides an explicit formula for $D\Phi$ in terms of $\Phi$.

\subsection{Gateaux-differentiability}\label{GdsubSec}
In this section we will show that the gauge-reversing bijection $\Phi_{usc}\colon C_{usc} \to D_{lsc}$ extending $\Phi$ has a (one-sided) Gateaux derivative at every point $x \in C$ in each direction $y \in C_{usc}$ (see \Cref{diffSec} for a discussion of Gateaux-differentiability). This Gateaux-derivative, denoted $D\Phi(x)(y)$, belongs to $-D_{usc}$ since $\Phi_{usc}$ is order-reversing. We show that the map $-D\Phi(x): C_{usc} \to D_{usc}$ extends to an isomorphism $(V, V_+, x) \cong (W, W_+, \Phi(x))$, culminating in \Cref{DPhi}.\\
\ind We will first prove Gateaux-differentiability in the direction of an extremal vector of $C_{usc}$, by adopting the arguments \cite[Lemmata 2.1 and 6.1]{LRW25} to the context of extended cones.

\begin{lemma}\label{x<ye}
Let $x \in C$ and $y \in C_{usc}$ with $y > x$. Then the following are equivalent:
\begin{enumerate}[label={\normalfont(\arabic*)}]
\item we have that $y-x \in C_{usc}$ is an extremal vector of $C_{usc}$;
\item we have $\{z \in C_{usc}: x \le z \le y\} = \{(1-t)x + ty: 0 \le t \le 1\}$;
\item the set $\{z \in C_{usc}: x \le z \le y\}$ is totally ordered.
\end{enumerate}
\end{lemma}
\begin{proof}
The validity of each of the $3$ statements is unaltered if we subtract $x$ from $x$ and $y$, so we may and do assume that $x = 0$. Then (1) $\iff$ (2) holds by \Cref{extD}, while (2) $\implies$ (3) is plain from $y \ge 0$. As to (3) $\implies$ (2), note that the containment $\supseteq$ in (2) always holds.\\
\ind Now assume $\{z \in C_{usc}: 0 \le z \le y\}$ is totally ordered, and let $z$ be one of its members. We finish the proof by showing that $z = \lambda y$ for some $\lambda \in [0,1]$.
For every $\mu\in [0,1]$ either $\mu y\le z$ or $\mu y \ge z$. Let $\lambda = \sup \{\mu \in [0,1]: \mu y \le z\} = \inf \{\mu \in [0,1]: \mu y \ge z\}$. For each $\epsilon > 0$ we have $(\lambda - \epsilon)y \le z \le (\lambda + \epsilon)y$, or $ -\epsilon y \le z - \lambda y \le \epsilon y$, and since $(V, V_+)$ is Archimedean this implies $z - \lambda y = 0$. We conclude that $z \in \{\lambda y: \lambda \in [0, 1]\}$ as desired.
%For each $\epsilon > 0$ we have $(\lambda - \epsilon)y \le u \le (\lambda + \epsilon)y$, so $-\epsilon M(y,v)v \le -\epsilon y \le u - \lambda y \le \epsilon y \le \epsilon M(y,v)v$. This shows that $\lVert u-\lambda y\rVert_v \le \epsilon M(y,v)$ for each $\epsilon > 0$, hence $\lVert u-\lambda y\rVert_v = 0$. We conclude that $u = \lambda y$, as desired. 
\end{proof}

\begin{lemma}\label{y<xe}
Let $x \in C$ and $y \in C_{lsc}$ with $y < x$. Then the following are equivalent:
\begin{enumerate}[label=\normalfont{(\arabic*)}]
\item we have that $x-y \in C_{usc}$ is an extremal vector of $C_{usc}$;
\item we have $\{z \in C_{lsc}: y \le z \le x\} = \{(1-t)x + ty: 0 \le t \le 1\}$;
\item the set $\{z \in C_{lsc}: y \le z \le x\}$ is totally ordered.
\end{enumerate}
\end{lemma}
\begin{proof}
This is deduced from \Cref{x<ye} using the order-anti-isomorphism $z \mapsto x-z$ between $\{z \in C_{lsc}: y \le z \le x\}$ and $\{z \in C_{usc}: 0 \le z \le x-y\}$. 
\end{proof}

Since condition (3) in \Cref{x<ye} and \Cref{y<xe} is purely order-theoretic, it follows that a gauge-reversing bijection sends extremal rays to extremal rays. A calculation of gauge-functions then yields the following more precise results.

\begin{theorem}\label{lpqLem}
Let $\Phi\colon C \to D$ be a gauge-reversing bijection between the open cones of complete order unit spaces.
\begin{enumerate}[label=\normalfont{(\arabic*)}]
\item Let $x \in C$ and $p \in \ext(C_{usc})$ with $m(x,p) = 1$. Then there exists $q \in \ext(D_{usc})$ with $m(\Phi(x),q)=1$ such that for all $\lambda \in (0, \infty)$ one has
\begin{equation}\label{Phiuscp}
    \Phi_{usc}(x + \lambda p) = \Phi(x) - \frac{\lambda}{\lambda + 1}q \in D_{lsc}.
\end{equation}
\item Let $z \in D$ and $q \in \ext(D_{usc})$ with $m(z,q) = 1$. Then there exists $p \in \ext(C_{usc})$ with $m(\Phi^{-1}(z),p)=1$ such that for all $\lambda \in (-1, 0)$ one has
\begin{equation}\label{Philscp}
    (\Phi^{-1})_{lsc}(z + \lambda q) = \Phi^{-1}(z) - \frac{\lambda}{\lambda + 1}p \in C_{usc}.
\end{equation}
\end{enumerate}
\end{theorem}
%(3) The maps $D\Phi_{usc}(x)\colon \ext_x C_{usc} \to \ext_{\Phi(x)} D_{usc}$ and $D(\Phi^{-1})_{lsc}(\Phi(x))\colon \ext_{\Phi(x)} D_{usc} \to \ext_x C_{usc}$ are mutually inverse bijections.
\begin{proof}
Fix some large $\kappa > 0$ for the moment. Since $\kappa p$ is an extremal vector, the order interval $\{y \in C_{usc}: x \le y \le x + \kappa p\}$ is totally ordered by \Cref{x<ye}. Since $\Phi_{usc}$ is an order-anti-isomorphism, it follows that
$$\Phi_{usc}(\{y \in C_{usc}: x \le y \le x + \kappa p\}) = \{z \in D_{lsc}: \Phi_{usc}(x+\kappa p) \le z \le \Phi(x)\}$$
is totally ordered as well. This means that condition (3) of \Cref{y<xe} is satisfied, and so the other two conditions also hold. First, conditions (1) gives that $\Phi(x) - \Phi_{usc}(x+\kappa p)$ is an extremal vector of $D_{usc}$, hence a multiple of an extremal vector $q$ with $M(q,\Phi(x)) = 1$. Secondly, condition (2) implies that for each $0 \le \lambda \le \kappa$ there exists $t(\lambda) \le 0$ with $\Phi(x + \lambda p) = \Phi(x) + t(\lambda)q$. To calculate $t(\lambda)$, use equations \eqref{Maxby1} and \eqref{Maxby3} as well as the fact that $m(x,p)=1=m(\Phi(x),q)$ to see that 
$$\frac{1}{1 + \lambda} = m(x,x + \lambda p) = m(\Phi(x)+ t(\lambda)q,\Phi(x)) = 1 + t(\lambda).$$
It follows that $t(\lambda) = 1/(1+t) - 1 = -t/(t+1)$. Since increasing $\kappa$ does not change $q$, we conclude that \eqref{Phiuscp} holds for all $\lambda \ge 0$. Part (2) is shown in a similar fashion.
\end{proof}

%Same lemma for \Phi_{lsc}, show bijection between extremal vectors

\begin{proposition}\label{DPhiusc}
Let $x \in C$. There exists a homogeneous of degree $1$, order-reversing map
$$D\Phi_{usc}(x)\colon C_{usc} \to -D_{usc},$$
such that $D\Phi_{usc}(x)(y)$ is the Gateaux-derivative of $\Phi_{usc}$ at $x$ in the direction $y$ for each $y \in C_{usc}$. Moreover, $D\Phi_{usc}(x)$ maps $V_+$ into $-W_+$, sends $x$ to $-\Phi(x)$, and maps $C$ into $-D$. 
\end{proposition}
\begin{proof}
There is no loss of generality in assuming $x = v$ and $\Phi(x) = w$, cf.  \Cref{Cndv}. Fix $y \in C_{usc}$, and define $f\colon (0, \infty) \to D_{usc}$ by the difference quotient
\begin{equation*}
f(\mu) := \mu^{-1}(w - \Phi_{usc}(v+\mu y)).
\end{equation*}
Note here that $v+\mu y \ge v$ implies $w \ge \Phi_{usc}(v + \mu y)$, so that $\mu^{-1}(w - \Phi_{usc}(v+\mu y)) \in D_{usc}$. We will show there exists $f(0) \in D_{usc}$ such that
\begin{equation}\label{fmu0}
\lim_{\mu \downarrow 0} \lVert f(\mu) - f(0)\rVert_w = 0.
\end{equation}
Let $q \in \ext_w D_{usc}$ and use \Cref{lpqLem}(2) to obtain $p \in \ext_v C_{usc}$ such that \eqref{Philscp} holds with $z = w$ and $\Phi^{-1}(z) = v$. Let $\mu, \nu \in (0, \infty)$. We claim that
\begin{equation}\label{MF}
m(f(\mu),q) = \frac{m(y,p)}{1+\mu m(y,p)}.
\end{equation}
For $\lambda \ge 0$ the following are equivalent:
\begin{align*}
\lambda \le \mu m(y,p) \iff \lambda p \le \mu y &\iff v + \lambda p \le v + \mu y\\
&\iff \Phi(v + \mu y) \le \Phi(v + \lambda p) = w - \frac{\lambda}{1 + \lambda}q \\&\iff
\frac{\lambda}{1 + \lambda}q \le w - \Phi(v + \mu y).
\end{align*}
It follows that
$$m(\mu^{-1}(w-\Phi(v+\mu y)),q)=\mu^{-1} \cdot \frac{\mu m(y,p)}{1 + \mu m(y,p)} = \frac{m(y,p)}{1 + \mu m(y,p)},$$
whence our claim follows. A little calculation shows that
\begin{equation}\label{mfb}
m(f(\mu),q)-m(f(\nu),q) =
\frac{m(y,p)}{1+\mu m(y,p)} - \frac{m(y,p)}{1+\nu m(y,p)}
=
\frac{(\nu - \mu)m(y,p)^2}{(1 + \mu m(y,p))(1+ \nu m(y,p))}.
\end{equation}
Since $0 \le m(y,p) \le \lVert y\rVert_v$, we conclude that
$$\sup\{|m(f(\mu),q)-m(f(\nu),q)|: q \in \ext_w D_{usc}\} \le \lVert y\rVert_v^2|\mu - \nu|.$$
By \Cref{myz} it follows that
\begin{equation}\label{fmunu}
\lVert f(\mu) - f(\nu) \rVert_w \le \lVert y\rVert_v^2|\mu - \nu|.
\end{equation}
Since each $\lVert\cdot\rVert_w$-Cauchy net in $D_{usc}$ converges in $D_{usc}$ by \Cref{sclim}(1), it follows that there exists $f(0) \in D_{usc}$ satisfying \eqref{fmu0}. If $y \in V_+$, then $f(\mu) \in W_+$ for all $\mu > 0$ so $f(0) \in W_+$ because $(W_+, \lVert\cdot\rVert_w)$ is complete. Setting $D\Phi_{usc}(v)(y) = -f(0) \in -D_{usc}$ and letting $\nu \downarrow 0$ in \eqref{fmunu} gives
\begin{equation}\label{fmu2}
\lVert \mu^{-1}(\Phi_{usc}(v+\mu y) - \Phi(v)) - D\Phi_{usc}(v)(y)\rVert_w = 
\lVert f(\mu) - f(0)\rVert_w \le \mu \lVert y\rVert_v^2.
\end{equation}
We conclude that $D\Phi_{usc}(v)(y)$ is the Gateaux-derivative of $\Phi_{usc}$ at $v$ in the direction $y$.\\
\ind Now let $x\in C$ be arbitrary. By a general principle of Gateaux-derivatives, $D\Phi_{usc}(x)$ is homogeneous of degree $1$. It reverses the order since $\Phi_{usc}$ does. Since $\Phi$ is homogeneous of degree $-1$, one easily sees that $D\Phi_{usc}(x)(x) = -\Phi(x)$. This implies by homogeneity that $D\Phi_{usc}(x)(\delta x)=-\delta \Phi(x)$ for each $\delta > 0$. Because we have shown that $D\Phi_{usc}(x)$ maps $V_+$ into $-W_+$ and is order-reversing, it follows that $C = \bigcup_{\delta > 0} \{y \in V_+: y \ge \delta x\}$ is mapped into $-D = \bigcup_{\delta > 0} \{z \in -W_+: z \le -\delta \Phi(x)\}$ by $D\Phi_{usc}(x)$.
\end{proof}

\begin{proposition}\label{DPhilsc}
Let $x \in C$. There exists a homogeneous order-reversing map
$$D\Phi_{lsc}(x)\colon -C_{usc} \to D_{usc},$$
such that $D\Phi_{lsc}(x)(-y)$ is the Gateaux-derivative of $\Phi_{lsc}$ at $x$ in the direction $-y$ for each $y \in C_{usc}$. Moreover, $D\Phi_{lsc}(x)$ maps $-V_+$ inside $W_+$, sends $-x$ to $\Phi(x)$, and maps $-C$ into $D$. 
\end{proposition}
\begin{proof}
The proof being very similar to that of \Cref{DPhiusc}, we only indicate which adaptions ought to be made. Again we may assume that $x = v$ and $\Phi(x) = w$.
Let $y \in C_{usc}$ and define $f\colon (-\lVert y\rVert_v, 0) \to D_{usc}$ by
$$f(\mu) = \mu^{-1}(w - \Phi_{lsc}(v + \mu y)).$$
Now one establishes for $-\max\{1,\lVert y\rVert_v\}^{-1} < \mu < \nu < 0$ that \eqref{MF} and \eqref{mfb} hold, and deduces from \Cref{myz} that
$$\lVert f(\nu) - f(\mu) \rVert_w \le \frac{(\nu - \mu)\lVert y\rVert_v^2}{(1+\mu \lVert y\rVert_v)^2}.$$
The proof is concluded in the same way as the previous proposition.
\end{proof}

\begin{lemma}\label{DPhisc}
The maps $D\Phi_{usc}(x)\colon C_{usc}  \to -D_{usc}$ and $D(\Phi^{-1})_{lsc}(\Phi(x))\colon -D_{usc} \to C_{usc}$ are mutually inverse for every $x\in C$.
\end{lemma}
\begin{proof}
The maps $\Phi_{usc}\colon C_{usc} \to D_{lsc}$ and $(\Phi^{-1})_{lsc}\colon D_{lsc} \to C_{usc}$ are mutually inverse by \Cref{ghProp}. As we have seen in \Cref{biLip}, both maps are locally Lipschitz continuous, so the chain rule for Gateaux-derivatives, \Cref{chr}, holds:
$$D(\Phi^{-1})_{lsc}(\Phi(x)) \circ D\Phi_{usc}(x) = \id_{C_{usc}} \text{ and } D\Phi_{usc}(x) \circ D(\Phi^{-1})_{lsc}(\Phi(x)) = \id_{-D_{usc}}.$$

\end{proof}
%XXX: remove forward reference

\begin{theorem}\label{DPhi}
For each $x\in C$ there exists a unique isomorphism of order unit spaces 
\begin{equation}\label{-DPhiEq}
-D\Phi(x)\colon (V, V_+, x) \to (W, W_+, \Phi(x))
\end{equation}
such that for all $y \in V_+$ one has
\begin{equation}\label{GPhi}
\lVert \Phi(x+y) -  \Phi(x) - D\Phi(x)(y)\rVert_{\Phi(x)} \le \lVert y\rVert^2_x.
\end{equation}
\end{theorem}
\begin{proof}
Propositions \ref{DPhiusc} and \ref{DPhilsc} and \Cref{DPhisc} 
yield that the map $-D\Phi_{usc}(x)\colon C_{usc} \to D_{usc}$ is an order-isomorphism and homogeneous of degree $1$, hence a gauge-preserving bijection by \Cref{graor}. By the same propositions, $-D\Phi_{usc}(x)$ restricts to a gauge-preserving bijection from $C$ into $D$, which we denote by $-D\Phi(x)\colon C \to D$. By \Cref{NSThmB} the map $-D\Phi(x)\colon C \to D$ extends to a linear isomorphism also denoted by $-D\Phi(x)\colon V \to W.$ Since $D\Phi(x)(x) = -\Phi(x)$ holds by \Cref{DPhiusc}, we have that $-D\Phi(x)\colon (V, V_+, x) \to (W, W_+, \Phi(x))$ is an isomorphism of order unit spaces.\\
\ind Taking $\mu = 1$ in (\ref{fmu2})  we obtain that for all $y \in V_+$ we have
$$\lVert \Phi(v + y) - \Phi(v) -  D\Phi(v)(y) \rVert_w \le \lVert y\rVert_v^2.$$
Since the order unit of $(V, V_+)$ may be chosen at will, the last formula holds with $v$ replaced by an arbitrary element $x$ of $C$ and $w$ replaced by $\Phi(x)$.
\end{proof}

%$$-D\Phi(v)(y) = f(0) = \lim_{\mu \downarrow 0} \mu^{-1}(w - \Phi(v + \mu y))  \in D_{usc}.$$
%$$-D\Phi_{usc}(x)(y) =\lim_{\mu \downarrow 0} \frac{1}{\mu}(\Phi_{usc} - \Phi_{usc}(x+\mu y)) \in D_{usc}.$$

\begin{remark}
\Cref{DPhi} yields a proof of \Cref{grevThm} in the case that both order unit spaces are complete. We will dispense with this completeness assumption in \Cref{ncmplSec}.
\end{remark}

\subsection{Noll-Sch\"afer theorem}\label{ThmBSec}

Let $(V, V_+, v)$ and $(W, W_+, w)$ be complete order unit spaces. By a theorem of Noll and Sch\"afer, \cite[Theorem~B]{NollSch77, Sch78}, if $\Psi\colon C \to D$ is a gauge-preserving bijection with $\Psi(v) = w$, then $\Psi$ extends to an isomorphism of order unit spaces $(V, V_+, v) \cong (W, W_+, w)$. To make our exposition more self-contained, we will give a new proof of this theorem.

\begin{theorem}\label{ThmBsa}
Let $(V, V_+, v)$ and $(W, W_+, w)$ be complete order unit spaces. Let $\Psi\colon C \to D$ be a gauge-preserving bijection with $\Psi(v) = w$. Then $\Psi$ extends to an isomorphism of order unit spaces $(V, V_+, v) \cong (W, W_+, w)$.
\end{theorem}
\begin{proof}
Let $\Psi_{usc}\colon C_{usc} \to D_{usc}$ be the gauge-preserving bijection extending $\Psi$, which has been constructed in \Cref{ggProp}. Denote the state spaces of $(V, V_+, v)$ and $(W, W_+, w)$ by $K$ and $S$, respectively. Let $f \in C_{lsc}$. We have seen in \Cref{extdeBij} that  $f = I^{\infty}_{\psi}$ for some pure state $\psi \in \partial_e K$ if and only if $M(v, f) = 1$ and $\{h \in C_{lsc}: h \ge f\} = \{\lambda f: \lambda \in [1, \infty]\}$. Similarly, $\Phi_{usc}(f)$ has the shape $I^{\infty}_{\phi}$ for a pure state $\phi \in \partial_e S$ if and only if $M(w, \Phi_{lsc}(f)) = 1$ and $\{k \in D_{lsc}: k \ge \Psi_{lsc}(f)\} = \{\mu \Psi_{lsc}(f): \mu \in [1, \infty]\}$. Since $\Phi_{lsc}$ is homogeneous of degree $1$ and an order-isomorphism, it follows that $\Phi_{lsc}$ restricts to a bijection between $\{I^{\infty}_\psi: \psi \in \delta_e K\}$ and $\{I^{\infty}_\phi: \phi \in \delta_e S\}$. Thus we obtain a bijection $\partial_e K \to \partial_e S$, $\psi \mapsto \phi$, where $\phi$ is determined by $\Psi_{lsc}(I^{\infty}_{\psi})=I^{\infty}_\phi$. Moreover, for $x\in C_{lsc}$ we have
$$\Phi_{lsc}(x)(\phi) = M(\Phi_{lsc}(x), I^{\infty}_\phi) = M(\Phi_{lsc}(x), \Phi_{lsc}(I^{\infty}_\psi)) = M(x, I^{\infty}_\psi) = x(\psi).$$
It follows that $\Psi(x+y) = \Psi(x)+\Psi(x)$ for all $x,y \in C$. Indeed, for $\phi \in \partial_e K$ and $\psi \in \partial_e S$ as above, we have
$$\Phi(x+y)(\phi)=(x+y)(\psi)=x(\psi)+y(\psi)=\Phi(x)(\phi)+\Phi(y)(\phi)=(\Phi(x)+\Phi(y))(\phi).$$
This equation holds for all $\phi \in \delta_e S$, whence $\Psi(x+y) = \Psi(x)+\Psi(x)$. We conclude that $\Psi$ is additive on $C$.\\
\ind Since $V = C - C$, we find that $\Psi\colon C \to D$ extends to a linear mapping from $V$ to $W$. Similarly, we obtain a linear map from $W$ to $V$ extending $\Psi^{-1}\colon D \to C$. These two linear maps are inverse to one other and preserve the open cones $C$ and $D$, hence the closed cones $V_+$ and $W_+$. Therefore, we have obtained an isomorphism of order unit spaces $(V, V_+, v) \cong (W, W_+, w)$ extending $\Psi$.
\end{proof}

\subsection{Hua's identity}\label{HuasubSec}

The principal result of this subsection is that every gauge-reversing bijection $\Phi$ automatically satisfies Hua’s identity. This identity will play a central role in all subsequent subsections. As a corollary we obtain a formula for the Gateaux-derivative $D\Phi$ of $\Phi$ solely in terms of $\Phi$ itself, which will be instrumental in establishing higher-order smoothness properties of $\Phi$.

\begin{theorem}[Hua's identity]\label{HuaPhi}
Let $\Phi\colon C \to D$ be a gauge-reversing bijection between the open cones of complete order unit spaces $(V, V_+, v)$ resp. $(W, W_+, w)$. Then for all $x,y \in C$ we have
\begin{equation}\label{Hua}
\Phi(x) - \Phi(x+y) = -D\Phi(x)(\Phi^{-1}(\Phi(x)+\Phi(y))).
\end{equation}
\end{theorem}
\begin{proof}
In view of \Cref{Cndv}, we may and do assume that $x = v$ and $\Phi(x) = w$, so our objective becomes to show the equation
$$w-\Phi(v+y) = -D\Phi(v)(\Phi^{-1}(w+\Phi(y))).$$
Both members reside in $D$, since $w-\Phi(v+y) \ge w-\Phi(v+m(y,v)v)=(1-(1+m(y,v))^{-1})w$ and because $-D\Phi(v): (V, V_+,v) \to (W, W_+, w)$ is an isomorphism of order unit spaces by \Cref{DPhi}.\\
\ind By \Cref{supAtom} it suffices to prove for each $q \in \ext_w D_{usc}$ that
\begin{equation}\label{LL}
m(w-\Phi(v+y),q) = m({-D\Phi(v)}(\Phi^{-1}(w+\Phi(y))),q).
\end{equation}
Let $p \in \ext_v C_{usc}$ and $q \in \ext_w D_{usc}$ be a pair of extremal vectors corresponding to one another as in \Cref{lpqLem}(1), i.e.\ such that for each $\lambda > 0$ we have 
\begin{equation}\label{Phivp}
\Phi(v + \lambda p) = w - \frac{\lambda}{\lambda+1}q. 
\end{equation}
\ind To calculate the left-hand side of \eqref{LL}, observe that for each $\lambda > 0$ it holds that
\begin{equation*}
\lambda p \le y \iff v + \lambda p \le v + y \iff w - \frac{\lambda}{\lambda + 1}q \ge \Phi(v+y) \iff \frac{\lambda}{\lambda + 1}q \le w - \Phi(v+y).
\end{equation*}
Because we defined $m(y,p) = \sup\{\lambda > 0: \lambda p \le y\}$, this shows that
$$m(w-\Phi(v+y),q)= \frac{m(y,p)}{m(y,p)+1}.$$
\ind We now turn to the right-hand side of \eqref{LL}. One sees from \eqref{Phivp} that the Gateaux-derivative of $\Phi_{usc}$ at $v$ in the direction of $p$ equals $D\Phi_{usc}(v)(p) = -q$. Since $-D\Phi_{usc}(v): C_{usc} \to D_{usc}$ is a gauge-preserving bijection by \Cref{DPhi}, for each $g \in C$ we have that
$$M(q,{-D\Phi}(v)(g))=M({-D\Phi_{usc}}(v)(p),{-D\Phi_{usc}}(v)(g))=M(p,g).$$
Moreover, \Cref{extdeBij} grants us a state $\psi$ on $(W,W_+,w)$ such that $M(p,g) = \Phi(g)(\psi)$ for each $g \in C$. An application of these results to $g = \Phi^{-1}(w+\Phi(y))$ yields
$$
M(q,{-D\Phi(v)}(\Phi^{-1}(w+\Phi(y))))=M(p,\Phi^{-1}(w+\Phi(y))) = w(\psi) + \Phi(y)(\psi)
= 1 + M(p,y).
$$
Because $L$ and $M$ are multiplicatively inverse, we have $m(y,p)=M(p,y)^{-1}$ and
$$m(w-\Phi(v+y)/q)M(q,{-D\Phi(v)}(\Phi^{-1}(w+\Phi(y)))) = \frac{m(y,p)}{m(y,p)+1}(1+M(p,y))=1.$$
For the same reason this equality is equivalent to \eqref{LL}, whence the proof is complete.
\end{proof}

We obtain the following formula for the Gateaux derivative $D\Phi(x)$ of $\Phi$ at points $x \in C$ in directions $u \in C$.

\begin{corollary}\label{HuaDPhi} For all $x, u \in C$ with $2u \le x$, the Gateaux derivative of $\Phi$ at $x$ in the direction $u$ is given by
\begin{equation}\label{DPhixy}
D\Phi(x)(u) = \Phi(x+\Phi^{-1}(\Phi(u)-\Phi(x)))-\Phi(x).
\end{equation}
\end{corollary}
\begin{proof}
Since $2u \le x$ gives $\Phi(x) \le \frac{1}{2}\Phi(u)$, we have $\Phi(u) - \Phi(x) \in D$. As a result, the element $y = \Phi^{-1}(\Phi(u) - \Phi(x)) \in C$ is well-defined and satisfies $\Phi(x) + \Phi(y) = \Phi(u)$. Therefore, Hua's identity \eqref{Hua} yields $D\Phi(x)(u) = \Phi(x+y) - \Phi(x)$, which is \eqref{DPhixy}.
\end{proof}

\section{Proof of the main theorems}\label{MTSec}
In this section, we shall construct the JB-algebra structure on $V$ from the gauge-reversing bijection $\Phi\colon C \to D$. We perform the analysis to prove that gauge-reversing bijections are (twice) continuously differentiable in \Cref{difSec} and \Cref{twdSec}. In \Cref{symsubSec} we consider $d_T$-symmetries and partially construct the quadratic product as a map $P\colon C \to L(V)$. In \Cref{smgrSec}, we will extend $P$ to a quadratic map on $V$ and deduce that gauge-reversing bijections are analytic. We show that the quadratic product $P\colon V \to L(V)$ turns $V$ into a quadratic Jordan algebra in \Cref{quadSec}. We verify that $V$ is a (quadratic) JB-algebra for the order unit norm in \Cref{qJBSec}. Finally, \Cref{MAINTHM} and \Cref{MAINTHM2} will be proven in \Cref{MAINTHMsection}.\\
\ind In \Cref{ncmplSec} we consider non-complete order unit spaces and prove \Cref{grevThm}. \Cref{nuJBSec} is concerned with a characterization of non-unital JB-algebras among the ordered Banach spaces.
%will show that our prospective inversion map $j\colon C \to C$ is twice continuously differentiable in the order unit norm $\lVert \cdot \rVert_v$. We start by showing that $j$ is Lipschitz on sets on the form $\{x \in C: x \ge \lambda^{-1}v\}$ for fixed $\lambda > 0$. Then we show that the map $Dj\colon C \to L(V, V)$ is also Lipschitz on such sets. This result is used in \Cref{jdiff} to deduce (uniform) Fréchet differentiability of $j$ on these sets from Gateaux-differentiability in positive directions. In turn this result is used to prove that the quadratic representation $P\colon C \to L(V, V)$ is continuously differentiable. As \Cref{jtwdiff} the map $j$ will be shown to be twice continuously differentiable, using that $Dj = P \circ j$.

\subsection{Differentiability}\label{difSec}

In \Cref{DPhi} we have established that for each $x \in C$ there exists an isomorphism of order unit spaces $-D\Phi(x): (V, V_+, x) \to (W, W_+, \Phi(x))$ such that if $y \in V_+$ then $D\Phi(x)(y)$ is the Gateaux-derivative of $\Phi$ at $x$ in the direction $y$. In this section we show that $D\Phi(x)$ is in fact the Fréchet-derivative of $\Phi$ at $x$.\\

The space $L(V, W)$ of all bounded linear operators $A \colon V\to W$ is a Banach space in the operator norm $$\lVert A\rVert_{v,w} := \sup \{\lVert Ax\rVert_w: x \in V, \lVert x\rVert_v \le 1\}.$$

\begin{lemma}\label{Aopnm}
For each $A \in L(V, W)$ we have
\end{lemma}
\begin{equation}\label{A3nm}
\lVert A\rVert_{v,w} \le 3 \sup \{\lVert Ax\rVert_w: x \in C, \lVert x\rVert_v \le 1\}.
\end{equation}
\begin{proof}
Every $x\in V$ can be written as $x = x_1 - x_2$ with $x_1, x_2 \in V_+$ and $\lVert x_1\rVert_v + \lVert x_2\rVert_v \le 3\lVert x\rVert_v$. Indeed, we may take $x_2 = \lVert x\rVert_vv$ and $x_1 = x + x_2$. The lemma follows since $C \cap [0,v]$ is $\lVert\cdot\rVert_v$-dense in $[0, v]$.
\end{proof}

\begin{lemma}\label{Phislw} For all $x \in C$ we have
\begin{equation}\label{Top}
\lVert D\Phi(x)\rVert_{v,w} \le M(v,x)^2.
\end{equation}
\end{lemma}
\begin{proof}
 The isomorphism of order unit spaces $-D\Phi(x)\colon (V, x) \to (W, \Phi(x))$ is an isometry with respect to order unit norms, that is, $\lVert -D\Phi(x)y\rVert_{\Phi(x)} = \lVert y\rVert_{x}$ for each $y \in V$. The order unit norms $\lVert\cdot\rVert_x$ and $\lVert\cdot\rVert_v$ on $V$ satisfy $\lVert \cdot\rVert_x \le M(v,x)\lVert \cdot\rVert_v$. Likewise the norms $\lVert\cdot\rVert_{w}$ and $\lVert \cdot \rVert_{\Phi(x)}$ on $W$ are related by $\lVert \cdot\rVert_w \le M(\Phi(x),w)\lVert \cdot\rVert_{\Phi(x)}$. Since $\Phi$ is gauge-reversing, it follows that
$$\lVert D\Phi(x)y\rVert_w \le M(\Phi(x),w) \lVert D\Phi(x)y\rVert_{\Phi(x)} = M(v,x) \lVert y\rVert_{x} \le M(v,x)^2 \lVert y\rVert_v$$
for every $y \in V$, whence $\lVert D\Phi(x)\rVert_{v,w}\le M(v,x)^2.$

\end{proof}

The following lemma will be used to show continuity of the map $D\Phi: C \to L(V, W)$.
%XXX: change notation from L to B, for spaces of bounded linear map

\begin{lemma}\label{DPhixyu}
Let $\lambda > 0$. For all $x,y,u \in C$ with $x,y \ge \lambda^{-1}v \ge  2u$ we have
\begin{equation}
\lVert (D\Phi(x) - D\Phi(y))(u) \rVert_w \le 3\lambda^2 \lVert x-y\rVert_v.
\end{equation}
\end{lemma}
\begin{proof}
\Cref{HuaDPhi} applied twice to $x \ge 2u$ and $y\ge 2u$ gives
\begin{equation}\label{DPhixy0}
D\Phi(x)(u) - D\Phi(y)(u) = \Phi(x + \Phi^{-1}(\Phi(u)-\Phi(x)))-\Phi(y+\Phi^{-1}(\Phi(u)-\Phi(y)))-\Phi(x)+\Phi(y).
\end{equation}
By our assumption $x, y \ge \lambda^{-1}v$, the Lipschitz property \eqref{jlip} of $\Phi$ gives that 
\begin{equation}\label{DPhixya}
\lVert \Phi(x)-\Phi(y)\rVert \le \lambda^2 \lVert x-y\rVert_v.
\end{equation}
Because $\Phi(u)-\Phi(x) \ge \Phi((2\lambda)^{-1}v)-\Phi(\lambda^{-1}v) = 2\lambda w - \lambda w = \lambda w$ and likewise $\Phi(u) - \Phi(y) \ge \lambda w$, we find again by \eqref{jlip} that
\begin{equation}\label{Phiinv}
\lVert \Phi^{-1}(\Phi(u)-\Phi(x))-\Phi^{-1}(\Phi(u)-\Phi(y))\rVert_v \le \lambda^{-2} \lVert \Phi(x)-\Phi(y)\rVert_w\\
\le \lVert x-y\rVert_v.
\end{equation}
Once more applying \eqref{jlip}, this time to $x+\Phi^{-1}(\Phi(u)-\Phi(x))$ and $y + \Phi^{-1}(\Phi(u)-\Phi(y))$, gives using the triangle inequality and (\ref{Phiinv}) that
\begin{align}\label{DPhixyb}
\lVert \Phi(x+\Phi^{-1}(\Phi(u)-\Phi(x))) - \Phi(y+\Phi^{-1}(\Phi(u)-\Phi(y)))\rVert_w \le
%\lambda^2 \lVert x+\Phi^{-1}(\Phi(u)-\Phi(x)) - y - \Phi^{-1}(\Phi(u)-\Phi(y))\rVert_v &\le\\
%&\lambda^2(\lVert x-y\rVert_v + \lVert \Phi^{-1}(\Phi(u)-\Phi(x))-\Phi^{-1}(\Phi(u)-\Phi(y))\rVert_v )
2\lambda^2 \lVert x-y\rVert_v.
\end{align}
Using the triangle inequality on \eqref{DPhixy0}, it follows from \eqref{DPhixyb} and \eqref{DPhixya} that
\begin{align*}
\lVert D\Phi(x)(u) - D\Phi(y)(u) \Vert_w \le
%\lVert \Phi(x + \Phi^{-1}(\Phi(u)-\Phi(x)))-\Phi(y+\Phi^{-1}(\Phi(u)-\Phi(y)))\Vert_w + \lVert \Phi(x)-\Phi(y)\rVert_v &\le\\
2\lambda^2 \lVert x-y\rVert_v + \lambda^2 \lVert x-y\rVert_v &= 3\lambda^2\lVert x-y\rVert_v.
\end{align*}
\end{proof}

\begin{proposition}\label{Djcont}
Let $\lambda > 0$. For all $x, y \in C$ with $x,y \ge \lambda^{-1}v$ one has
$$\lVert D\Phi(x) - D\Phi(y) \rVert_{v,w} \le 18\lambda^3 \lVert x-y\rVert_v.$$
Moreover, for all $u \in V_+$ one has
\begin{equation}\label{DPhixyV+}
\lVert (D\Phi(x) - D\Phi(y))(u) \rVert_{w} \le 6\lambda^3 \lVert x-y\rVert_v\lVert u\rVert_v.
\end{equation}
\end{proposition}
\begin{proof}
In view of \Cref{Aopnm}, it suffices to show the last estimate.\\
\ind First assume that $u \in C$. Then $\hat{u} = (2\lambda \lVert u\rVert_v)^{-1}u$ satisfies $2\hat{u} \le \lambda^{-1}v$. By the result of the previous paragraph, and the linearity of $D\Phi(x) - D\Phi(y)$ we obtain
\begin{equation*}
\lVert (D\Phi(x)-D\Phi(y))u\rVert_w = 2\lambda \lVert u\rVert_v \cdot \lVert D\Phi(x)\hat{u} - D\Phi(y)\hat{u} \rVert_w \le 2\lambda \lVert u\rVert_v \cdot 3\lambda^2 \lVert x-y\rVert_v = 6\lambda^3 \lVert x-y\rVert_v \lVert u\rVert_v.
\end{equation*}
\ind Since $C$ is $\lVert\cdot\rVert_v$-dense in $V_+$, the inequality holds for all $u \in V_+$.
\end{proof}

We conclude that $D\Phi\colon C\to L(V, W)$ is locally Lipschitz continuous. The following lemma will be used to show that $D\Phi$ is in fact the Fréchet derivative of $\Phi$. 

\begin{lemma}\label{jdiff}
Let $\lambda > 0$. For all $x \in C$ and $z\in V$ with $x,x+z \ge \lambda^{-1}v$ one has
$$\lVert \Phi(x+z) - \Phi(x) - D\Phi(x)(z) \rVert_w \le 11 \lambda^3 \lVert z\rVert_v^2.$$
\end{lemma}
\begin{proof}
As in the proof of \Cref{Aopnm}, we can write $z = z_2 - z_1$ with $0 \le z_i \le i\lVert z\rVert_vv$ for $i=1,2$. 
In \Cref{DPhi} we have shown that $D\Phi(x)(z_2)$ is the Gateaux-derivative of $\Phi$ at $x$ in the direction $z_2 \in V_+$ as a consequence of the estimate 
$$\lVert \Phi(x+z_2) - \Phi(x) - D\Phi(x)(z_2)\rVert_{\Phi(x)} \le \lVert z_2\rVert_x^2.$$
In view of the fact that $M(\Phi(x),v) = M(v,x)\le \lambda$, we have $\lVert\cdot\rVert_w \le \lambda \lVert\cdot\rVert_{\Phi(x)}$ and $\lVert \cdot \rVert_x \le \lambda \lVert\cdot\rVert_v$, so our estimate entails that
$$
\lVert \Phi(x+z_2) - \Phi(x) - D\Phi (x)(z_2) \rVert_w \le \lambda^3 \lVert z_2\Vert_v^2.$$
Similarly, because $M(v,x+z) \le \lambda$ and $D\Phi(x+z)(z_1)$ is the Gateaux-derivative of $\Phi$ at $x+z$ in the direction $z_1 \in V_+$, we obtain
$$
\lVert \Phi(x+z + z_1) - \Phi(x + z) - D\Phi(x + z)(z_1) \rVert \le \lambda^3 \lVert z_1\rVert_v^2.
$$
Moreover, by the bound \eqref{DPhixyV+} in \Cref{Djcont} we have
$$\lVert (D\Phi(x+z)- D\Phi(x))(z_1) \rVert_v \le 6 \lambda^3 \lVert z\rVert_v \lVert z_1\rVert_v.$$
Using the triangle inequality on
\begin{align*}
\Phi(x+z) - \Phi(x) - D\Phi(x)z &= 
\Phi(x+z_2) - \Phi(x) - D\Phi(x)z_2 - (D\Phi(x+z) - D\Phi(x))z_1\\
&\phantom{=} - (\Phi(x+z+z_1) - \Phi(x+z) - D\Phi(x+z)z_1) ,
\end{align*}
it follows that
$$\lVert \Phi(x+z) - \Phi(x) - D\Phi(x)z\rVert_v \le \lambda^3(\lVert z_2\rVert_v^2 + 6 \lVert z\rVert_v \lVert z_1\rVert_v + \lVert z_1\rVert_v^2) \le \lambda^3 (2^2 + 6 + 1)\lVert z\rVert_v^2 = 11\lambda^3 \lVert z\rVert_v^2.$$ 
\end{proof}

\begin{theorem}\label{C1Phi}
The map $\Phi: C \to D$ is continuously differentiable.
\end{theorem}
\begin{proof}
Let $x \in C$, and set $\lambda = 2m(x,v)^{-1}$. If $z \in V$ is small enough so that $\lVert z\rVert_v \le \lambda^{-1}$, then $x+z \ge m(x,v)v - \lambda^{-1}v=\lambda^{-1}v$.  Consequently, we may apply \Cref{jdiff} to obtain that $\lVert \Phi(x+z) - \Phi(x) - D\Phi(x)(z)\rVert_w \le 11\lambda^3 \lVert z\rVert_v^2$. We conclude that $D\Phi(x)$ is the Fréchet-derivative of $\Phi$ at $x$. Similarly, by \Cref{Djcont} we have $\lVert D\Phi(x+z) - D\Phi(x)\rVert \le 18\lambda^3 \lVert z\rVert_v$, hence $D\Phi: C \to L(V, W)$ is continuous at $x$. Since $x \in C$ was arbitrary, the theorem is proved.
\end{proof}

\subsection{Symmetries and the quadratic representation}\label{symsubSec}

In this subsection, we define and construct symmetries of $C$ at each of its points. As a particular case, the symmetry $j = S_v$ at $v$ will later feature as the inversion map of the JB-algebra structure on $V$ with identity element $v$. Moreover, we introduce the quadratic representation and prove various formulae for it.\\
\ind The following lemma shows that a gauge-reversing bijection is determined by its linearization.

\begin{lemma}\label{!Sx}
Let $\Phi, \Upsilon \colon C \to D$ be two gauge-reversing bijections. If there exists $x \in C$ such that $\Phi(x) = \Upsilon(x)$ and $D\Phi(x) = D\Upsilon(x)$, then $\Phi$ and $\Upsilon$ are equal.
\end{lemma}
\begin{proof}
Since $\Phi$ and $\Upsilon^{-1}$ are gauge-reversing bijections, we deduce that the composition $T = \Upsilon^{-1} \circ \Phi$ is a gauge-preserving bijection, hence extends to a bounded linear map on $V$ by \Cref{ThmBsa}. Two applications of the chain rule, and our assumptions on $\Phi$ and $\Upsilon$ at $x$ give that $$DT(x) = D(\Upsilon^{-1})(\Phi(x)) \circ D\Phi(x) = D(\Upsilon^{-1})(\Upsilon(x)) \circ D\Upsilon(x) = \id_V.$$ It follows that $T = DT(x)|_C = \id_C$, whence $\Phi = \Upsilon \circ T = \Upsilon$.
\end{proof}

\begin{definition}\label{symC}
For every $x\in C$ we define a \emph{symmetry of $C$ at }$x$ to be a gauge-reversing bijection $S_x\colon C \to C$ such that $S_x(x) = x$ and $DS_x(x) = -\id_V$.
\end{definition}

By \Cref{!Sx}, if a symmetry $S_x$ of $C$ at $x$ exists, it is unique. Existence is granted by the next lemma which is inspired by \cite[Lemma 3.3]{Walsh18}.

\begin{lemma}\label{SymMake} Let $\Upsilon\colon C \to D$ be any gauge-reversing bijection. Then for every $x \in C$, the map $S_x\colon C \to C$ defined by
$$S_x = -[D\Upsilon(x)]^{-1} \circ \Upsilon$$
is the symmetry of $C$ at $x$.
\end{lemma}
\begin{proof}
Recall from \Cref{DPhi} that $-D\Upsilon(x) \colon (V, V_+, x) \to (W, W_+, \Upsilon(x))$ is an isomorphism of order unit spaces. It follows that $S_x: C \to C$ is a well-defined gauge-reversing bijection with $S_x(x)=x$. Since $DS_x(x)= -[D\Upsilon(x)]^{-1} \circ D\Upsilon(x) = -\id_V$ by the chain rule, we conclude that $S_x$ is a symmetry of $C$ at $x$.
\end{proof}

In \Cref{dTsymDef} we defined a $d_T$-symmetry at a point $x\in C$ to be an involutive $d_T$-isometry having $x$ as an isolated fixed point. According to the following lemma, the symmetries of $C$ are $d_T$-symmetries and in particular are involutory.

\begin{lemma}\label{SymEquiv}
Let $x \in C$ and suppose $S_x\colon C\to C$ is the symmetry of $C$ at $x$. Then $S_x$ is a $d_T$-symmetry at $x$.
\end{lemma}
\begin{proof}
Observe that $S_x^2$ is a gauge-preserving bijection, which extends to a linear automorphism of $V$ by \Cref{NSThmB}. Since $D(S_x^2)(x) = DS_x(x) \circ DS_x(x) = -\Id_V \circ -\Id_V = \Id_V$, we find that $S_x^2 = \Id_C$. As in the proof of \cite[Theorem~7.4]{LRW25}, it follows from $DS_x(x) = -\Id_V$ that $x$ is an isolated fixed point of $S_x$.

\end{proof}

For the remainder of this subsection, we assume that some gauge-reversing bijection $\Phi\colon C \to D$ exists. For $x\in C$, we denote by $S_x$ the unique symmetry of $C$ at $x$. From \Cref{SymEquiv} we see that $S_x$ is involutory, i.e.\ satisfies $S_x^2 = \id_C$. The gauge-preserving bijection $S_x \circ S_v: C \to C$ extends by \Cref{ThmBsa} to an automorphism $P(x)$ of $(V, V_+)$. It follows that $S_x = S_x \circ S_v \circ S_v = P(x) \circ S_v$, so any two symmetries are related by a linear automorphism of $(V, V_+)$.  

\begin{definition}\label{Pdef}
Let $(V, V_+, v)$ be a complete order unit space with open cone $C := V_+^{\circ}$. Assume that for each $x\in C$ there exists a symmetry $S_x\colon C\to C$ of $C$ at $x$.
\begin{enumerate}[label={\normalfont(\arabic*)}]
\item We denote $j: C \to C$ the symmetry $S_v$ of $C$ at the given order unit $v$.
\item We define the \emph{quadratic representation} of an element $x \in C$ to be the automorphism $P(x)$ of $(V, V_+)$ determined by $S_x = P(x) \circ j$.
\end{enumerate}
\end{definition}

\ind Note that $P(x)|_C = S_x \circ j$, hence it follows from \Cref{SymMake} that for all $x \in C$ we have 
\begin{equation}\label{PDjj}
P(x) = [-Dj(x)]^{-1} = -Dj(j(x)).
\end{equation}

Hua's identity \eqref{Hua} for the symmetry $S_x$, which satisfies $S_x(x) = x$ and $-DS_x(x) = \id_V$, reads
\begin{equation}\label{Sxxy}
S_x(x + y) + S_x(x + S_x(y)) = x.
\end{equation}

\begin{corollary}\label{PHua}
For all $x,y \in C$ we have
\begin{equation}\label{Pxy}
P(x)(y) = j(j(x) - j(x+j(y))) - x
\end{equation}
\end{corollary}
\begin{proof}
Since $S_x = P(x)\circ j$ gives that $x=S_x(x)=P(x)j(x)$, applying $P(x)^{-1}$ to \eqref{Sxxy} gives that
\begin{equation}\label{jxy}
j(x + y) + j(x + P(x)j(y)) = j(x)
\end{equation}
We may replace $y$ by $j(y)$, and rewrite this as
$$j(x + P(x)y) = j(x) - j(x+j(y)).$$
Since $j^2 = \id_{C}$, applying $j$ to both members gives
$$x + P(x)y = j(j(x) -j(x+j(y))).$$
Subtract $x$ to arrive at \eqref{Pxy}.
\end{proof}

\begin{theorem}\label{xysq}
For all $x,y \in C$ we have
\begin{equation}\label{Gstr}
(P(x)-P(y))j(x+y)= x - y.
\end{equation}
\end{theorem}
\begin{proof}
Recall from \Cref{Pdef} and \Cref{SymEquiv} that $P(x) \circ j = S_x = S_x^{-1} = j \circ P(x)^{-1}$ and $S_x(x) = x$. From Hua's identity \eqref{Sxxy} for the symmetries $S_x$ and $S_y$ we obtain that
$$(P(x)-P(y))j(x+y) = S_x(x+y)-S_y(x+y)=x-y-(S_x(x+S_xy)-S_y(S_yx+y)).$$
The final term cancels since
$$S_x(x+S_x(y))=S_x^{-1}(S_x(x)+S_x(y))=(j \circ P(x)^{-1})(P(x)(j(x))+P(x)(j(y)))=j(j(x)+j(y)),$$
and similarly with $x$ and $y$ interchanged.
\end{proof}

The uniqueness of symmetries has the following two consequences, as has been observed by Loos in \cite[Lemma 1.1(a)]{Loos69}.

\begin{lemma}\label{SSS}
For all $x,y \in C$ we have
$S_xS_yS_x=S_{S_x(y)}$.
\end{lemma}
\begin{proof}
It suffices to show that $S_xS_yS_x$ is the unique symmetry of $C$ at $S_x(y)$. Note first that $(S_xS_yS_x)(S_x(y))=(S_xS_y)(y)=S_x(y)$. Since $DS_y(y)=-\id_V$, using the chain rule we obtain $D(S_xS_yS_x)(S_x(y))=DS_x(y)\circ DS_y(y) \circ DS_x(S_x(y)) = - DS_x(y) \circ DS_x(S_x(y)) = -D(S^2_x)(S_x(y))$. Since $S^2_x=\id_{C}$, it follows that $D(S_xS_yS_x)(S_x(y))= -\Id_V$. In view of \Cref{symC}, we conclude that the map $S_xS_yS_x$ is equal to the symmetry $S_{S_x(y)}$ of $C$ at $S_x(y)$.
\end{proof}

\begin{lemma}\label{PPP}
For all $x,y \in C$ we have
\begin{equation}\label{QJ3}
P(x)P(y)P(x) = P(P(x)y).
\end{equation}
\end{lemma}
\begin{proof}
By repeated applications of \Cref{Pdef} and \Cref{SSS} we find that
$$P(x)P(y)P(x)=S_x(S_vS_yS_v) S_xS_v=(S_xS_{S_v(y)}S_x)S_v=S_{S_x(S_v(y))}S_v=S_{P(x)y} S_v=P(P(x)y).$$
\end{proof}

In the remainder of this section, we give a new proof of a result due to
Youngson--Wright, asserting that each surjective unital isometry $T\colon A \to B$ between JB-algebras is a Jordan isomorphism. In their proof \cite[Thm. 4]{WrYo}, they consider the bidual map $T^{**}\colon A^{**} \to B^{**}$ and use that the linear span of projections is norm dense in the JBW-algebras $A^{**}$ and $B^{**}$. Of course, this argument mandates the construction of the Arens product on the bidual of a JB-algebra making it into a JBW-algebra, a feat which is accomplished in \cite[Thm. 1.2]{ShAr}. However, a proof which avoids biduals altogether can be given, using Hua's identity and the uniqueness of symmetries.

\begin{theorem}[Youngson--Wright]\label{YWsym}
Let $A$ and $B$ be unital JB-algebras. Let $T\colon A \to B$ be a surjective linear isometry with $T(1_A) = 1_B$. Then $T$ is an isomorphism of JB-algebras.
\end{theorem}
\begin{proof}
Any surjective unital linear isometry between order unit spaces is an order-isomorphism (cf. e.g.\ \cite[Cor. II.1.4]{Alfsen71}). Therefore $T$ restricts to a gauge-preserving bijection $\Psi\colon A^{\circ}_+ \to B^{\circ}_+$. Let $j_A \colon A^{\circ}_+ \to A^{\circ}_+$ and. $j_B \colon B^{\circ}_+ \to B^{\circ}_+$ be the Jordan inverse maps. Then $\Psi^{-1} \circ j_B \circ \Psi$ is a symmetry of $A^{\circ}_+$ at $1_A$, so by uniqueness of symmetries, \Cref{!Sx}, we have $\Psi^{-1} \circ j_B \circ \Psi = j_A$. Thus for each $x\in A^{\circ}_+$ we have
\begin{equation}\label{jBTjA}
    j_B(T(x)) = T(j_A(x)),
\end{equation}
 and also $U_A(x) = -Dj_A(j_A(x)) =: P_A(x)$. Now Hua's identity \eqref{Pxy} gives
\begin{equation}
U_A(x)(y) = j_A(j_A(x)-j_A(x+j_A(y)))-j_A(x).
\end{equation}
Repeated application of \eqref{jBTjA} now gives
\begin{align*}
T(U_A(x)(y)) &= T(j_A(j_A(x)-j_A(x+j_A(y)))-j_A(x))\\
&= j_B(j_B(T(x))-j_B(T(x)+j_B(T(y))))-j_B(T(x))\\
&= U_B(T(x))(T(y)).
\end{align*}
Therefore $T\colon (A, U_A, 1_A) \to (B, U_B, 1_B)$ is an isomorphism of unital quadratic JB-algebras, as desired.
\end{proof}

\subsection{Twice differentiability}\label{twdSec}

We have shown in \Cref{C1Phi} that each gauge-reversing bijection $\Phi\colon C\to D$ is continuously differentiable. The aim of this subsection is to show in \Cref{jtwdiff} that $\Phi$ is even twice continuously differentiable.\\
\ind Let us first show that the map $P\colon C \to L(V)$ is continuously differentiable. For every $u \in V$, let $\ev_u\colon L(V) \to V$ be the continuous linear map given by evaluation in $u$. Define $Q\colon C \times V \to L(V)$ by
$$Q(x, u) := D(\ev_u \circ P)(x).$$
It is plain that $Q$ is linear in its second variable. Note that if $u \in C$ then
$$(\ev_u \circ P)(x) = P(x)(u) = j(j(x) - j(x+j(u))) - x,$$
so by the chain rule and Fréchet differentiability of $j$ we have
\begin{equation}\label{Qxu}
Q(x, u) = D(\ev_u \circ P)(x) =  Dj(j(x) - j(x+j(u))) \circ (Dj(x) - Dj(x + j(u))) - \Id.    
\end{equation}
Since $Dj\colon C \to L(V)$ is continuous by \Cref{Djcont}, it follows that $Q$ is continuous on $C \times C$, and then by linearity on all of $C \times V$. Hence there exists a continuous map $DP\colon C \to L(V, L(V))$ such that for all $x \in C$ and $z,u \in V$ one has
$$DP(x)(z)(u) = Q(x, u)(z).$$

\begin{theorem}\label{Pdiff}
Let $\lambda > 1$. Then for all $x \in C$ and $z \in V$ such that $x, x+z \in [\lambda^{-1}v, \lambda v]$ and $\lVert z\rVert_v \le (2\lambda)^{-3}$ we have
$$\lVert P(x+z) - P(x) - DP(x)(z) \rVert_v \le 10^4\lambda^6 \lVert z\rVert_v^2.$$
\end{theorem}
\begin{proof}
Fix $u \in C$ with $\lVert u\rVert_v \le \lambda^{-1}$. We shall exhibit a positive integer $N_0$ such that
\begin{equation}\label{dPN0}
\lVert (P(x+z) - P(x) - DP(x)(z))u \rVert_v \le N_0 \lambda^5 \lVert z\rVert_v^2.
\end{equation}
Set $y := x + j(u)$ and define the quantities
\begin{align*}
a(z) &:= j(x) - j(y) + (Dj(x) - Dj(y))z,\\
E_0(z) &:= j(x+z) - j(y+z) - a(z),\\
E_1(z) &:= j(j(x+z)-j(y+z))-j(a(z)),\\
E_2(z) &:= j(a(z)) - j(j(x) - j(y)) - Dj(j(x) - j(y))(Dj(x) - Dj(y))(z).
\end{align*}
By substitution of \eqref{Pxy} and \eqref{Qxu} we obtain
\begin{align*}
(P(x+z) - P(x) - DP(x)(z))u &=\\
j(j(x+z) - j(x+z+j(u)) - (x + z)
- j(j(x)-j(x+j(u))) + x &\phantom{=}\\
- Dj(j(x) - j(y))(Dj(x) - Dj(y))(z) +z &=
E_1(z) + E_2(z).
\end{align*}
Since $y\ge x\ge \lambda^{-1}v$ and $y+z\ge x+z\ge\lambda^{-1}v$ we have by \Cref{jdiff} that
\begin{equation*}
\lVert j(x+z) - j(x) - Dj(x)z\rVert_v \le 11\lambda^3 \lVert z\rVert_v^2
\end{equation*}
and
\begin{equation*}
\lVert j(y+z) - j(y) - Dj(y)z\rVert_v \le 11\lambda^3 \lVert z\rVert_v^2.
\end{equation*}
It follows that
$$\lVert E_0(z)\rVert_v \le 2\cdot 11\lambda^3 \lVert z\rVert_v^2.$$
\ind We have $u\le \lambda^{-1}v \le j(x)$, so using \eqref{Pxy} we find, since $P(x)j(x)=S_x(x)=x$, that
$$j(x) - j(y) = j(x + P(x)u) \ge j(x + P(x)j(x)) = j(2x) \ge (2\lambda)^{-1}v.$$
The same holds with $x$ replaced by $x+z$, hence
$$j(x+z) - j(y+z) \ge (2\lambda)^{-1}v.$$
Since $y \ge x \ge \lambda^{-1}v$, we have by \Cref{Phislw} that $\lVert Dj(x)z\rVert_v \le \lambda^2 \lVert z\rVert_v$ and $\lVert Dj(y)z\rVert_v \le \lambda^2 \lVert z\rVert_v$, whence
$$\lVert (Dj(x) - Dj(y))z\rVert_v \le 2\lambda^2 \lVert z\rVert_v \le \frac{2\lambda^2}{8\lambda^3} = \frac{1}{4\lambda}.$$
It follows that
$$a(z) = j(x) - j(y) + (Dj(x) - Dj(y))z \ge (2\lambda)^{-1}v-(4\lambda)^{-1}v = (4\lambda)^{-1}v.$$
Since $j(x+z)-j(y+z), a(z) \ge (4\lambda)^{-1}v$, it follows by \eqref{jlip} that
$$\lVert E_1(z)\rVert_v \le (4\lambda)^2 \lVert E_0(z)\rVert_v \le (4\lambda)^2 \cdot 2\cdot 11 \lambda^3 \lVert z\rVert_v^2 =: N_1 \lambda^5 \lVert z\rVert_v^2.$$ Finally, since $a(z), j(x) - j(y) \ge (4\lambda)^{-1}v$, we obtain using \Cref{jdiff} that
$$\lVert E_2(z)\rVert_v \le 11(4\lambda)^3 \lVert (Dj(x) - Dj(y))z\rVert_v^2 \le 11(4\lambda)^3 (2\lambda^2 \lVert z\rVert_v)^2 =: N_2 \lambda^5 \lVert z\rVert_v^2.$$ Since $(P(x+z) - P(x) - DP(x)(z))u = E_1(z) + E_2(z)$, we conclude that \eqref{dPN0} holds with constant $N_0 = N_1 + N_2 = 4^2\cdot 2\cdot 11 + 11 \cdot 4^3 \cdot 2^2 \le 10^4/3$.\\
\ind It now follows that for arbitrary $u \in C$ one has
$$\lVert (P(x+z) - P(x) - DP(x)(z))u \rVert_v \le N_0\lambda^5 \lVert z\rVert_v^2 \lVert u\rVert_v.$$
We conclude using \Cref{Aopnm} that
$$\lVert P(x+z) - P(x) - DP(x)(z)\rVert_v \le 3N_0 \lambda^6 \lVert z\rVert_v^2 \le 10^4 \lambda^6 \lVert z\rVert_v^2.$$
\end{proof}

\begin{theorem}\label{jtwdiff}
Each gauge-reversing bijection $\Phi\colon C\to D$ is twice continuously differentiable.
\end{theorem}
\begin{proof}
Since $\Phi = -D\Phi(v) \circ j$, where $-D\Phi(v)\colon (V, \lVert\cdot\rVert_v) \to (W, \lVert\cdot\rVert_w)$ is an isometric isomorphism, it suffices to prove that $j$ is twice continuously differentiable. In \Cref{jdiff} it has been shown that $j$ is continuously differentiable. Moreover, in \Cref{Pdiff} we showed that $P\colon C \to L(V)$ is continuously differentiable as well. Since $Dj = -P \circ j$, it follows that the map $Dj\colon C \to L(V)$ is continuously differentiable, that is, $j$ is twice continuously differentiable.     
\end{proof}

\subsection{Analyticity}\label{smgrSec}

Recall from \Cref{Pdef} that we defined the quadratic representation map $P\colon C \to L(V)$ to be the composition $P = -Dj \circ j$. Since $j$ is homogeneous of degree $-1$, the map $Dj\colon C \to L(V)$ is homogeneous of degree $-2$, and so $P = - Dj \circ j$ has degree $-2 \times -1 = 2$. In the present section we extend $P$ to a quadratic polynomial map $P\colon V \to L(V)$, using the same derivation as in \cite[Theorem 3.4]{McC77}. As a consequence, we prove that gauge-reversing bijections are analytic.

 %We will now deduce from \Cref{xysq} and \Cref{jtwdiff} that $P$ is a quadratic map
\begin{theorem}[{cf. \cite[Theorem 3.4]{McC77}}]\label{Pqw} The quadratic representation map $P\colon C \to L(V)$ extends to a continuous quadratic polynomial map $P\colon V \to L(V)$.
\end{theorem}
\begin{proof}
First fix $z\in C$ and let $y := j(z) \in C$. We will show that $\ev_z \circ P\colon C \to V$, $x \mapsto P(x)z$ extends to a quadratic polynomial map on $V$.\\
\ind Replacing $x$ by $tx$ in \Cref{xysq} gives
\begin{equation*}
(t^2P(x) - P(y))j(tx+y)=tx-y.
\end{equation*}
Differentiating with respect to $t$ gives
\begin{equation}
2tP(x)j(tx+y) + (t^2P(x) - P(y))Dj(tx+y)(x) = x.
\end{equation}
Differentiate again at $t=0$ to find that
\begin{equation}\label{Pd2}
2P(x)z = P(y)[D^2j(y)](x)(x).
\end{equation}
Since $j$ is twice continuously differentiable, its second differential $D^2j(y) \colon V\times V \to V$ at $y$ is a continuous symmetric bilinear map, see e.g.\ \cite[Theorem 1.1.8]{Hor90I}. Therefore, formula \eqref{Pd2} defines an extension of $\ev_z \circ P\colon C\to V$ to a quadratic polynomial map on $V$, as asserted. It follows from $V = C - C$ and linearity of $P(x)$ for each $x\in C$, that such an extension in fact exists for each $z\in V$. Hence we obtain a quadratic polynomial map $P\colon V \to \{T\colon V\to V \colon T \text{ linear}\}$.\\
\ind The map $P$ being a quadratic polynomial, it satisfies for all $x,y \in V$ the parallelogram identity
$$P(x+y) + P(x-y) = 2P(x) + 2P(y).$$
By virtue of this identity, it follows from $V = C - C$ and continuity of $P(x)$ for each $x\in C$, that for arbitrary $x \in V$ the linear map $P(x)$ is continuous, i.e.\ $P(x) \in L(V)$. Similarly, from the continuity of $P$ on $C$ asserted by \Cref{jtwdiff}, one deduces that $P$ is continuous on all of $V = C - C$. Thus $P\colon V \to L(V)$ is a continuous quadratic polynomial map.
\end{proof}

\begin{theorem}\label{jan}
Each gauge-reversing bijection $\Phi\colon C \to D$ is analytic.
\end{theorem}
\begin{proof}
As in the proof of \Cref{jtwdiff} it may be assumed that $\Phi = j$. For each $x\in C$ we have $-Dj(x)x = j(x)$, hence
\begin{equation}\label{jxPx}
    j(x) = P(x)^{-1}(x).
\end{equation}
\ind \Cref{invan} for the Banach algebra $L(V)$ asserts that $G = \{T \in L(V): T\text{ is invertible}\}$ is open and that $G \to G$, $T \mapsto T^{-1}$ is analytic. For each $x \in C$, we have $P(x) = -[Dj(x)]^{-1}\in G$. Moreover $P\colon V \to L(V)$ is analytic, because it is a continuous quadratic polynomial map by \Cref{Pqw}. Moreover, the evaluation map $L(V) \times V \to V$, $(T, x) \mapsto T(x)$ is continuous and bilinear, hence analytic.
Because the composition of analytic maps is again analytic, it follows from \eqref{jxPx} that $j$ is analytic.
\end{proof}

\subsection{Verification of the quadratic Jordan algebra axioms}\label{quadSec}
In this section we will prove that the quadratic map $P: V \to L(V)$ satisfies axioms (QJ1-QJ3) for a quadratic Jordan algebra, following \cite{McC77}.

\begin{proposition}\label{PId}
The quadratic map $P$ satisfies (QJ1), i.e.\
\begin{equation}\label{Pe}
P(v)=\Id_V.
\end{equation}
\end{proposition}
\begin{proof}
Because $j$ is the symmetry of $C$ at $v$, we have that $Dj(v)=-\Id_V$. Therefore, it holds that $P(v) := -[Dj(v)]^{-1} = \Id_V$.
\end{proof}

Let us define the bilinear map $P\colon V \times V \to L(V)$ by
\begin{equation}\label{Pxz}
P(x, z) := \tfrac{1}{2}(P(x + z) - P(x) - P(z)).
\end{equation}
Note that $P(x) = P(x, x)$.

\begin{lemma}
For all $x \in C$ and $y \in V$ we have
\begin{equation}\label{Canc}
P(x,y)j(x)=y.
\end{equation}
\end{lemma}
\begin{proof}
Differentiate the relation $ P(x+ty)j(x+ty)=x+ty$ with respect to $t$ at $t=0$, to obtain
$2P(x,y)j(x)+P(x)Dj(x)y=y$. Since $P(x) \circ Dj(x)=-[Dj(x)]^{-1}\circ Dj(x)=-\Id_V$ it follows that $2P(x,y)j(x)=2y$.
\end{proof}

\begin{proposition}\label{PPyzP}
The quadratic map $P$ satisfies (QJ2), i.e.\ for all $x,y \in V$ we have
\begin{equation}\label{PPPV}
P(x)P(y)P(x) = P(P(x)y).
\end{equation}
\end{proposition}
\begin{proof}
We have proved in \Cref{PPP} that the equation \eqref{PPPV} holds if $x,y \in C$. In view of \Cref{Pqw} this is a polynomial equation in $L(V)$ of degree $4$ in $x$ and of degree $2$ in $y$. Since $V = C - C$ it follows that \eqref{PPPV} holds for all $x,y \in V$.
\end{proof}

On linearizing \eqref{PPPV} we obtain for all $x,y,z\in V$ the identity
\begin{equation}\label{PPyzPeq}
P(x)P(y,z)P(x)=P(P(x)y,P(x)z).    
\end{equation}

\begin{proposition}\label{PxyzP}
The quadratic map $P$ satisfies (QJ3), i.e.\ for all $x,y,z \in V$ we have
$$P(x)P(y,z)x = P(P(x)y, x)z.$$
\end{proposition}
\begin{proof}
Since this is a polynomial equation in $V$ (of degree $3$ in $x$, and of degree $1$ in $y$ and in $z$), we may assume $x,y,z\in C$. Let $a, b \in V$. By \eqref{Canc} we have $P(a, x)j(x) = a$. Differentiate with respect $x$ in the direction $b$, to find $P(a, b)j(x) + P(a, x)Dj(x)(b) = 0$, or
\begin{equation}\label{5.7Mc}
P(a, b)j(x) = P(a, x)P(x)^{-1}(b).
\end{equation}
Substituting $a=P(x)y$ and $b=P(x)z$ gives, using \eqref{PPyzPeq} and \eqref{5.7Mc}, the desired identity
\begin{align*}
P(x)P(y, z)x &= P(x)P(y, z)P(x)j(x) = P(P(x)y, P(x)z)j(x) = P(a,b)j(x)\\
&= P(a,x)P(x)^{-1}(b)=P(P(x)y,x)P(x)^{-1}P(x)z\\ &=P(P(x)y,x)z.
\end{align*}
\end{proof}

\begin{corollary}\label{QJACor}
The quadruple $(V, v, P)$ is a quadratic Jordan algebra.
\end{corollary}
\begin{proof}
The map $P\colon V \to L(V)$ is a quadratic polynomial map by \Cref{Pqw}. It satisfies the quadratic Jordan algebra axioms (QJ1-3) by \Cref{PId}, \Cref{PPyzP} resp. \Cref{PxyzP}. We conclude that the triple $(V, v, P)$ is a quadratic Jordan algebra in the sense of \Cref{QJADef}.
\end{proof}

\subsection{Verification of the quadratic JB-algebra axioms}\label{qJBSec}
In the previous section, we have shown that $(V, V_+, v, P)$ is a quadratic Jordan algebra. To show this quadruple is a quadratic JB-algebra in the sense of \Cref{AltDef}, it remains to be shown that $P(x)$ is positive with norm $\lVert P(x)\rVert_v = \lVert x\rVert_v^2$ for each $x\in V$.

\begin{theorem}\label{PPos}
The linear map $P(x)$ is positive for each $x \in V$.
\end{theorem}
\begin{proof}
We have already seen for each $x \in C$ that $P(x)$ is an automorphism of $(V, V_+)$ in \Cref{Pdef}(2). Fix an arbitrary $x\in V$. Since $V_+$ is the closure of $C = j(C)$, it suffices to show that $P(x)(j(y)) \in V_+$ for all $y \in C$. In view of \eqref{Pd2} and the positivity of $P(y)$, we need only show that $[D^2j(y)](x)(x) \in V_+$. Owing to Kadison's \Cref{KadThm}, it is equivalent to prove that $\psi([D^2j(y)](x)(x))\ge 0$ for each state $\psi \in K$.\\
\ind So let $\psi \in K$, and define the function $f\colon I \to \R$ by $f(t) := \psi(j(y + tx))$ on the open interval $I := \{t\in \R: y + tx \in C\}$ around $0$. Because $\Phi$ is convex by \Cref{Phiconv} and $\psi$ is a positive functional, $f$ is convex as well. Since $\psi$ is a continuous linear functional and $j$ is twice continuously differentiable by \Cref{jtwdiff}, so is the map $f$, and $f''(0) = [D^2(\psi \circ j)(y)](x)(x) = \psi([D^2j(y)](x)(x))$. Since $f$ is convex, it follows that $\psi([D^2j(y)](x)(x)) = f''(0) \ge 0$, as desired. 
%But according to \Cref{Phiconv}, the map $j$ is convex, which implies that
%$$[D^2j(y)](x)(x) = \lim_{t\rightarrow 0} \frac{j(y+tx)+j(y-tx)-2j(y)}{t^2}$$
%belongs to the closed cone $V_+$.
\end{proof}

The formula in the following theorem has been demonstrated in \Cref{xysq} for all $x,y \in C$, and will now be extended to its natural domain of definition.

\begin{lemma}\label{xysqV} For all $x,y \in V$ with $x+ y \in C$ one has
\begin{equation}
(P(x) - P(y))j(x+y) = x - y.
\end{equation}
\end{lemma}
\begin{proof}
Let $U = \{(x,y) \in V \times V \colon x + y \in C\}$. Then $U$ is a connected open subset of the Banach space $(V^2, \lVert\cdot\rVert_{(v,v)})$. The mapping $f\colon U \to V$, $f(x, y) := (P(x) - P(y))j(x+y) - (x - y)$ is analytic on $U$ by \Cref{jan}. By \Cref{xysq} we have $f(x,y) = 0$ for all $(x,y) \in C \times C$. The identity principle for analytic functions \eqref{idprinc} implies that $f$ vanishes identically on $U$.
\end{proof}

We have shown that the map $j$ is analytic in \Cref{jan}. We now show its power series expansion around $v$ is given by the familiar geometric series, following \cite[Theorem 4.4]{McC77}.

\begin{proposition}\label{PExp}
There exists $0 < r \le 1$ such that for all $h \in V$ with $\lVert h\rVert_v < r$ one has
\begin{equation}\label{jvh}
    j(v-h) = \sum_{k=0}^{\infty} h^k.
\end{equation}
\end{proposition}
\begin{proof}
Since $j$ is analytic by \Cref{jan}, there exist $0 < r \le 1$ and for each $k \in \Z_{\ge 0}$ a continuous symmetric homogeneous polynomial map $c_k\colon V \to V$ such that for all $h \in V$ with $\lVert h\rVert < r$ one has
$$j(v + h) = \sum_{k=0}^{\infty} (-1)^kc_k(h).$$
%Let us show that $c_k$ is equal to the continuous symmetric homogeneous polynomial map $d_k\colon V \to V$ defined by $d_k(h) = (-1)^kh^k$.\\
We proceed to show for all $k \in \Z_{\ge 0}$ and $h \in V$ that $$c_k(h) = (-1)^kh^k.$$
If $h = 0$ then $c_0(0) = v = h^0$ and $c_k(0) = 0 = h^k$ for all $k \in \Z_{>0}$. Now assume $h \neq 0$. Let $t$ be a real number with $|t| < \lVert h\rVert_v^{-1}$. Then $v + th \in C$ holds, so \Cref{xysqV} with $x=v$ and $y=th$ gives $(I - t^2P(h))j(v + th) = v - th.$
If in addition $t^2\lVert P(h)\rVert_v < 1$, then $I - t^2P(h)$ is invertible in $L(V)$.  The powers of $h$ in the Jordan algebra $V$ being given by $h^{2k}=P(h)^kv$ and $h^{2k+1}=P(h)^kh$, it follows that whenever $|t| < \min(\lVert h\rVert_v^{-1}, \lVert P(h)\rVert_v^{-1/2})$ one has
$$j(v+th) = (I - t^2P(h))^{-1}(v-th) = \left(\sum_{k=0}^{\infty} t^{2k}P(h)^k\right)(v-th) = \sum_{k=0}^{\infty} (-1)^kt^kh^k.$$
Comparison of both series for $j(v+th)$ yields $c_k(h) = (-1)^kh^k$, whence \eqref{jvh}.
\end{proof}

\begin{remark}
We will shortly see in \Cref{Pnm} that $\lVert P(x)\rVert_v = \lVert x\rVert_v^2$ for each $x\in V$, which implies that \Cref{PExp} holds with $r=1$.    
\end{remark}

\begin{lemma}\label{jvx}
For all $x \in V$ with $\lVert x\rVert_v < 1$ we have
$$v-x^2 = 2j(j(v-x) + j(v+x)).$$
\end{lemma}
\begin{proof}
Let $r > 0$ be as in \Cref{PExp}. Because $x \mapsto x^2 = P(x)v$ is continuous by \Cref{Pqw}, there exists $0 < r_1 \le r$ such that $\lVert x\rVert < r_1$ implies $\lVert x^2\rVert < r$ for each $x \in V$. Let $x \in V$ with $\lVert x\rVert < r_1$. Then for each $h \in \{x, -x, x^2\}$ we have $\lVert h\rVert < r$, hence the power series expansion \eqref{jvh} is valid. We obtain that
$$
j(v-x) + j(v+x) = \sum_{k=0}^{\infty}x^k + \sum_{k=0}^{\infty}(-1)^kx^k = 2\sum_{k=0}^{\infty} x^{2k} = 2j(v-x^2).$$
Now apply $j$ to obtain $v - x^2 = 2j(j(v-x)+j(v+x))$ provided that $\lVert x\rVert < r_1$.\\
\ind Both sides of this equation are analytic functions of $x$ on the open unit ball, and have just been shown to be equal on the open ball of radius $r_1 > 0$ around $0$. In view of the identity principle, they are identically equal on the whole open unit ball. 
\end{proof}

\begin{lemma}\label{ASlem}
For all $x\in V$ with $-v \le x \le v$ we have $0 \le x^2 \le v$.
\end{lemma}
\begin{proof}
First assume that $\lVert x\rVert_v < 1$. By convexity of $j$ we have
$$0 \le j(\tfrac{1}{2}j(v-x) + \tfrac{1}{2}j(v+x)) \le j(v) = v.$$
Upon using \Cref{jvx}, we obtain $0 \le v - x^2 \le v$. Now a limit argument shows this also holds if $\lVert x\rVert_v = 1$.
\end{proof}

\begin{proposition}\label{Pnm}
For each $x \in V$ we have $\lVert P(x)\rVert_v = \lVert x\rVert^2_v$.
\end{proposition}
\begin{proof}
Since $P \in L(V)$ is positive, we have $\lVert P(x)\rVert_v = \lVert P(x)v\rVert_v = \lVert x^2\rVert_v.$ Thus our objective is to show that $\lVert x^2\rVert_v = \lVert x\rVert_v^2$ for each $x\in V$. Since both members of this equation are homogeneous of degree $2$ in $x$, it suffices to prove that $\lVert x^2\rVert_v \le 1$ if and only if $\lVert x\rVert_v \le 1$.\\
\ind Suppose that $\lVert x\rVert_v \le 1$, i.e.\ $-v \le x \le v$. Then \Cref{ASlem} gives $0 \le x^2 \le v$, whence $\lVert x^2\rVert_v \le 1$. Conversely suppose that $\lVert x^2\rVert_v \le 1$, i.e.\ $-v \le x^2 \le v$. Using that $y^2 \in V_+$ for each $y \in V$, we find that
$$-v \le \tfrac{1}{2}((x+v)^2 - x^2 - v) =x = \tfrac{1}{2}(x^2 + v - (x-v)^2) \le v,$$
whence $\lVert x\rVert_v \le 1$. This concludes the proof of the proposition.
\end{proof}

\begin{corollary}\label{qJBAlg} The quadruple $(V, V_+, v, P)$ is a (quadratic) JB-algebra.
\end{corollary}
\begin{proof}
In \Cref{QJACor} we have seen that $(V, v, P)$ is a quadratic Jordan algebra. For $(V, V_+, v, P)$ to be a quadratic JB-algebra in the sense of \Cref{AltDef}, we need to verify in addition, for each $x \in V$, that $P(x)$ is a positive operator with norm $\lVert P(x)\rVert_v = \lVert x\rVert_v^2$, which has been done in \Cref{PPos} and \Cref{Pnm}.
\end{proof}

\subsection{Proof of the main theorems}\label{MAINTHMsection}
In this section we collect the main results of this paper concerning complete order unit spaces. In particular, we complete the proof of \Cref{MAINTHM} and \Cref{MAINTHM2} announced in the introduction.\\
\ind The if direction of \Cref{MAINTHM} has been shown in \Cref{JBigr}; the only if direction is part (4) of the following theorem.

%\begin{theorem}\label{MAINTHMC}
%Let $(V, V_+, v)$ be a complete order unit space. Suppose there exists a gauge-reversing bijection $\Phi\colon C \to D$ to a second complete order unit space $(W, W_+, w)$. Then there exists a unique JB-algebra structure on $V$ with cone of squares $V_+$ and unit element $v$.
%\end{theorem}

\begin{theorem}\label{MAINTHMC}
Let $(V, V_+, v)$ be a complete order unit space with open cone $C := V_+^{\circ}$. Let $\Phi\colon C \to D$ be a gauge-reversing bijection to the open cone $D := W^{\circ}_+$ of a second complete order unit space $(W, W_+, w)$.
\begin{enumerate}
    \item[$(1)$] The map $\Phi\colon C \to D$ is bianalytic.
    \item[$(2)$] The map $-D\Phi(v): (V, V_+, v) \to (W, W_+, w)$ is an isomorphism of order unit spaces.
    \item[$(3)$] The map $S_v := -[D\Phi(v)]^{-1} \circ \Phi$ is the unique symmetry of $C$ at $v$.
    \item[$(4)$] There exists a unique JB-algebra structure on $V$ with cone of squares $V_+$ and unit element $v$.
    \item[$(5)$] Its set of invertible squares is $C$, and inversion on $C$ is given by the symmetry $S_v$ of $C$ at $v$.
\end{enumerate}
\end{theorem}

\begin{remark}
Wright-Youngson have shown in \cite[Theorem 4]{WrYo} the uniqueness assertion of part (4), i.e.\ that on a complete order unit space $(V, V_+, v)$ there exists at most one JB-algebra structure with cone of squares $V_+$ and unit element $v$. Equivalently, they showed that if $V$ and $W$ are JB-algebras, then each isomorphism between the associated order unit spaces $(V, V_+, v)$ and $(W, W_+, w)$ is in fact a JB-algebra isomorphism. Below we will give a different proof of the uniqueness of a JB-algebra structure on $(V, V_+, v)$ based on the uniqueness of a symmetry $S_v$ of $C$ at $v$. 
\end{remark}

\begin{proof}
The gauge-reversing bijections $\Phi$ and $\Phi^{-1}$ are analytic by \Cref{jan}, hence $\Phi$ is bianalytic. In particular, $\Phi$ is differentiable, so part (2) is contained in \Cref{DPhi}. Part (3) is shown in \Cref{!Sx} and \Cref{SymMake}. To ease notation set $j = S_v$. In \Cref{Pqw} we have seen that the map $P \colon C \to \Aut((V, V_+))$ sending $x$ to $P(x) := -Dj(j(x))$ extends to a quadratic polynomial map $P\colon V \to L(V)$. In \Cref{qJBAlg} we have established that $(V, V_+, v, P)$ is a (quadratic) JB-algebra, whence the existence assertion in (4). We have encountered part (5) in \eqref{jxPx}. It remains to prove the uniqueness assertion in (4).\\
\ind Suppose that $P'\colon V \to L(V)$ is another structure of quadratic JB-algebra on $(V, V_+, v)$ and let $j'\colon C \to C$ be the corresponding inversion map. Then $j$ and $j'$ are both symmetries of $C$ at $v$, so \Cref{!Sx} concerning the uniqueness of symmetries gives that $j = j'$. Now in any JB-algebra $A$ with unit element $1$, we have for all $x \in A^{
\circ}_+$ Hua's identity $x^2 = (x^{-1} - (x+1)^{-1})^{-1}$. Indeed, this identity may easily be verified in the associative unital JB-subalgebra of $A$ generated by $x$. Using this identity and the fact that $V = C - C$ one shows that $j = j'$ implies $P = P'$ (cf. the second paragraph of the proof of \Cref{evThm}). 
This establishes uniqueness of a JB-algebra structure on $(V, V_+, v)$ with identity element $v$ and cone of squares $V_+$.
\end{proof}

We proceed to prove \Cref{MAINTHM2}. The following theorem has been shown in \cite[Theorem~8.1]{LRW25} for reflexive order unit spaces, and more generally, for complete order unit spaces which satisfy $C = C_{usc}$ and $V = \spn \ext C$.

\begin{theorem}\label{MAINTHMC2}
Let $(V,V_+,v)$ be a complete order unit space and let $C = V_+^{\circ}$ be its open cone. Then the following are equivalent:
\begin{enumerate}[label={\normalfont(\arabic*)}]
    \item we have that $(C, d_T)$ is a symmetric Banach--Finsler manifold;
    \item there exists a $d_T$-symmetry $S_p: C \to C$ at some point $p \in C$;
    \item there exists a gauge-reversing bijection $\Phi\colon C \to D$ to the open cone $D := W^\circ_+$ of a complete order unit space $(W,W_+.w)$;
    \item there exists a unital JB-algebra structure on $(V,\lVert\cdot\rVert_v)$ with cone of squares $V_+$.
\end{enumerate}
\end{theorem}

\begin{proof}
(1) $\implies$ (2): The definition of a symmetric Banach--Finsler manifold entails that there exists a $d_T$-symmetry of $C$ at each of its points, in particular at $p$.\\
(2) $\implies$ (3): Each $d_T$-symmetry of $C$ is a gauge-reversing bijection by \cite[Thm. 3.5]{LRW25}.\\
(3) $\implies$ (4): This is part (4) of \Cref{MAINTHMC}.\\
(4) $\implies$ (1): Choose a JB-algebra structure on $V$. Denote $j\colon C \to C$ the inversion map and $v$ its multiplicative identity. By \Cref{JBigr} we have that $j\colon C \to C$ is a gauge-reversing bijection. Let $p\in C$ be an arbitrary point. Then the map $S_p := -[Dj(p)]^{-1} \circ p$ is a $d_T$-symmetry at $p$ by \Cref{SymMake} and \Cref{SymEquiv}. We conclude that $(C, d_T)$ is a symmetric Banach--Finsler manifold.
\end{proof}

\subsection{Gauge-reversing bijections on non-complete order unit spaces}\label{ncmplSec}

%We obtain an affirmative a question raised by Noll and Sch\"afer, whether two order unit spaces between which a gauge-reversing bijection exists are necessarily isomorphic, which we stated as .
%XXX: Only do gauge-reversing case. I do not know it for gauge-preserving case.

In this section we will prove \Cref{grevThm}, which asserts that two order unit spaces are necessarily isomorphic if there exists a gauge-reversing bijection between their open cones. For \emph{complete} order unit spaces this has been shown in \Cref{DPhi}; in order to reduce to this case, we will extend a gauge-reversing bijection between possibly non-complete order unit spaces to their completions. Finally, we exploit this extension to deduce from \Cref{MAINTHMC} that a non-complete order unit space that admits a gauge-reversing bijection to a second order unit space is a norm dense subalgebra of a JB-algebra.\\
\ind Let $(V, V_+, v)$ be a, possibly non-complete, order unit space. In \Cref{VhatL} we constructed a complete order unit space $(\hat{V}, \hat{V}_+, v)$ such that $(\hat{V}, \lVert\cdot\rVert_v)$ is the completion of the normed space $(V, \lVert\cdot \rVert_v)$.

\begin{lemma}\label{dTcompl} Let $(V, V_+, v)$ be an order unit space, and let $(\hat{V}, \hat{V}_+, v)$ be its completion. The inclusion map $i \colon (C, d_T) \to (\hat{C}, d_T)$ is a completion of the metric $(C, d_T)$.
\end{lemma}
\begin{proof}
The map $i$ is an isometry, because the ordering on $V$ is the restriction of the ordering on $\hat{V}$ to $V$. Hence it suffices to show that $C$ is $d_T$-dense in $\hat{C}$. Indeed, let $x \in \hat{C}$ be given and pick $\epsilon > 0$ such that $x \ge \epsilon v$. Let $(y_n)_n$ be a sequence in $V$ such that $\lVert y_n - x \rVert_v \to 0$. Set $x_n := y_n + \lVert y_n - x \rVert_v v$ for each $n$, and note that $x_n \ge x$. Consequently $x_n \ge \epsilon v$, so \Cref{dTlN} gives that $d_T(x_n, x) \le \epsilon^{-1} \lVert x_n - x\rVert_v \le 2\epsilon^{-1} \lVert y_n - x\rVert_v$. We conclude that $d_T(x_n, x) \to 0$ as $n \rightarrow \infty$. Since $x \in \hat{C}$ was arbitrary, this shows that $C$ is $d_T$-dense in $\hat{C}$ as desired.
\end{proof}

\begin{lemma}\label{hatPhi} Let $\Phi\colon C \to D$ be a bijective $d_T$-isometry between the open cones of order unit spaces $V$ resp. $W$. Then there is a unique bijective $d_T$-isometry $\hat{\Phi}\colon \hat{C} \to \hat{D}$ extending $\Phi$. If $\Phi$ is gauge-preversing (resp. gauge-reserving), then $\hat{\Phi}$ is also gauge-preversing (resp. gauge-reserving).
\end{lemma}
\begin{proof}
By the universal property of the completion of a metric space, there is a unique bijective isometry $\hat{\Phi}\colon (\hat{C}, d_T) \to (\hat{D}, d_T)$ extending $\Phi$.\\
\ind Suppose that $\Phi$ is order-reserving. Let $x, y \in \hat{C}$ be such that $x \le y$. We shall show that $\hat{\Phi}(x) \ge \hat{\Phi}(y)$. As in the proof of \Cref{dTcompl} we can find sequences $(x_n)_n$ and $(y_n)_n$ in $C$ which $d_T$-converge to $x$ resp. $y$, and satisfy $x_n \le x \le y \le y_n$ for each $n$. Since $\Phi$ is order-reversing, we have $\Phi(x_n) \ge \Phi(y_n)$ for each $n$. Moreover, we have that $d_T(x_n, x) = d_T(\Phi(x_n), \hat{\Phi}(x)) \to 0$, so \Cref{NldT} gives that $\lVert \Phi(x_n) - \hat{\Phi}(x)\rVert_w \to 0$ as $n \to \infty$.  Similarly, we have $\lVert \Phi(y_n) - \hat{\Phi}(y)\rVert_w \to 0$ as $n \to \infty$. Because the cone $\hat{V}_+$ is the closure of $V_+$, it follows that $\hat{\Phi}(x) \ge \hat{\Phi}(y)$, which establishes that $\hat{\Phi}$ is order-preserving. Similarly, one may show that $\hat{\Phi}$ is order-preversing whenever $\Phi$ is order-preversing.\\
\ind Assume that $\lambda > 0$ and $\mu > 0$ are such that $\mu \Phi(x) = \Phi(\lambda x)$ for all $x \in C$. The map $m_{C,\lambda}\colon C \to C$ given by $m_{C,\lambda}(x) := \lambda x$ for each $x\in C$ is a bijective $d_T$-isometry. We may reformulate our assumption as $m_{D,\mu} \circ \Phi = \Phi \circ m_{C, \lambda}$. It is plain that $\widehat{m_{C,\lambda}} = m_{\hat{C}, \lambda}$, whence $m_{\hat{D}, \mu} \circ \hat{\Phi} = \hat{\Phi} \circ m_{\hat{C}, \lambda}$. We conclude that $\mu \hat{\Phi}(x) = \hat{\Phi}(\lambda x)$ for all $x \in \hat{C}$. Applying this result with $\lambda > 0$ arbitrary and $\mu = \lambda$ (resp. $\mu = \lambda^{-1}$) shows that $\hat{\Phi}$ is homogeneous of degree $1$ (resp. of degree $-1$) whenever $\Phi$ is.\\
\ind We conclude using \Cref{graor} and \Cref{gphop}.
\end{proof}

\Cref{grevThm} will now be demonstrated.

\begin{theorem}\label{ncThmAB} Let $\Phi\colon C \to D$ be a gauge-reversing bijection between the open cones of order unit spaces $(V, V_+, v)$ resp. $(W, W_+, w)$. Then these order unit spaces are isomorphic.
\end{theorem}
\begin{proof}
\Cref{hatPhi} yields a gauge-reserving bijection $\hat{\Phi}\colon \hat{C} \to \hat{D}$ between the open cones of the completions that extends $\Phi$. Now \Cref{DPhi} furnishes the isomorphism of order unit spaces $$-D\hat{\Phi}(v)\colon (\hat{V}, \hat{V}_+, v) \to (\hat{W}, \hat{W}_+, w),$$
whose inverse is given by $-D(\hat{\Phi}^{-1})(w)\colon (\hat{W}, \hat{W}_+, w) \to (\hat{V}, \hat{V}_+, v).$\\
\ind Let us show that $D\hat{\Phi}(v)[V] \subset W$. First, consider $z \in C$ with $z \le \tfrac{1}{2}v$. Then we obtain that $\Phi(z) - w \ge \Phi(\tfrac{1}{2}v) - w = 2w - w = w,$ so $\Phi(z) - w \in D$. Applying 
\Cref{HuaPhi} with $x = v$ and $y = \Phi^{-1}(\Phi(z) - w)$ gives
$$\Phi(v+y) - D\hat{\Phi}(v)(\Phi^{-1}(\Phi(v)+\Phi(y))) = \Phi(v).$$
Since $\Phi(v) + \Phi(y) = \Phi(z)$ and $\hat{\Phi}$ extends $\Phi$, we obtain
$$-D\hat{\Phi}(v)(z) = w - \Phi(v+y) \in D.$$
Since $\{z \in C: z \le \tfrac{1}{2}v\}$ spans $V$ and $-D\hat{\Phi}(v)$ is linear, it follows that $-D\hat{\Phi}(v)[V] \subset W$. Similarly one shows that $-D(\widehat{\Phi}^{-1})(w)[W] \subset V$. We conclude that $-D\widehat{\Phi}(v)$ restricts to an isomorphism of order unit spaces from $(V, V_+, v)$ onto $(W, W_+, w)$.
\end{proof}

\begin{theorem}\label{evThm}
Let $(V, V_+, v)$ be an order unit space. Suppose there exists a gauge-reversing bijection $\Phi\colon C \to D$ to a second order unit space $(W, W_+, w)$. Then there exists a unique Jordan algebra structure on $V$ with identity element $v$ such that $V_+$ equals the closure of $\{x^2: x \in V\}$ in $V$ and $V$ embeds isometrically as a subalgebra of a JB-algebra.
\end{theorem}
\begin{proof}
Let $(\hat{V}, \hat{V}_+, v)$ be the completion of the order unit space $(V, V_+, v)$. According to \Cref{hatPhi} we can extend $\Phi$ to a gauge-reversing bijection $\hat{\Phi}\colon \hat{C} \to \hat{D}.$ Invoke \Cref{MAINTHMC} to obtain a unique JB-algebra structure on $\hat{V}$ with cone of squares $\{x^2: x \in \hat{V}\}=\hat{V}_+$ and unit element $v$. Let us show that $V$ is a Jordan subalgebra of $\hat{V}$, i.e.\ that $x,y \in V$ implies $x \circ y \in V$.\\
\ind Suppose first that $x,y \in C$, so that $x+y \in C$ as well. Identity \eqref{Pxy} gives that
$$x^2 = P(x)v = j(j(x) - j(x+v)) \in V,$$
and similarly $y^2, (x+y)^2 \in V$. It follows that $x 
\circ y = \tfrac{1}{2}((x+y)^2-x^2-y^2) \in V.$ Since $V = C - C$ and the product $\circ$ is bilinear, we obtain that $x \circ y \in V$ for all $x,y\in V$. Therefore $V$ is a norm dense Jordan subalgebra of the JB-algebra $\hat{V}$ containing the identity element $v$.\\
\ind Let us now show that the closure of $\{x^2: x \in V\}$ in $V$ is equal to $V_+$. \Cref{MAINTHMC} asserts that $\{x^2: x \in \hat{V}\} = \hat{V}_+$. By consequence, $\{x^2: x \in V\} \subset \{x^2: x \in \hat{V}\} \cap V = \hat{V}_+ \cap V = V_+$. Since $\hat{V}_+$ is closed in $\hat{V}_+$, it follows that $V_+ = V \cap \hat{V}_+$ it closed in $V$ and the closure of $\{x^2: x \in V\}$ in $V$ is a subset of $V_+$. Conversely, suppose that $y \in V_+$. Then there exists $x \in \hat{V}_+$ with $x^2 = y$. Since $V$ is dense in $\hat{V}$, we can choose a sequence $(x_n)_n$ in $V$ converging to $x$. Because the squaring map is continuous in any JB-algebra, it follows that $(x_n^2)_n$ converges to $x^2 = y$. Thus $y$ belongs to the closure of $
\{x^2: x\in V\}$ in $V$, as desired. This concludes the existence part of the theorem.\\
\ind We now turn to the uniqueness part. Let $B$ be a JB-algebra containing $V$ isometrically as a Jordan subalgebra such that $V_+$ equals the closure of $\{x^2: x\in V\}$ in $V$ and $v$ is an identity element of $V$. Because the embedding of $V$ into $B$ is isometric, we may and will identify the completion $\hat{V}$ of $(V, \lVert\cdot\rVert_v)$ with the closure of $V$ in $B$. Since the multiplication of $B$ is continuous, one sees that $\hat{V}$ is a JB-subalgebra of $B$ with identity element $v$. For the same reason, we have that the set $\{x^2: x \in \hat{V}\}$, which is closed in $\hat{V}$ since $\hat{V}$ is a JB-algebra, is equal to the closure of $\{x^2: x \in V\}$ in $\hat{V}$. Now we have assumed $V_+$ to be the closure of $\{x^2: x\in V\}$ in $V$, and the closure of $V_+$ in $\hat{V}$ is given by $\hat{V}_+$, from where $\{x^2: x \in \hat{V}\} = \hat{V}_+$. According to \Cref{MAINTHMC} there is at most one structure of JB-algebra on $\hat{V}_+$ having cone of squares $\hat{V}_+$ and unit element $v$, and it follows that the structure of Jordan algebra on $V$ having the properties stated in the theorem is unique as well.
\end{proof}

\subsection{Non-unital JB-algebras}\label{nuJBSec}
The main result of this section is \Cref{MAINTHMNU}, which provides a characterization of not necessarily unital JB-algebras among the ordered Banach spaces. As apposed to the rest of this article, JB-algebras are not assumed to possess an identity element.\\

Let us first note two properties of a JB-algebra $A$. First, the positive cone $A_+ := \{x^2: x \in A\}$ is norm closed. Second, the spectral norm $\lVert \cdot \rVert$ is \emph{full}, meaning that $x \le y \le z$ implies $\lVert y\rVert \le \max\{\lVert x\rVert, \lVert z\rVert\}$ for all $x,y,z\in A$. Therefore, the only ordered Banach spaces which potentially admit a compatible structure of JB-algebra have a closed cone and a full norm.\\
\ind We will make essential use of the \emph{order unitization} of an ordered normed space $A$, considered in
\cite{vDdB25}[Thm. 1.3], which is an initial object $h\colon A \to V$ in the category of positive contractive maps of $A$ into an order unit space.
The map $h$ is a bipositive isometry if and only if $A_+$ is closed and the norm is full \cite{vDdB25}[Cor. 3.11 and Cor. 5.7]. In view of the preceding paragraph, we may restrict our attention to this case.
%In this special case, we will give an exposition of the order unitization $V$ based on its open cone $V^\circ_+$ rather than the . This approach, rendered possible by \Cref{gflc}, is better adapted our purposes, for it is on the open cone $V^\circ_+$ that we will seek to define a gauge-reversing bijection.\\

\begin{proposition}\label{hAdj}
Let $A$ be a ordered normed space with closed cone and full norm. Then there exists an isometric order embedding $h\colon A \to V$ of $A$ as a closed hyperplane in an order unit space $(V, V_+, v)$ such that for any order unit space $(W, W_+, w)$ and positive contractive linear map $k\colon A \to W$ there exists a unique positive linear map $T\colon V \to W$ such that $T(v) = w$ and $T \circ h = k$.
\end{proposition}
\begin{proof}
See \cite[Thm. 1.3, Cor. 3.11 and Cor. 5.7]{vDdB25}.
\end{proof}

For the convenience of the reader, we give a brief description of the construction of the order unitization $(V, V_+, v)$ in \cite[Def. 3.1]{vDdB25}. The vector space $V$ is the direct sum
$$V := A \oplus \R v$$
and $h\colon A \to V$, $h(x) = x$ is the inclusion map. Denote the open unit ball of $A$ by $B_A^{\circ}$ and set
\begin{equation}\label{Gdef}
G := B_A^{\circ} + A_+.
\end{equation}
%Then $G$ is an open convex set satisfying $G \cap -G = B^{\circ}_A$ and $
%A_+ = \{x\in A: x + G \subset G\}.$
%This follows from the fact that $B_A^{\circ}$ is a full subset of $A$, and that the cone $A_+$ is closed.
By virtue of our assumption that $A_+$ is closed and $\lVert\cdot\rVert$ is full, the open cone $C := V_+^{\circ}$ is given by
\begin{equation}\label{GDf}
C := \bigcup_{\lambda > 0} \lambda (G + v).
\end{equation}

The following theorem characterizes the ordered Banach spaces which are isometrically order-isomorphic to a JB-algebra. 
%TODO: change iota to Upsilon
\begin{theorem}\label{MAINTHMNU}
Let $A$ be an ordered Banach space with closed cone and full norm. We put $G := B_A^{\circ} + A_+$. Then $A$ is a JB-algebra with cone of squares $A_+$ and norm $\lVert\cdot\rVert$ if and only if there exists a bijection $\phi\colon G \to G$ such that for all $x,y \in G$ and $\mu > \lambda > 0$ one has
\begin{equation}\label{mlyx}
    \frac{\mu y - \lambda x}{\mu - \lambda} \in G \iff \frac{\mu \phi(x) - \lambda \phi(y)}{\mu - \lambda} \in G.
\end{equation}
\end{theorem}
\begin{proof}
Suppose that $A$ is a JB-algebra. We may view $A$ as a JB-subalgebra of its \emph{unitization} $$V := A \oplus \R v,$$
endowed with the product
$$(a + \lambda v) \circ (b + \mu v) = a \circ b + \lambda b + \mu a + \lambda \mu v,$$
which is is a JB-algebra with identity element $v$ (see \cite[Section 3.3]{HOSt84}). The linear map $\chi\colon V \to \R$ given by $\chi(x+\lambda v)=\lambda$ for all $x\in A$ and $\lambda \in \R$
is a \emph{character} of $V$, i.e.\ a unital Jordan homomorphism of $V$ onto $\R$, whose kernel is $\ker \chi = A$.\\
\ind Let $C := V_+^{\circ}$ be the open cone of $V$, and note that $C \cap \chi^{-1}(1) = C \cap (v + A) = v + G$. Denote $j\colon C \to C$ the inversion map of the unital JB-algebra $V$. By multiplicativity of $\chi$ we have $\chi(j(u))=\chi(u^{-1})=(\chi(u))^{-1}$, hence $j(C\cap \chi^{-1}(1)) = C \cap \chi^{-1}(1)$. Therefore, we may define a map $\phi\colon G \to G$ by $\phi(x) = j(v+x) - v$. Since $j^2 = \Id_C$, it follows that $\phi^2 = \Id_G$, hence $\phi$ is a bijection. %Since $j$ restricts to an order-anti-isomorphism of $C \cap \chi^{-1}(1)$, we see that $\phi$ is an order-anti-isomorphism. 
Finally, let $\mu > \lambda > 0$ and $x,y \in G$. It follows from
$$
(\mu - \lambda)\left(v + \frac{\mu y - \lambda x}{\mu - \lambda}\right) =
(\mu - \lambda)v + \mu y - \lambda x = 
\mu(v+y) - \lambda(v+x)$$
that
\begin{equation}\label{mliff}
\mu(v+y) - \lambda(v+x) \in C \iff \frac{\mu y - \lambda x}{\mu - \lambda} \in G.   
\end{equation}
Since $j$ is a gauge-reversing bijection, by \Cref{graor}(2)$\iff$(3) we have
\begin{align*}
\mu(v+y) - \lambda(v+x) \in C &\iff j(\lambda(v+x)) - j(\mu(v+y)) \in C\\
&\iff
\lambda^{-1}(v+\phi(x)) - \mu^{-1}(v+ \phi(y)) \in C\\
&\iff
\mu(v+ \phi(x)) - \lambda(v + \phi(y)) \in C.
\end{align*}
The equivalence \eqref{mliff} with $\phi(x)$ in the role of $y$, and $\phi(y)$ in the role of $x$, gives
\begin{equation}\label{mliiff}
\mu(v+ \phi(x)) - \lambda(v + \phi(y)) \in C \iff \frac{\mu \phi(x) - \lambda \phi(y)}{\mu - \lambda} \in G.
\end{equation}
Stringing these equivalences together yields \eqref{mlyx}.\\
\ind Now assume that $\phi\colon G \to G$ is a bijection satisfying \eqref{mlyx}. Let $(V, V_+, v)$ be the order unit space constructed in the proof of \Cref{hAdj}, and let $C$ be its open cone given by \eqref{GDf}. We define $\Phi\colon C \to C$ by 
\begin{equation}
\Phi(\lambda(v+x))=\lambda^{-1}(v+\phi(x)).
\end{equation} Then clearly $\Phi$ is homogeneous of degree $-1$. Since $\phi$ is bijective, and the union \eqref{GDf} is disjoint, it is easy to see that $\Phi$ is bijective. Finally, let us show that $\Phi$ is an order-anti-isomorphism.
Let $\mu(v+y), \lambda(v+x) \in C$. By \Cref{graor}(2)$\iff$(3) it suffices to show that $\mu(v+y) - \lambda(v+x) \in C$ if and only if $\Phi(\lambda(v+x)) - \Phi(\mu(v+y)) \in C$. Using our assumption \eqref{mlyx}, this follows from the equivalences \eqref{mliff} and \eqref{mliiff} as before.\\
\ind Now by \Cref{MAINTHMC}, the map $j := -[D\Phi(v)]^{-1} \circ \Phi$ is the inversion map for a unique JB-algebra structure on $V$ with cone of squares $V_+$ and unit element $v$. It remains to be shown that $A$ is a JB-subalgebra of $V$.\\  
\ind Let us first show that $j(v + G) = v + G$. By construction, the map $\Phi$ satisfies $\Phi(v + G) = v + G$. Since $G$ is an open neighbhorhood of $0$ in $A$, it follows that $D\Phi(v)[A] \subset A$.
Let $w := \Phi(v) \in v + G$. Since $-D\Phi(v)(v)=w$ by \Cref{DPhiusc}, it follows that $-D\Phi(v)$ sends $v + A$ into $v+A$. Applying the same reasoning to $\Phi^{-1}$ and $w$, instead of $\Phi$ and $v$, shows that $-[D\Phi(v)]^{-1}=-D(\Phi^{-1})(w)$ sends $v+A$ into $v+A$. We conclude that $-[D\Phi(v)]^{-1}(v+A) = v+A$. Since $-D\Phi(v)$ is an automorphism of $(V, V_+, v)$ by \Cref{MAINTHMC}(2), we have that $-[D\Phi(v)]^{-1}[C]=C$, hence $-[D\Phi(v)]^{-1}[v+G]=v+G$. Recalling that $j = -[D\Phi(v)]^{-1} \circ \Phi$, we arrive at $j(v+G) = v+G$, as desired. Since $j$ is homogeneous of degree $-1$, we find that $j(\lambda(v+G)) = \lambda^{-1}(v+G)$ for all $\lambda \in G$. In other words, the map $\chi\colon V \to \R$ given by $\chi(x+\lambda v)=\lambda$ for all $x\in A$ and $\lambda \in \R$ satisfies $\chi(j(u))=\chi(u)^{-1}$ for all $u\in C$.\\
\ind We will now prove that $A$ is an ideal of $V$, by proving that $\chi\colon V \to \R$ is a character. 
It suffices to show for all $x \in V$ with $\lVert x\rVert < 1$ that $\chi(x^2)=\chi(x)^2$. But this follows from the fact that $\chi$ intertwines the inversion on the open cones of the unital JB-algebras $V$ and $\R$, in which \Cref{jvx} holds valid:
\begin{align*}
1 - \chi(x^2) &= \chi(v-x^2)\\
&= \chi(2j(j(v-x)+j(v+x)))\\
&= 2(\chi(v-x)^{-1}+\chi(v+x)^{-1})^{-1}\\
&= 2((1-\chi(x))^{-1}+(1+\chi(x))^{-1})^{-1}\\
&= 1-\chi(x)^2,
\end{align*}
as desired. We conclude that $\chi(x\circ y)=\chi(x)\chi(y)$ for all $x,y \in V$. Since $A = \ker \chi$, it follows that $A$ is an ideal of $V$, so in particular $A$ is a subalgebra of $V$. 
Since $A$ is a closed hyperplane in $V$ by \Cref{hAdj}, it follows that $A$ is a JB-subalgebra of $V$. By the same proposition, the embedding of $(A, \lVert\cdot\rVert)$ into $(V, \lVert\cdot\rVert_v)$ is isometric, whence the assertion concerning norms follows. Finally, since $V_+$ is the cone of squares in $V$ we have $$\{x^2: x \in A\} = \{x^2: x \in V\} \cap A = V_+ \cap A = A_+.$$
The proof of the theorem is complete.
\end{proof}

\begin{remarks}
(1) In the backward implication of \Cref{MAINTHMNU}, assume in addition that $\phi(0)=0$ and $-D\Phi(0) = \Id_A$. Then the map $j\colon C \to C$ constructed in the course of proof satisfies $j(v) = v$ and $-Dj(v) = \Id_V$, so $j$ is the symmetry of $C$ at $v$. Then by \Cref{MAINTHMC}, the inversion on $C$ is given by $j$. It follows from \Cref{jvx} that the Jordan product obtained on $A$  is given explicitly, for $x\in A$ with $\lVert x\rVert < 1$, by
\begin{equation}\label{xsq}
x^2 = -\phi\left(\tfrac{1}{2}\phi(x)+ \tfrac{1}{2}\phi(-x)\right).
\end{equation}
Indeed, we have \begin{align*}
x^2 &= v - j(\tfrac{1}{2}j(v-x) + \tfrac{1}{2}j(v+x))\\
&= v - j(\tfrac{1}{2}(v+\phi(-x))+\tfrac{1}{2}(v+\phi(x)))\\
&= v - j(v + \tfrac{1}{2}(\phi(x)+\phi(-x)))\\
&= v - (v + \phi(\tfrac{1}{2}(\phi(x)+\phi(-x)))\\
&= -\phi(\tfrac{1}{2}\phi(x) + \tfrac{1}{2}\phi(-x)).
\end{align*}
(2) Suppose that $A$ is an ordered non-complete normed space with closed cone and full norm, and that $\phi\colon G \to G$ is a bijection satisfying \eqref{mlyx}. Then the same proof in conjunction with \Cref{evThm} shows that $A$ order-isometrically embeds as a non-closed Jordan subalgebra of a unital JB-algebra.
\end{remarks}

%Acknowlegdments: Lorentz centre
%\input{outputJB.bbl}
\bibliography{referenties.bib}

\end{document}